\documentclass[14pt,a4paper,reqno]{amsart}
\usepackage{vmargin,color}
\usepackage[latin1]{inputenc}
\usepackage{amsmath,amsfonts,amsthm,epsfig,graphicx}
\usepackage{pstricks,pst-plot,multido}
\usepackage{caption}
\usepackage{esint} 
\usepackage{pgfplots}
\usepackage{graphicx,tikz,subfigure,color,calc}
\usetikzlibrary{decorations.pathreplacing}

\pgfplotsset{compat=1.10}
\usepgfplotslibrary{fillbetween}
\usepgfplotslibrary{polar}
\usetikzlibrary{patterns}
\usepackage{relsize}

\usepackage{bigints}

\usepackage{mathrsfs}

\definecolor{darkgreen}{rgb}{0.0,0.5,0.0}
\definecolor{darkblue}{rgb}{0.0,0.0,0.3}
\definecolor{nicosred}{rgb}{0.65,0.1,0.1}
\definecolor{light-gray}{gray}{0.6}
\definecolor{really-light-gray}{gray}{0.8}

\usepackage{marginnote}

\def\H{\mathcal{H}}

\def\R{\mathbb R}

\def\Dx{\nabla_{\!  x {\scriptscriptstyle\parallel}}}

\def\pae{\partial^{{\rm e}}}

\newcommand{\diver}{\mathrm{div}} 
\newcommand{\mres}{\mathbin{\vrule height 1.6ex depth 0pt width  
0.13ex\vrule height 0.13ex depth 0pt width 1.3ex}}    

\newcommand{\Upp}{^{\vee}}
\newcommand{\Low}{^{\wedge}}

\DeclareMathOperator*{\aplim}{ap\,lim}

\def\big{\bigskip}

\newtheorem{theorem}{Theorem}[section]
\newtheorem{remark}[theorem]{Remark}
\newtheorem{definition}[theorem]{Definition}

\newtheorem{proposition}[theorem]{Proposition}
\newtheorem{lemma}[theorem]{Lemma}
\newtheorem{corollary}[theorem]{Corollary}
\newtheorem{example}[theorem]{Example}

\numberwithin{equation}{section}
\numberwithin{figure}{section}

\begin{document}

\title{Rigidity for the P\'olya-Szeg\"o inequality \\
under circular rearrangement}

\author{F. Cagnetti}
\address{%
Dipartimento di Scienze Matematiche, Fisiche e Informatiche \\
Plesso di Matematica \\
Universit\`a di Parma \\
Viale Parco Area delle Scienze 53/A \\
I-43124 PARMA (Italy)}%
\email{filippo.cagnetti@unipr.it}

\author{G. Domazakis}
\address{Department of Mathematical Sciences, University of Durham, DH1 3LE, DURHAM (United Kingdom)}%
\email{georgios.domazakis@durham.ac.uk}

\author{M. Perugini}
\address{Dipartimento di Ingegneria e Scienze dell'Informazione e Matematica\\ Universit\` a degli studi dell'Aquila \\ Via Vetoio,  L'AQUILA (Italy)}%
\email{matteo.perugini@univaq.it}

\author{F. Seuffert}
\address{Department of Mathematics,
University of Pennsylvania,
David Rittenhouse Lab,
209 South 33rd Street,
Philadelphia,
PA 19104-6395}
\email{francis.seuffert@gmail.com}

\begin{abstract}
{\rm A P\'olya-Szeg\"o inequality for the circular rearrangement is proven, 
under general assumptions. In addition, sufficient conditions are given, under which all the extremals of the inequality are symmetric.}
\end{abstract}

\maketitle

\section{Introduction}
Rearrangement inequalities are a fundamental tool in the Calculus of Variations, since they allow to prove symmetry of solutions of variational problems and PDEs. For instance, the radial decreasing rearrangement and the Steiner rearrangement are by now standard tools that are commonly employed in various contexts.
However, these techniques cannot be used in situations in which the functions involved can change sign, or when they 
are defined in domains that are not simply connected. 
To tackle these problems, in this paper we consider the circular rearrangement of functions, in which level sets are 
rearranged using circular arcs. Our main contributions are the following:

\medskip

\begin{enumerate}

\item We establish a P\'olya-Szeg\"o inequality for the circular rearrangement 
in any dimension and under general assumptions, both on the integrand
and on the class of functions under consideration, see Theorem~\ref{thm_Polya-Szego inequality local} and Theorem~\ref{thm_Polya-Szego inequality} (see also Theorem~\ref{the thing} for a more general statement).

\begin{enumerate}

\vspace{.2cm}

\item[(1a)] To the best of our knowledge, the P\'olya-Szeg\"o inequalities currently available in the literature for the circular symmetrization only deal with nonnegative functions \cite{kawohl_book_85, polya50, SmetsWillem03, VanSchaftingen}. Instead, we consider also the case of functions that can change sign.
\vspace{.2cm}

\item[(1b)] Our results can be applied in any dimension, in situations in which one expects minimizers to have cylindrical symmetry (or, more precisely, to be 
\textit{$2$-sectionally foliated Schwarz symmetric}, see \cite[Definition~1.2]{2020DamaPace}). In \cite{kawohl_book_85, SmetsWillem03, VanSchaftingen}, the authors consider the spherical symmetrization, in which level sets are rearranged using spherical caps. 
This coincides with the circular symmetrization only in dimension $2$. Even in the $2$-dimensional case, we extend the known results, by considering a general class of integrands, and functions that can change sign.

\vspace{.2cm}

\item[(1c)] Our proof of the P\'olya-Szeg\"o inequality is \textit{not} obtained via approximation. 
We give a direct proof, and this allows us to study in detail the equality cases.
 
\end{enumerate}
 
\vspace{.2cm}

\item We give sufficient conditions under which rigidity holds, that is,
under which all extremals are symmetric, 
see Theorem~\ref{thm suff conditions}.
Such conditions are given in terms of the notion of essential connectedness (see Definition~\ref{definition ess con} and \cite{CagnettiColomboDePhilippisMaggiSteiner,ccdpmGAUSS}). 

%

\vspace{.2cm}

\item The results we prove in this paper will be instrumental in the study of a general P\'olya-Szeg\"o inequality under spherical symmetrization.
In this case, when trying to give a direct proof of the inequality one encounters a major technical difficulty, 
due to the existence of functions that are the spherical counterpart of what  Almgren and Lieb call \textit{Coarea irregular functions} \cite[Definition~1.2.6]{AL}. This implies that an identity of the type \eqref{measure 0 for the bad set} cannot be obtained for the spherical rearrangement,
 see Remark~\ref{if higher dim slices} and Remark~\ref{spherical difficult}. This will be the object of further study \cite{CagnettiDomazakisPerugini}.

\end{enumerate}

\bigskip

We now describe how the rest of the Introduction is arranged.
We start by 
recalling the classical P\'olya--Szeg\"o inequality (Section~\ref{section polya}), 
and then we explain how the circular symmetrization operates on sets (Section~\ref{sect circular symm sets}) and on functions (Section~\ref{sect circular symm functions}). In Section~\ref{sec polya-szego} we state the P\'olya--Szeg\"o inequality under general assumptions (see Theorem~\ref{thm_Polya-Szego inequality local} and Theorem~\ref{thm_Polya-Szego inequality}), which is our first main result.
In Section~\ref{sect comments polya} we motivate the assumptions of Theorem~\ref{thm_Polya-Szego inequality local}, giving some counterexamples. Section~\ref{intro rigidity} contains our second main result, which gives sufficient conditions for rigidity (Theorem~\ref{thm suff conditions}),
while in Section~\ref{sect comments rigidity} we exhibit examples in which the assumptions of Theorem~\ref{thm suff conditions}
are not satisfied and rigidity fails.

\subsection{P\'olya-Szeg\"o inequality for the Schwarz rearrangement} \label{section polya}
In its most well-known version, the P\'olya-Szeg\"o inequality states that if $u : \mathbb{R}^n \to \mathbb{R}$
is any nonnegative smooth function with compact support and $1 \leq p < \infty$, then
\begin{equation} \label{PS radial}
\| \nabla u^* \|_{L^p ( \mathbb{R}^n)} \leq \| \nabla u \|_{L^p ( \mathbb{R}^n)},
\end{equation}
where $u^*: \mathbb{R}^n \to \mathbb{R}$ denotes the radial decreasing (also known as Schwarz) rearrangement of $u$, see \cite{polyaszego51}.
More in general, \eqref{PS radial} is satisfied if $u$ is nonnegative with compact support and belongs to the Sobolev space $W^{1, p} (\mathbb{R}^n)$.
We say that rigidity holds for the previous inequality if the following implication is true:
\[
\text{ $u$ satisfies equality in \eqref{PS radial} }
\quad
\Longrightarrow 
\quad \exists \, c \in \R^n : \quad 
u (\cdot) = u^* (\cdot + c) \, \, \mathcal{L}^n\text{-a.e. in } \R^n.
\]
where $\mathcal{L}^n$ denotes the $n$-dimensional Lebesgue measure.
Let us now set $M = \text{ess\,sup\,} u^*$, and  $C^*:= \{  \nabla u^* = 0 \} \cap (u^*)^{-1}((0, M))$. 
In the seminal paper \cite{brothersziemer}, Brothers and Ziemer showed that if 
%
%
%
%
%
%
%
%
%
%
%
%
%
\begin{align}\label{no flat parts intro}
\mathcal{L}^n ( C^* ) = 0
\end{align}
then rigidity holds, see \cite[Theorem~1.1]{brothersziemer}. 
Very recently, the first author has proved that condition \eqref{no flat parts intro} is also necessary for rigidity 
\cite{CagnettirigidityPolyaSzego}. 
Inequality \eqref{PS radial} and its rigidity have also been studied 
for functions of bounded variation \cite{cianchifusco1}, and for the Steiner rearrangement of codimension $1$ \cite{cianchifusco2} 
and of any codimension \cite{capriani2}.
When rigidity holds, the stability of the inequality can also be investigated, see \cite{BarchiesiCaprianiFuscoPisante} for the cases of Steiner and Schwarz rearrangements.

\medskip

Inequality \eqref{PS radial} has been used in a number of applications, 
as for instance the characterization of the extremals of Sobolev inequality \cite{Talenti76}, 
and the proof of a sharp quantitative version of the same inequality \cite{ CFMP09}. 

\medskip

There are, however, several variational problems in which these techniques cannot be applied.
This happens, for instance, if the domain of the functions under consideration
is not simply connected (e.g. annular domains in the plane), or in which 
the functions can change sign and have nodal domains. 

\medskip

As an example of a situation in which Steiner and Schwarz symmetrization, 
or the moving planes method \cite{Serrin71} cannot be applied, we mention \cite{H-S_20_2}, where the authors 
use the spherical symmetrization to show the symmetry of the extremals of Morrey's inequality (see also \cite{H-S_20_1, H-S_21}, 
where several other properties of the extremals are proved).
We observe that in \cite{H-S_20_2}, using additional symmetry properties of the extremals, the authors are able 
to apply the rearrangement to non negative functions. One of the goals of this paper is to show that the circular rearrangement can also be used with functions that change sign, 
see Definition~\ref{def w1p0tau} and Remark~\ref{different vmu} for more details. 

\medskip

%

\subsection{Circular symmetrization} \label{sect circular symm sets}
To the best of our knowledge, the circular symmetrization of sets 
and the associated rearrangement for functions 
were firstly introduced by P\'olya in \cite{polya50}. 
Let $N \in \mathbb{N}$ with $N \geq 2$, and let us label the points of $\R^N$ as $(x, z)$, with $x \in \R^2$ and $z \in \mathbb{R}^{N-2}$.
%
%
%
Moreover, for every $r > 0$ we set $D (r) = \{ x \in \R^2 : |x| < r \}$ and 
$\partial D (r) = \{ x \in \R^2 : |x| = r \}$. 

Let now $E \subset \R^N$ be a Lebesgue measurable set.
For every $(r, z) \in (0,\infty) \times \mathbb{R}^{N-2}$,
we define the slice $E_{(r, z)}$ of $E$ at $(r, z)$ as the subset of $\mathbb{R}^2$ given by
\[
E_{(r, z)} 
:= \{ x \in \partial D (r)  \text{ such that } (x,z) \in E \}.
\] 
We introduce the circular projection $\Pi_{N-1} (E)$ of $E$ and its annular part $\Pi_{N-1}^{\textnormal{a}}(E)$ as
\begin{equation} \label{def_PiN}
\Pi_{N-1} (E):= \{ (r, z) \in (0, \infty) \times \mathbb{R}^{N-2}: 
\mathcal{H}^1 (E_{(r, z)}) \neq 0  \},
\end{equation}
and
\begin{equation} \label{def_Pi_a}
\Pi_{N-1}^{\textnormal{a}}(E) := \left\{(r,z)\in \Pi_{N-1} (E):\,  \mathcal{H}^1 (E_{(r, z)}) =  2 \pi r \right\},
\end{equation}
respectively, where $\mathcal{H}^1$ stands for the $1$-dimensional Hausdorff measure in $\R^N$.

The circular symmetrization (with respect to the half-hyperplane $\{ x_2 = 0 \} \cap \{ x_1 > 0 \}$)
of a Lebesgue measurable set $E$ is the Lebesgue measurable set $E^s$ such that, for every $(r, z) \in (0,\infty) \times \mathbb{R}^{N-2}$, the slice $E^s_{(r, z)}$ is a connected arc centred at the point $(r, 0)$ having the same $\mathcal{H}^1$-measure
as $E_{(r, z)}$. More precisely, setting $\mathbb{R}^2_0 = \mathbb{R}^2 \setminus \{ (0, 0) \}$ and $\hat{x} = x/|x|$, 
we define $E^s$ as (see Figure~\ref{first figure})
\begin{equation*} 
E^s := \{ (x, z) \in \mathbb{R}^2_0 \times \mathbb{R}^{N-2} : 2 | x | \arccos (\hat{x} \cdot e_1) < 
\mathcal{H}^1 (E_{(r, z)}) \},
 \end{equation*}
 where $e_1 = (1,0)$.
Let us observe that, if $E$ is open, then the set $E^s$ defined above is open, see Remark~\ref{Es no points -1} and Proposition~\ref{prop_Omega^s is open}.
Moreover, the circular symmetrization preserves the $N$-dimensional Lebesgue measure and does not increase the perimeter.
\begin{center}
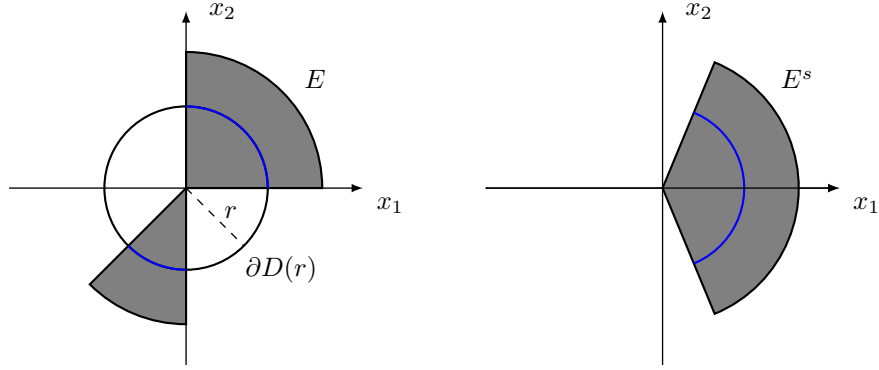
\begin{figure}[h]

 \begin{tikzpicture}[>=latex, scale=1.8]
 
    \draw[fill=gray, style=thick]
(0,0) -- +(90:1) arc(90:0:1) -- cycle;
    \draw[fill=gray, style=thick]
(0,0) -- +(225:1) arc(225:270:1) -- cycle;
    \draw(.8,.8)node[right]{$E$};
  \draw[->, line width = .5pt] (-1.3,0)--(1.3,0); 
  \draw(1.5,0)node[below]{$x_1$};
 \draw[->, line width = .5pt] (0,-1.3)--(0,1.3); 
  \draw(.1,1.3)node[right]{$x_2$};

 \draw[->, line width = .5pt] (3.5-1.3,0)--(3.5+1.3,0); 
  \draw(3.5+1.5,0)node[below]{$x_1$};
 \draw[->, line width = .5pt] (3.5,-1.3)--(3.5,1.3); 
  \draw(.1+3.5,1.3)node[right]{$x_2$};

\draw[black,thick] (0,0) circle (0.6);
\draw[black, dashed, line width = .5pt] (0,0)--(0.6/1.4,-0.6/1.4); 
\draw(0.3/1.4,-0.25/1.4)node[right]{$r$};
\draw [blue,thick,domain=0:90] plot ({0.6*cos(\x)}, {0.6*sin(\x)});
\draw [blue,thick,domain=225:270] plot ({0.6*cos(\x)}, {0.6*sin(\x)});
 \draw(.7,-.6)node{$\partial D (r)$};

\draw [blue,thick,domain=-67.5:67.5] plot ({3.5+ 0.6*cos(\x)}, {0.6*sin(\x)});

    \draw[fill=gray, style=thick]
(3.5,0) -- +(-67.5:1) arc(-67.5:67.5:1) -- cycle;
    \draw(4.3,.8)node[right]{$E^s$};
\draw [blue,thick,domain=-67.5:67.5] plot ({3.5+ 0.6*cos(\x)}, {0.6*sin(\x)});
 \draw[->, line width = .5pt] (3.5-1.3,0)--(3.5+1.3,0);

%
%
  
\end{tikzpicture}
 \caption{An example of circular symmetrization when $N = 2$. 
For a given value of $r > 0$, the slices $E_r$ and $E^s_r$ are highlighted in blue.}
       \label{first figure}
\end{figure}
\end{center}
\noindent
In order to give a precise statement of this fact, let us denote by $P(E)$  the distributional perimeter of $E$
and, for $A \subset \R^N$ Borel, let $P(E, A)$ stand for the distributional perimeter of $E$ in $A$
(see Section~\ref{section_preliminaries}). Moreover, let $\Phi_N : (0, \infty) \times \mathbb{R}^{N-2} 
\times \mathbb{S}^1 \to \mathbb{R}^2_0 \times \mathbb{R}^{N-2}$
be the diffeomorphism given by
\begin{equation} \label{def_PhiN}
\Phi_N (r, z, \omega) := (r \omega, z ) 
\quad \text{ for every } (r,z, \omega) \in (0, \infty) \times \mathbb{R}^{N-2} \times \mathbb{S}^1.
\end{equation}
We then have the following result, see \cite[Theorem 1.1]{CagnettiPeruginiStoger}, and \cite[Theorem 1.3]{PeruginiCircolare}.
\begin{theorem} \label{per ineq circ}
Let $E \subset \R^N$ be a set of finite perimeter in $\R^N$ with $\mathcal{L}^N (E) < \infty$. Then, 
$E^s$ is a set of finite perimeter in $\R^N$ with $\mathcal{L}^N (E^s) = \mathcal{L}^N (E)$. Moreover, 
\begin{equation} \label{per ineq circ eq}
P (E^s; \Phi_N (B \times \mathbb{S}^1)) \leq P (E; \Phi_N (B \times \mathbb{S}^1)) \quad
\text{ for every Borel set } B \subset (0, \infty) \times \mathbb{R}^{N-2}.
\end{equation}
\end{theorem}
\noindent
Choosing $B = (0, \infty) \times \mathbb{R}^{N-2}$ in \eqref{per ineq circ eq},
we obtain that $P (E^s) \leq P(E)$, see Figure~\ref{third figure}.
However, more in general inequality \eqref{per ineq circ eq} holds locally, 
as shown in Figure~\ref{fourth figure}.
\begin{center}
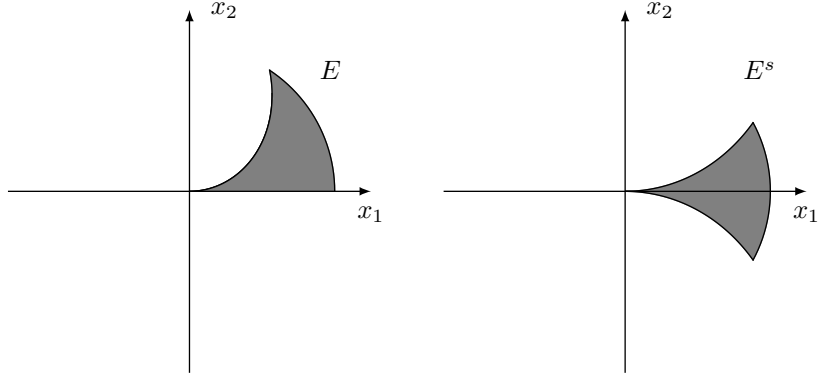
\begin{figure}[h]
  \begin{tikzpicture}[>=latex, scale=1.6]
  

%
%
  
  
  \draw [name path=E, color=gray, line width = .5pt, samples=200, domain=0:56.7]
plot (\x:{1.25*0.017*(\x)))});  
    
\draw [name path=F, color=gray, line width = .5pt, domain=0:56.7] plot ({1.2*cos(\x)}, {1.2*sin(\x)});  
  
 \draw [name path=G, color=gray, line width = .5pt] (.8,0)--(1.2,0);  
  
 \draw [name path=H, color=gray, line width = .5pt] (0,0)--(.9,0);  
  
    \draw [name path=I, line width = .5pt, samples=200, domain=0:50]
plot (\x:{1.25*0.017*(\x)))});  
  
  \tikzfillbetween[
    of=F and G
  ] {color=gray};
  
    \tikzfillbetween[
    of=I and H
  ] {color=gray};
  
 
  \draw [line width = .5pt, samples=200, domain=0:56.7]
plot (\x:{1.25*0.017*(\x)))});  
    
\draw [line width = .5pt, domain=0:56.7] plot ({1.2*cos(\x)}, {1.2*sin(\x)});


  \draw[->, line width = .5pt] (-1.5,0)--(1.5,0); 
  \draw(1.5,-.05)node[below]{$x_1$};
 \draw[->, line width = .5pt] (0,-1.5)--(0,1.5); 
  \draw(.1,1.5)node[right]{$x_2$};

 \draw(1,1)node[right]{$E$};


%


\begin{scope}[shift={(3.6,0)}]


%
%


    \draw [name path=A, color=gray, line width = .5pt, samples=200, domain=0:56.7/2]
plot (\x:{2*1.25*0.017*(\x)))});    

  \begin{scope}[yscale=-1,xscale=1]

 \draw [name path=B, color=gray, line width = .5pt, samples=200, domain=0:56.7/2]
plot (\x:{2*1.25*0.017*(\x)))});  

\tikzfillbetween[
    of=A and B
  ] {color=gray};
  
 \end{scope}
  
  \draw [name path=C, color=gray,  line width = .5pt, domain=-56.7/2:56.7/2] plot ({1.2*cos(\x)}, {1.2*sin(\x)});
 \draw [name path=D, color=gray, line width = .5pt] (1.2*0.88, 1.2*0.4749)--(1.2*0.88, -1.2*0.4749);  
  \tikzfillbetween[
    of=C and D
  ] {color=gray};

  
   \draw [line width = .5pt, samples=200, domain=0:56.7/2]
plot (\x:{2*1.25*0.017*(\x)))});    
  
  \begin{scope}[yscale=-1,xscale=1]

 \draw [line width = .5pt, samples=200, domain=0:56.7/2]
plot (\x:{2*1.25*0.017*(\x)))});  

  \end{scope}

\draw [line width = .5pt, domain=-56.7/2:56.7/2] plot ({1.2*cos(\x)}, {1.2*sin(\x)});


\draw[->, line width = .5pt] (-1.5,0)--(1.5,0); 
  \draw(1.5,-.05)node[below]{$x_1$};
 \draw[->, line width = .5pt] (0,-1.5)--(0,1.5); 
  \draw(.1,1.5)node[right]{$x_2$};


%
%
%
 
  \end{scope}

 \draw(4.5,1)node[right]{$E^s$};

 \end{tikzpicture}
 \caption{Perimeter inequality for $N=2$. The perimeter of the set $E^s$ (right), is less than or equal to the perimeter of the set $E$ (left).}
       \label{third figure}
\end{figure}
\end{center}
Let us mention that some variants of the circular symmetrization can be used when one wants to preserve 
the barycenter or some symmetry properties of a set, see for instance \cite{bonnesen1929, Gamb}.
\begin{center}
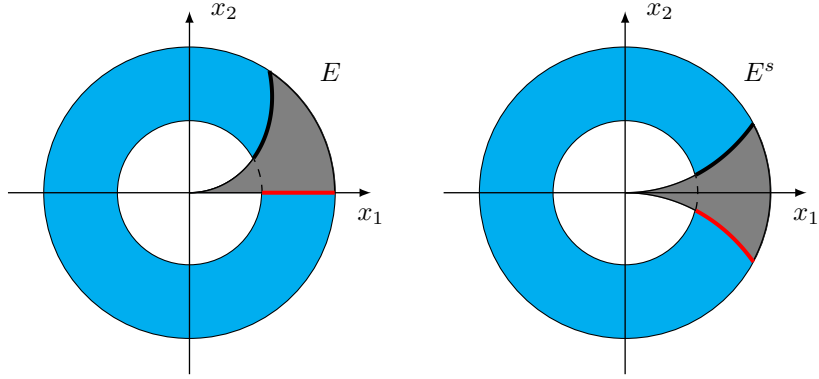
\begin{figure}[h]

  \begin{tikzpicture}[>=latex, scale=1.6]
  

  \draw[name path=N, black,thick] (0,0) circle (0.6);
  
  \draw[name path=O, black,thick] (0,0) circle (1.2);
 
 \tikzfillbetween[
    of=N and O
  ] {color=cyan};
  
  
  \draw [name path=E, color=gray, line width = .5pt, samples=200, domain=0:56.7]
plot (\x:{1.25*0.017*(\x)))});  
    
\draw [name path=F, color=gray, line width = .5pt, domain=0:56.7] plot ({1.2*cos(\x)}, {1.2*sin(\x)});  
  
 \draw [name path=G, color=gray, line width = .5pt] (.8,0)--(1.2,0);  
  
 \draw [name path=H, color=gray, line width = .5pt] (0,0)--(.9,0);  
  
    \draw [name path=I, line width = .5pt, samples=200, domain=0:50]
plot (\x:{1.25*0.017*(\x)))});  
  
  \tikzfillbetween[
    of=F and G
  ] {color=gray};
  
    \tikzfillbetween[
    of=I and H
  ] {color=gray};
  
 
  \draw [line width = .5pt, samples=200, domain=0:56.7]
plot (\x:{1.25*0.017*(\x)))});  
    
\draw [line width = .5pt, domain=0:56.7] plot ({1.2*cos(\x)}, {1.2*sin(\x)});


  \draw[->, line width = .5pt] (-1.5,0)--(1.5,0); 
  \draw(1.5,-.05)node[below]{$x_1$};
 \draw[->, line width = .5pt] (0,-1.5)--(0,1.5); 
  \draw(.1,1.5)node[right]{$x_2$};

 \draw(1,1)node[right]{$E$};

\draw [line width = .5pt, black, dashed, domain=0:27.5] plot ({0.6*cos(\x)}, {0.6*sin(\x)});

  \draw [line width = 1.5pt, samples=200, domain=56.7/2:56.7]
plot (\x:{1.25*0.017*(\x)))});  

 \draw [color=red, line width = 1.5pt] (.6,0)--(1.2,0);


\begin{scope}[shift={(3.6,0)}]


  \draw[name path=L, black,thick] (0,0) circle (0.6);
  
  \draw[name path=M, black,thick] (0,0) circle (1.2);
 
 \tikzfillbetween[
    of=L and M
  ] {color=cyan};


    \draw [name path=A, color=gray, line width = .5pt, samples=200, domain=0:56.7/2]
plot (\x:{2*1.25*0.017*(\x)))});    

  \begin{scope}[yscale=-1,xscale=1]

 \draw [name path=B, color=gray, line width = .5pt, samples=200, domain=0:56.7/2]
plot (\x:{2*1.25*0.017*(\x)))});  

\tikzfillbetween[
    of=A and B
  ] {color=gray};
  
 \end{scope}
  
  \draw [name path=C, color=gray,  line width = .5pt, domain=-56.7/2:56.7/2] plot ({1.2*cos(\x)}, {1.2*sin(\x)});
 \draw [name path=D, color=gray, line width = .5pt] (1.2*0.88, 1.2*0.4749)--(1.2*0.88, -1.2*0.4749);  
  \tikzfillbetween[
    of=C and D
  ] {color=gray};

  
   \draw [line width = .5pt, samples=200, domain=0:56.7/2]
plot (\x:{2*1.25*0.017*(\x)))});    
  
  \begin{scope}[yscale=-1,xscale=1]

 \draw [line width = .5pt, samples=200, domain=0:56.7/2]
plot (\x:{2*1.25*0.017*(\x)))});  

  \end{scope}

\draw [line width = .5pt, domain=-56.7/2:56.7/2] plot ({1.2*cos(\x)}, {1.2*sin(\x)});


\draw[->, line width = .5pt] (-1.5,0)--(1.5,0); 
  \draw(1.5,-.05)node[below]{$x_1$};
 \draw[->, line width = .5pt] (0,-1.5)--(0,1.5); 
  \draw(.1,1.5)node[right]{$x_2$};

\draw [line width = .5pt, black, dashed, domain=-13.68:13.68] plot ({0.6*cos(\x)}, {0.6*sin(\x)});

   \draw [line width = 1.5pt, samples=200, domain=56.7/4:56.7/2]
plot (\x:{2*1.25*0.017*(\x)))});  
 
   \begin{scope}[yscale=-1,xscale=1]

  \draw [line width = 1.5pt, color=red, samples=200, domain=56.7/4:56.7/2]
plot (\x:{2*1.25*0.017*(\x)))});  

  \end{scope}
 
  \end{scope}

 \draw(4.5,1)node[right]{$E^s$};

 \end{tikzpicture}
 \caption{The same sets considered in Figure~\ref{third figure}. 
If one chooses $B$ as an open interval, the corresponding set $\Phi_2 (B \times \mathbb{S}^1)$
is the open annulus highlighted in blue.
Then, inequality \eqref{per ineq circ eq} states that the \textbf{sum} of the lenghts of the black and red arcs
in the right, is less than or equal to the \textbf{sum} of the lenghts of the two corresponding arcs in the figure 
in the left. Note that this \textbf{does not mean} that the length of each arc decreases after symmetrization.
In fact, the length of the red arc in the right is larger than the length of the corresponding red arc in the left.}
       \label{fourth figure}
\end{figure}
\end{center}

\subsection{State of the art} \label{sect circular symm functions}
Let $n \in \mathbb{N}$, $n \geq 2$, and let us denote points of $\R^n$ as $(x, y)$, 
with $x \in \R^2$ and $y \in \R^{n-2}$.
Let $\Omega \subset \R^n$ be open, 
and let $u: \Omega \to \R$ be a Lebesgue measurable function.
The circular rearrangement $u^s$ of $u$ is the Lebesgue measurable function 
$u^s: \Omega^s \to \R$ such that
\begin{equation} \label{def us}
\{ u^s > t \} = \{ u > t \}^s, \quad \text{ for every } t \in \mathbb{R},
\end{equation}
where $\Omega^s$ and $\{ u > t \}^s$ are the circular rearrangements (in $\R^n$)
of the sets $\Omega$ and $\{ u > t \}$, respectively, see Figure~\ref{second figure}.
\tikzset{every picture/.style={line width=0.75pt}} 
\begin{center}
\begin{figure}[!htb]
\begin{tikzpicture}[x=0.75pt,y=0.75pt,yscale=-1,xscale=1]

\draw    (575.91,359.9) -- (576.5,488) ;
\draw  [dash pattern={on 4.5pt off 4.5pt}]  (586,448.5) -- (438,515.5) ;
\draw  [dash pattern={on 4.5pt off 4.5pt}]  (557.5,336) -- (557.5,463) ;
\draw  [dash pattern={on 4.5pt off 4.5pt}]  (499.6,487.22) -- (500.2,358.96) ;
\draw  [dash pattern={on 4.5pt off 4.5pt}]  (438.7,451.45) -- (553.5,519) ;
\draw    (500.2,358.96) -- (500.2,311.96) ;
\draw [shift={(500.2,309.96)}, rotate = 90] [color={rgb, 255:red, 0; green, 0; blue, 0 }  ][line width=0.75]    (10.93,-3.29) .. controls (6.95,-1.4) and (3.31,-0.3) .. (0,0) .. controls (3.31,0.3) and (6.95,1.4) .. (10.93,3.29)   ;
\draw    (438,515.5) -- (399,534) ;
\draw    (437.7,451.45) -- (402.5,431) ;
\draw    (586,448.5) -- (623.19,430.86) ;
\draw [shift={(625,430)}, rotate = 154.62] [color={rgb, 255:red, 0; green, 0; blue, 0 }  ][line width=0.75]    (10.93,-3.29) .. controls (6.95,-1.4) and (3.31,-0.3) .. (0,0) .. controls (3.31,0.3) and (6.95,1.4) .. (10.93,3.29)   ;
\draw    (553.5,519) -- (593.75,541.04) ;
\draw [shift={(595.5,542)}, rotate = 208.71] [color={rgb, 255:red, 0; green, 0; blue, 0 }  ][line width=0.75]    (10.93,-3.29) .. controls (6.95,-1.4) and (3.31,-0.3) .. (0,0) .. controls (3.31,0.3) and (6.95,1.4) .. (10.93,3.29)   ;
\draw    (437.41,387.4) -- (438,515.5) ;
\draw    (557.5,336) -- (437.41,387.4) ;
\draw    (557.5,336) .. controls (566.5,339.5) and (573.5,346.5) .. (575.91,359.9) ;
\draw    (576.5,488) .. controls (578.09,501.6) and (550.5,561) .. (438,515.5) ;
\draw  [dash pattern={on 4.5pt off 4.5pt}]  (557.5,463) .. controls (566.5,466.5) and (573.5,473.5) .. (575.91,486.9) ;
\draw    (575.91,359.9) .. controls (577.5,373.5) and (545.91,426.9) .. (437.41,387.4) ;
\draw    (294.91,359.9) -- (295.5,488) ;
\draw  [dash pattern={on 4.5pt off 4.5pt}]  (305,448.5) -- (219.79,488.07) ;
\draw  [dash pattern={on 4.5pt off 4.5pt}]  (276.5,336) -- (276.5,463) ;
\draw    (218.6,487.22) -- (219.2,358.96) ;
\draw  [dash pattern={on 4.5pt off 4.5pt}]  (157.7,451.45) -- (219.79,488.07) ;
\draw    (219.2,358.96) -- (219.2,311.96) ;
\draw [shift={(219.2,309.96)}, rotate = 90] [color={rgb, 255:red, 0; green, 0; blue, 0 }  ][line width=0.75]    (10.93,-3.29) .. controls (6.95,-1.4) and (3.31,-0.3) .. (0,0) .. controls (3.31,0.3) and (6.95,1.4) .. (10.93,3.29)   ;
\draw    (157,515.5) -- (118,534) ;
\draw    (156.7,451.45) -- (121.5,431) ;
\draw    (305,448.5) -- (342.19,430.86) ;
\draw [shift={(344,430)}, rotate = 154.62] [color={rgb, 255:red, 0; green, 0; blue, 0 }  ][line width=0.75]    (10.93,-3.29) .. controls (6.95,-1.4) and (3.31,-0.3) .. (0,0) .. controls (3.31,0.3) and (6.95,1.4) .. (10.93,3.29)   ;
\draw    (272.5,519) -- (312.75,541.04) ;
\draw [shift={(314.5,542)}, rotate = 208.71] [color={rgb, 255:red, 0; green, 0; blue, 0 }  ][line width=0.75]    (10.93,-3.29) .. controls (6.95,-1.4) and (3.31,-0.3) .. (0,0) .. controls (3.31,0.3) and (6.95,1.4) .. (10.93,3.29)   ;
\draw    (156.41,387.4) -- (157,515.5) ;
\draw    (219.79,488.07) -- (157,515.5) ;
\draw    (276.5,336) .. controls (285.5,339.5) and (292.5,346.5) .. (294.91,359.9) ;
\draw  [dash pattern={on 4.5pt off 4.5pt}]  (276.5,463) .. controls (285.5,466.5) and (292.5,473.5) .. (294.91,486.9) ;
\draw  [dash pattern={on 4.5pt off 4.5pt}]  (174.5,463.5) .. controls (154.5,465.5) and (142.5,497.5) .. (157,515.5) ;
\draw    (149.41,366.4) -- (150,494.5) ;
\draw    (173.91,335.4) .. controls (153.91,337.4) and (141.91,369.4) .. (156.41,387.4) ;
\draw    (173.91,335.4) -- (219.2,359.96) ;
\draw    (219.2,359.96) -- (273.91,390.9) ;
\draw    (273.91,390.9) -- (274.5,519) ;
\draw    (150,494.5) .. controls (150.5,505.5) and (150.5,502.5) .. (157,515.5) ;
\draw  [dash pattern={on 4.5pt off 4.5pt}]  (174.5,463.5) .. controls (154.5,465.5) and (142.5,497.5) .. (157,515.5) ;
\draw    (295.5,488) .. controls (294.5,496) and (293.5,511) .. (274.5,520) ;
\draw    (219.79,488.07) -- (274.5,519) ;
\draw    (219.2,359.96) -- (156.41,387.4) ;
\draw    (276.5,336) -- (219.2,359.96) ;
\draw    (294.91,359.9) .. controls (293.91,367.9) and (292.91,382.9) .. (273.91,391.9) ;
\draw  [dash pattern={on 4.5pt off 4.5pt}]  (174.5,336.5) -- (174.5,463.5) ;

\draw (579.5,544) node [anchor=north west][inner sep=0.75pt]   [align=left] {$x_1$};
\draw (627,433) node [anchor=north west][inner sep=0.75pt]   [align=left] {$x_2$};
\draw (507,304) node [anchor=north west][inner sep=0.75pt]   [align=left] {$t$};
\draw (495,495) node [anchor=north west][inner sep=0.75pt]   [align=left] {$O$};
\draw (298.5,544) node [anchor=north west][inner sep=0.75pt]   [align=left] {$x_1$};
\draw (346,433) node [anchor=north west][inner sep=0.75pt]   [align=left] {$x_2$};
\draw (226,304) node [anchor=north west][inner sep=0.75pt]   [align=left] {$t$};
\draw (220,342) node [anchor=north west][inner sep=0.75pt]   [align=left] {$1$};
\draw (220+270,342+5) node [anchor=north west][inner sep=0.75pt]   [align=left] {$1$};
\draw (270,304) node [anchor=north west][inner sep=0.75pt]   [align=left] {$u (x_1, x_2)$};
\draw (214,495) node [anchor=north west][inner sep=0.75pt]   [align=left] {$O$};
\draw (261,518) node [anchor=north west][inner sep=0.75pt]   [align=left] {$1$};
\draw (551,304) node [anchor=north west][inner sep=0.75pt]   [align=left] {$u^s (x_1, x_2)$};
\draw (80,300) node [anchor=north west][inner sep=0.75pt]   [align=left] {\mbox{ }}; 
\end{tikzpicture}
 \caption{In this case, $n = 2$ and $\Omega = D (1) = \{ |x| < 1 \}$.
 In the left, the graph of the measurable function $u: D (1) \to \R$ defined as $u (x_1, x_2) = 1$ if $x_1 x_2 > 0$, 
 and $u (x_1, x_2) = 0$ elsewhere in $D (1)$. In the right, the graph of the function $u^s$. We have 
 $u^s (x_1, x_2) = 1$ if $x_1 > 0$, and  $u^s (x_1, x_2) = 0$ elsewhere in $D (1)$.}
       \label{second figure}
\end{figure}
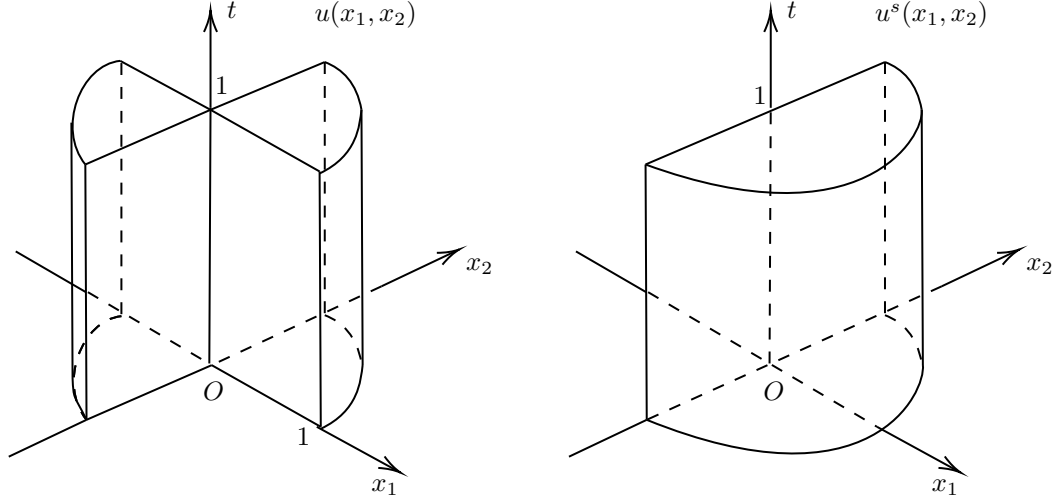
\end{center}
\noindent
P\'olya firstly observed that, if $n = 2$, $\Omega$ is bounded, and $u$ is smooth with $u > 0$ in $\Omega$
and $u = 0$ on $\partial \Omega$, then
\begin{equation} \label{pol sze classical}
\int_{\Omega^s} |\nabla u^s (x)|^2  \, dx  \leq \int_{\Omega} |\nabla u (x)|^2 \, dx.
\end{equation} 
The proof provided by P\'olya in \cite{polya50} is quite elegant and  goes along the following lines.
For every $\varepsilon > 0$, one can consider the subgraph of $\varepsilon u$:
\[
\Sigma^{\varepsilon u} := \{ (x, t) \in  \Omega  \times \R :   t < \varepsilon u (x) \}.
\]
Then, if $(\Sigma^{\varepsilon u})^s$ denotes the circular symmetrization of $\Sigma^{\varepsilon u}$ in $\R^3$, 
we have 
\[
(\Sigma^{\varepsilon u})^s = \Sigma^{\varepsilon u^s}.
\]
Since the circular symmetrization does not increase the perimeter, and since 
$\mathcal{L}^2 (\Omega^s) = \mathcal{L}^2 (\Omega)$, one has 
\begin{equation} \label{intermediate with epsilon}
\int_{\Omega^s} \left[ \sqrt{1 + \varepsilon^2 |\nabla u^s (x)|^2} - 1 \right] \, dx  \leq \int_{\Omega} \left[ \sqrt{1 + \varepsilon^2  |\nabla u (x)|^2} - 1 \right]  \, dx.
\end{equation}
Dividing the previous inequality by $\varepsilon^2$ and taking the limit as $\varepsilon \to 0^+$, we obtain \eqref{pol sze classical}.

In 1985, Kawohl showed that if $n =2$, $1 < p < \infty$, $\Omega \subset \R^2$ is bounded
and $u \in W^{1, p}_0 (\Omega)$ with $u \geq 0$, then $u^s \in W^{1, p}_0 (\Omega^s)$ and 
\[
\int_{\Omega^s} F (| x |, u^s ) |\nabla u^s |^p  \, d x 
\leq
\int_{\Omega} F (| x |, u ) |\nabla u |^p  \, d x,
\]
whenever $F : [0, \infty) \times \R \to [0, \infty)$ is continuous \cite[Corollary~2.35, ii)]{kawohl_book_85}.
Under the additional assumptions that $\partial \Omega$ is piecewise analytic and $u$ is analytic, he also proved that
\begin{equation} \label{strong assump kawohl}
\int_{\Omega^s} F (| x |, u^s ) \, G ( |\nabla u^s | )  \, d x 
\leq
\int_{\Omega} F (| x |, u ) \, G ( |\nabla u | )  \, d x,
\end{equation}
whenever $G : [0, \infty) \to \R$ is nondecreasing and convex.

A more recent result is due to Smets and Willem.
If $n = 2$, $1 < p < \infty$, $u \in W^{1, p}_0 (\Omega)$ and $u \geq 0$, they showed that 
\begin{equation} \label{willem}
\int_{\Omega^s} F (| x |) |\nabla u^s |^p  \, d x 
\leq
\int_{\Omega} F (| x |) |\nabla u |^p  \, d x,
\end{equation}
whenever $F$ is nonnegative, measurable and bounded \cite[Theorem~2.8]{SmetsWillem03}.
More precisely, the authors prove \eqref{willem} when $u^s$ is interpreted
as the foliated Schwarz symmetrization of $u$, in which level sets are rearranged 
via spherical symmetrization (see \cite[Definition~2.4]{SmetsWillem03}),
in any dimension.
We observe that, while in the case $n=2$ this coincides with the circular symmetrization, 
for $n \geq 3$ the two symmetrizations differ. 

\subsection{P\'olya--Szeg\"o inequality for the circular rearrangement under general assumptions} \label{sec polya-szego}
The first result we show in this paper is that a P\'olya-Szeg\"o inequality
holds for the circular symmetrization under very general assumptions.
First of all, let us introduce the class of functions under consideration. In the following, when $A,B \subset \mathbb{R}^n$ we write $A \subset \subset B$ if $A$ is compactly contained in $B$. 
\begin{definition} \label{def w1p0tau}
Let $n \in \mathbb{N}$ with $n \geq 2$, let $\Omega \subset \R^n$ be open, let $p \in [1, +\infty)$, let $u:\Omega \to \R$, and let  
$u_0$ be given by
\begin{equation} \label{def u0}
u_0 (x, y) := 
\begin{cases}
u (x, y) & \text{ if } (x, y) \in \Omega, \\
0 & \text{ if } (x, y) \in \Phi_n (\Pi_{n-1} (\Omega) \times \mathbb{S}^1) \setminus \Omega.
\end{cases}
\end{equation}
We say that $u \in W^{1, p}_{0, \tau} (\Omega)$ if the following conditions are satisfied:
\begin{itemize}

\item[(a$_p$)] $u_0 \in W^{1, p} (\Phi_n (A \times \mathbb{S}^1) )$ for every 
open set $A \subset (0, \infty) \times \mathbb{R}^{n-2}$ with $A \subset \subset \Pi_{n-1} (\Omega)$,

\vspace{.2cm}

\item[(b)] $u \geq 0$ \quad  $\mathcal{L}^n$-a.e. in $\Omega \setminus 
\Phi_n(\Pi_{n-1}^{\textnormal{a}}(\Omega) \times\mathbb{S}^1)$,

\end{itemize}
where $\Phi_n$, $\Pi_{n-1} (\Omega)$, and $\Pi_{n-1}^{\textnormal{a}}(\Omega)$ are defined by \eqref{def_PiN}, \eqref{def_Pi_a}, and \eqref{def_PhiN}, respectively.
\end{definition}
Let now $u:\Omega\to \mathbb{R}$ be Lebesgue measurable.
We define the distribution of $u$ as the function
$\mu  : \Pi_{n-1}(\Omega) \times \mathbb{R} \to [0, \infty)$  given by 
\begin{equation} \label{mu distr}
\mu (r, y, t) = \mathcal{H}^1 (\{ u_0 >t \}_{(r,y)}), \quad 
\text{ for every } 
(r, y, t) \in  \Pi_{n-1} (\Omega) \times \mathbb{R},
\end{equation}
where $u_0$ is given by \eqref{def u0}.
\begin{remark} \label{mu(t) right-cont}
We observe that, if \eqref{mu distr} holds for some Lebesgue measurable function $u:\Omega\to \mathbb{R}$, then
\[
0 \leq \mu (r, y, t) \leq 2 \pi r \quad \text{ for every }(r,y,t) \in \Pi_{n-1} (\Omega)\times\mathbb{R}, 
\]
and 
\[
t \longmapsto \mu (r, y, t) 
\quad \text{ is non-increasing and right-continuous, \quad for every }
(r, y) \in \Pi_{n-1} (\Omega).
\]  
Moreover, since $u (x, y)$ is finite for every $(x, y) \in \Omega$, we have 
\[
\lim_{t \to + \infty} \mu (r, y, t) = 0 
\quad \text{ and } \quad
\lim_{t \to - \infty} \mu (r, y, t) = 2 \pi r, \quad \text{ for every } (r, y) \in \Pi_{n-1} (\Omega).
\]
\end{remark}
\begin{remark}\label{rem_W^1,p_tau inside W^1,1_tau}
Let us note that, by H\"older inequality, $W^{1, p}_{0, \tau} (\Omega)\subset W^{1, 1}_{0, \tau} (\Omega)$ for every $p \in [1, + \infty)$, see also Remark~\ref{rem holder}.
\end{remark}
We stress the fact that if $u \in W^{1, p}_{0, \tau} (\Omega)$, this does not imply that $u = 0$ on all of $\partial \Omega$, as explained in the next remark. 
\begin{remark} \label{no inclus}
If $u \in W^{1, p}_0 (\Omega)$, then condition (a$_p$) is satisfied.
However, the opposite implication is false. 
Indeed, let $n = 2$, let $\Omega = D (1) \cap \{ x = (x_1, x_2) \in \R^2: x_1, x_2 > 0 \}$, let $u (x_1, x_2) = x_1 x_2$, 
and let $p \in [1, + \infty)$.
Then $u$ satisfies (a$_p$), but $u \notin W^{1, p}_0 (\Omega)$.
This example also shows that \( W^{1, p}_{0, \tau} (\Omega)  \not \subset W^{1, p}_{0} (\Omega) \).
\end{remark}


\begin{remark}
Note that $W^{1, p}_{0, \tau} (\Omega)$ is not a vector space.
Indeed, condition (b) is not closed under scalar multiplication.
For the same reason, unless $\Pi_{n-1}^{\textnormal{a}}(\Omega) = \Pi_{n-1} (\Omega)$, we have \(W^{1, p}_{0} (\Omega)  \not \subset  W^{1, p}_{0, \tau} (\Omega)\). 
However,
\[
u \geq 0 \quad \text{ and } \quad u \in W^{1, p}_{0} (\Omega)
\quad \Longrightarrow \quad
u \in W^{1, p}_{0, \tau} (\Omega).
\]
\end{remark}
 \begin{definition} \label{def distribution}
Let $n \in \mathbb{N}$ with $n \geq 2$, let $\Omega\subset \mathbb{R}^n$ be open and let  
$\mu :  \Pi_{n-1} (\Omega) \times \mathbb{R} \to [0, \infty)$ be a Lebesgue measurable function.
We say that $\mu$ is an admissible distribution if \eqref{mu distr} holds for (a representative of) some $u \in W^{1, 1}_{0, \tau} (\Omega)$. The set of all admissible distributions is denoted by $\mathcal{A}(\Pi_{n-1}(\Omega) \times \mathbb{R})$. 
We say that a function $u \in W^{1, 1}_{0, \tau} (\Omega)$ is $\mu$-distributed if there exists a representative of $u$ such that \eqref{mu distr} holds. 
\end{definition}

\noindent
It turns out that if $\mu \in \mathcal{A}(\Pi_{n-1}(\Omega) \times \mathbb{R})$, then 
$\mu \in BV (A \times (-d, + \infty))$ for every open set $A \subset \subset \Pi_{n-1} (\Omega)$ and for every $d > 0$, 
see Proposition~\ref{derivatives of mu}.
 
\medskip
 
The circular rearrangement of a function only depends on its distribution. For this reason, whenever $\mu \in \mathcal{A}(\Pi_{n-1}(\Omega) \times \mathbb{R})$ (or, more in general, whenever \eqref{mu distr} holds for some Lebesgue measurable function $u:\Omega\to \mathbb{R}$), we define 
$v_\mu:  \Phi_n (\Pi_{n-1} (\Omega) \times \mathbb{S}^1) \to \overline{\mathbb{R}}$ as
\begin{equation} \label{def vmu}
v_\mu (x, y) := \inf \left\{ t \in \mathbb{R} : \mu (|x|, y, t) \leq 2 |x|  \arccos (\hat{x} \cdot e_1)  \right\}.
\end{equation}
Thanks to Remark~\ref{mu(t) right-cont}, $v_\mu (x, y)$ is finite except possibly in the hyplerplane $\{ x_2 = 0\}$, which is $\mathcal{H}^n$-negligible. 
Note that, even if $u$ is defined only in $\Omega$, the function $v_\mu$ is defined in the (larger) set $ \Phi_n (\Pi_{n-1} (\Omega) \times \mathbb{S}^1)$,
see Remark~\ref{different vmu} for further comments.
 We also point out that the notion of distribution is well defined, as clarified by the next remarks.

\begin{remark}
Let $u_1, u_2 \in L^1_{\textnormal{loc}} (\Phi_{n}  (\Pi_{n-1} (\Omega)  \times \mathbb{S}^{1}))$, 
and let $\mu_{u_1}$ and $\mu_{u_2}$ be the distributions of
$u_1$ and $u_2$, respectively. 
If $u_1 = u_2$ $\mathcal{L}^n$-a.e. in $\Phi_{n}  (\Pi_{n-1} (\Omega)  \times \mathbb{S}^{1}))$, 
then one can see that \( \mu_{u_1} = \mu_{u_2} \) $\mathcal{L}^n$-a.e. in $\Pi_{n-1} (\Omega) \times \R$ (see Proposition \ref{prop several useful things}).
\end{remark}

\begin{remark}
By definition, $v_\mu$ is $\mu$-distributed.
Moreover, if $\mu_{1}, \mu_{2} \in \mathcal{A}(\Pi_{n-1}(\Omega) \times \mathbb{R})$
with \( \mu_{1} = \mu_{2} \) $\mathcal{L}^n$-a.e. in $\Pi_{n-1} (\Omega) \times \R$, 
one can see that $v_{\mu_1} = v_{\mu_2}$ $\mathcal{L}^n$-a.e. in $\Phi_n (\Pi_{n-1} (\Omega) \times \mathbb{S}^1)$  (see Proposition \ref{prop several useful things 2}).
Also, if $u \in W^{1, 1}_{0, \tau} (\Omega)$ is a $\mu$-distributed function, then $v_\mu \! \! \mid_{\Omega^s} \in W^{1, 1}_{0, \tau} (\Omega^s)$
(see Proposition~\ref{ us sobolev}), and $v_\mu = u^s$ $\mathcal{L}^n$-a.e. in $\Omega^s$, where $u^s$ is such that \eqref{def us} holds.
\end{remark}

 \medskip
 
\noindent
Let us now focus on the general assumptions on the integrand. 
If $u \in W^{1, p}_{0, \tau} (\Omega)$, for every $(x, y) \in \Omega$ such that the gradient $\nabla u (x, y)$ at $(x, y)$ is defined,  
we write \( \nabla u (x, y) = (\nabla_x u (x, y), \nabla_y u (x, y)) \), where $\nabla_x u (x, y) \in \R^{2}$ and $\nabla_y u (x, y) \in \R^{n-2}$, with
\[
\nabla_x u (x, y) = ( \partial_{x_1} u (x, y), \partial_{x_2} u (x, y) ), \quad
 \nabla_y u (x, y) = ( \partial_{y_1} u (x, y), \ldots, \partial_{y_{n-2}} u (x, y) ).
\]
Moreover, if $x = (x_1, x_2) \in \R^2_0$, we can further decompose $\nabla_x u (x, y)$ into radial and tangential components as
\[
\nabla_x u (x, y) = (\hat{x} \cdot \nabla_x u (x, y) ) \hat{x} + \Dx u (x, y) x_\parallel, 
\qquad \Dx u (x, y) = \nabla_x u (x, y) \cdot x_\parallel,
\]
where 
\begin{equation} \label{basis of R2}
\hat{x} = \left( \frac{x_1}{|x|}, \frac{x_2}{|x|} \right)
\qquad \text{ and } \qquad 
x_{\parallel} = \left( - \frac{x_2}{|x|}, \frac{x_1}{|x|} \right).
\end{equation}
\begin{definition}\label{def_admissable integrands}
Let $f: \R \times \R \times \R^{n-2} \to [0, \infty)$. We say that $f \in \mathscr{F}$ if the following assumptions are satisfied:
\begin{enumerate}
\vspace{.1cm} 
\item[(f1)]  \( f \) is convex;
\vspace{.2cm} 
\item[(f2)] \( f (\eta, -\tau, \zeta) = f (\eta, \tau, \zeta) \) for every $(\eta, \tau, \zeta) \in \R \times \R \times \R^{n-2}$.
\end{enumerate}
We say that $f \in \mathscr{F}'$ if, in addition,
\begin{enumerate}
\vspace{.2cm} 
\item[(f1')] \( f \) is \textbf{strictly} convex.
\vspace{.2cm} 
\end{enumerate}
\end{definition}

\begin{remark}
Note that the classical Dirichlet functional can be written as
\[
\int_{\Omega} |\nabla u |^2 \, dx \, dy = \int_{\Omega} f (\hat{x} \cdot \nabla_x u, \Dx u, \nabla_y u)  \, dx \, dy, 
\]
with $f \in \mathscr{F}'$ given by \( f (\eta, \tau, \zeta) = \eta^2 + \tau^2 + |\zeta|^2 \).
\end{remark}
\noindent
We can now state our first result, which shows that a general version of the P\'olya--Szeg\"o inequality holds locally. 
We recall that the set $\mathcal{A} (\Pi_{n-1} (\Omega) \times \R)$ of admissible distributions has been introduced in Definition~\ref{def distribution}. 
\begin{theorem}\label{thm_Polya-Szego inequality local}
Let $n \in \mathbb{N}$ with $n \geq 2$, let $\Omega \subset \R^n$ be open, and let $\mu \in \mathcal{A} (\Pi_{n-1} (\Omega) \times \R)$.
Let $a \in L^{\infty} ((0, \infty) \times \R^{n-2} \times \R)$ with $a \geq 0$ $\mathcal{H}^n$-a.e., 
and let $f \in \mathscr{F}$.
Then, for every $\mu$-distributed function $u \in W^{1, 1}_{0, \tau} (\Omega)$ we have
\begin{equation} \label{ineq local intro}
\begin{split} 
& \int_{\Phi_n ( B \times \mathbb{S}^1)} a (|x|, y, v_\mu) f (\hat{x} \cdot \nabla_x v_\mu ,  \Dx v_\mu , 
 \nabla_y v_\mu ) \, dx dy  \\
& \leq \int_{\Phi_n ( B \times \mathbb{S}^1)} a (|x|, y, u_0) f (\hat{x} \cdot \nabla_x u_0 ,  \Dx u_0 , 
 \nabla_y u_0 ) \, dx dy,
\end{split}
\end{equation}
for every Borel set $B \subset \Pi_{n-1} (\Omega)$, where $u_0$ and $v_\mu$ 
are given by \eqref{def u0} and \eqref{def vmu}, respectively.
\end{theorem} 
A consequence of the previous result is the following.
\begin{theorem}\label{thm_Polya-Szego inequality}
Let $n \in \mathbb{N}$ with $n \geq 2$, let $\Omega \subset \R^n$ be open, and let $\mu \in \mathcal{A} (\Pi_{n-1} (\Omega) \times \R)$.
Let $a \in L^{\infty} ((0, \infty) \times \R^{n-2} \times \R)$ with $a \geq 0$ $\mathcal{H}^n$-a.e.,
and let $f \in \mathscr{F}$.
Then, for every $\mu$-distributed function $u \in W^{1, 1}_{0, \tau} (\Omega)$ we have
\begin{equation} \label{ineq omega intro}
\begin{split} 
& \int_{\Omega^s} a(|x|, y, v_\mu) f (\hat{x} \cdot \nabla_x v_\mu,  \Dx v_\mu, 
 \nabla_y v_\mu  )  \, dx dy  \\
& \leq \int_{\Omega} a (|x|, y, u) f (\hat{x} \cdot \nabla_x u ,  \Dx u , 
 \nabla_y u ) \, dx dy,
\end{split}
\end{equation}
where $v_\mu$ is given by \eqref{def vmu}.
\end{theorem} 

\begin{remark}
In the statements above $\Omega$ does not need to be bounded.
Moreover, \eqref{ineq local intro} and \eqref{ineq omega intro} are inequalities between extended real numbers in $[0, \infty]$.
\end{remark}

We can actually prove a more general inequality than \eqref{ineq local intro}, see Theorem~\ref{the thing}.
From these results, in particular, it follows that if $p \in [1, + \infty)$ and $u \in W^{1, p}_{0, \tau} (\Omega)$, 
then $v_{\mu} \mid_{\Omega^s} \in W^{1, p}_{0, \tau} (\Omega^s)$, 
see Corollary~\ref{symmetric function in w1p}.

\subsection{Comments on the assumptions of Theorem~\ref{thm_Polya-Szego inequality local}} \label{sect comments polya}

\mbox{ }

\medskip

\noindent
In the next example we show why we need condition (a$_p$) of Definition~\ref{def w1p0tau}.
\begin{example}
Let $n = 2$, let $\Omega = D (1) \cap \{ x = (x_1, x_2) \in \R^2: x_1, x_2 > 0 \}$, and let $u : \Omega \to \R$ be given by
\[
u (x_1, x_2) = x_1 \quad \quad \text{ for every } (x_1, x_2) \in \Omega.
\] 
Then, $\Omega^s = D (1) \cap \{ x = (x_1, x_2) \in \R^2_0 : |x_2|< x_1  \}$ 
and
\[
v_\mu (x_1, x_2) = \frac{x_1^2 - x_2^2}{| x |} \quad \quad \text{ for every } (x_1, x_2) \in \Omega^s,
\] 
where $\mu$ and $v_\mu$ are given by \eqref{mu distr} and \eqref{def vmu}, respectively.
Thus, recalling that $\mathcal{L}^2 (\Omega^s)  = \mathcal{L}^2 (\Omega)$, we have
\[
\int_{\Omega^s} | \nabla v_\mu |^2 \, dx 
= \int_{\Omega^s} \left( 1 + 12 \frac{x_1^2 x_2^2}{|x|^4}\right) \, dx  
> \mathcal{L}^2 (\Omega) = \int_{\Omega} | \nabla u|^2 \, dx.
\]
\end{example}
Note that in the previous example $u \notin W^{1, p}_{0, \tau} (\Omega)$, since condition (a$_p$) of Definition~\ref{def w1p0tau} is not satisfied.
Let us explain why this causes inequality \eqref{pol sze classical} to fail.
First of all, we observe that in this case the perimeter inequality \eqref{per ineq circ eq} 
(with $N=3$ and $B = (0, 1) \times \R$), applied 
to the sets $E= \Sigma^u$ and $E^s=\Sigma^{v_\mu}$, still holds.
However, the boundary conditions of $u$ are such that $u \neq 0$ on $I= \{ (x_1, 0 ) : 0 < x_1 < 1 \}$. 
For this reason, the perimeter of $\Sigma^u$ in the open cylinder $\Phi_3 (B \times \mathbb{S}^1) 
= (D (1) \setminus \{ (0,0)\}) \times \R$ has a non trivial contribution coming
from $\partial \Sigma^u \cap (I \times \R)$.

Then, if we try to reproduce P\'olya's proof in this case, 
inequality \eqref{intermediate with epsilon}
has an additional term appearing in the right-hand side, 
and this does not allow us to conclude.
Note that one cannot simply disregard this additional term since, 
as shown in Figure~\ref{fourth figure}, the perimeter inequality 
does not hold if one cherry-picks subsets of $\partial \Sigma^u$. 
Thus, some boundary conditions (which are encoded in condition (a$_p$)) are needed if one wants to avoid this situation.

\medskip

\noindent

\begin{remark} \label{different vmu}
We observe that, although the function $u$ is defined in $\Omega$, the rearranged function $v_\mu$
is defined in the (larger) set $\Phi_n (\Pi_{n-1} (\Omega) \times \mathbb{S}^1)$.
A natural question is whether it is possible to just define $v_\mu$ in $\Omega^s$, without the need of considering 
the set $\Phi_n (\Pi_{n-1} (\Omega) \times \mathbb{S}^1)$.
In this case, the idea would be to substitute the function $\mu$ given in \eqref{mu distr}
with the function 
\[
\mu' (r, y, t):=  \mathcal{H}^1 (\{ u >t \}_{(r,y)}).
\] 
Note that $\mu'  \leq \mu$, since $\{ u >t \} \subset \{ u_0 >t \}$, and $u_0$ is an extension of $u$ to the whole 
$\Phi_n (\Pi_{n-1} (\Omega) \times \mathbb{S}^1)$.
Then, one could define a rearranged function $w_{\mu'}$ using the following variant of \eqref{def vmu}:
\begin{equation*} 
w_{\mu'} (x, y) := \inf \left\{ t \in \mathbb{R} : \mu' (|x|, y, t) \leq 2 |x|  \arccos (\hat{x} \cdot e_1)  \right\}
\quad \forall \, (x, y) \in \Omega^s \text{ with } \hat{x} \cdot e_1 > - 1.
\end{equation*}
One can check that, if $u \geq 0$ and condition (a$_p$) is satisfied, this definition is equivalent to the one given in \eqref{def vmu}.
That is, $w_{\mu'}$ coincides with the restriction of the function $v_\mu$ given in \eqref{def vmu} to the set $\Omega^s$.
However, if we are dealing with functions that change sign, the situation can be very different, as explained in the following example.
\end{remark}
In the next example we show why we need to impose condition (b) of Definition~\ref{def w1p0tau}, and why an approach like the one described in 
Remark~\ref{different vmu} does not work with functions that can change sign.
%
%
%
\begin{example} \label{counterex neg funct}
Let $n=2$, and let $\phi: (0, +\infty) \times [- \pi, \pi) \to \R^2_0$
be the polar change of coordinates, given by
\[
\phi (r, \theta) = (r \cos \theta, r \sin \theta). 
\]
Let $a \in (0,\pi/4)$ and let $\Omega = \phi (O)$, where
\[
O = \left\{ (r, \theta) : 1 < r < 2,\,\, |\theta| <  \frac{\pi}{4} + a(r-1) \right\}.
\]
Note that, by construction, $\Omega^s = \Omega$.
Consider now the function $u : \Omega \to \R$ given by $u = U \circ \phi^{-1}$, where 
$U: O \to \R$ is the piecewise affine function given by
\[
U (r, \theta) = 
\begin{cases}
| \theta | -  \frac{\pi}{4} \quad &\textnormal{if } | \theta | \leq \frac{\pi}{4}, \vspace*{0.2cm}\\
0\quad &\textnormal{otherwise in }O.
\end{cases}
\]
Let now $\mu'$ and $w_{\mu'}$ be given by Remark~\ref{different vmu}.
We have $w_{\mu'} = V \circ \phi^{-1}$, where 
$V: O \to \R$ is the following piecewise affine function:
\begin{align*}
V (r, \theta) =
\begin{cases}
0 \quad &\textnormal{if } | \theta | \leq a(r-1), \vspace*{0.2cm}\\
- | \theta | + a(r-1) \quad &\textnormal{if } a(r-1) <  | \theta |  < a(r-1) +\frac{\pi}{4}.
\end{cases}
\end{align*}
In this case, $u$ satisfies condition $(a_p)$ for every $p \in [1, +\infty)$. However, 
a direct calculation shows that 
\[
\int_{\Omega^s} |\nabla w_{\mu'} |^2\, dx 
= \frac{\pi}{2}\log(2) + \frac{3}{4}\pi a^2
\quad \text{ and } \quad   
\int_{\Omega} |\nabla u |^2\, dx = \frac{\pi}{2}\log(2).
\]
\end{example}
\noindent
In the 
example above, we have 
\[ 
u = 0 \quad \text{ on } \partial \Omega \cap \{ 1 < |x| < 2 \},
\]
so that $u$ satisfies condition $(a_p)$. Note, however, that $u \notin W^{1, p}_{0, \tau} (\Omega)$, since condition (b) of Definition~\ref{def w1p0tau} is not satisfied.
In polar coordinates, $u$ is independent of $r$ and
has global minimum $- \pi/4$, which is attained in the segment $\{ (x_1, 0 ) : 1 < x_1 < 2 \}$.
Instead, the function $w_{\mu'}$ satisfies
\[
w_{\mu'} = - \frac{\pi}{4} \quad \text{ on } \partial \Omega \cap \{ 1 < |x| < 2 \}, 
\]
Thus, the construction of the function $w_{\mu'}$ shifts the minimum value $- \pi/4$ 
towards (a subset of) the boundary of $\Omega$.
This means that, depending on the shape of $\Omega$, we can 
force the function $w_{\mu'}$ to also depend on $r$, therefore creating additional gradient.
In fact, we have 
\[
| \nabla u (x) |^2 = \frac{1}{|x|^2} 
\quad \quad \text{ for every } \, x \in \{\nabla u \neq 0 \},
\]
and
\[
| \nabla w_{\mu'} (x) |^2 = \frac{1}{|x|^2} + a^2
\quad \quad \text{ for every } \, x \in \{\nabla w_{\mu'} \neq 0 \},
\]
%
and this eventually causes the failure of P\'olya-Szeg\"o inequality.

\medskip

\subsection{Rigidity} \label{intro rigidity}
Once a P\'olya--Szeg\"o inequality is established, we want to understand if the equality
\begin{equation} \label{equality}
\begin{split} 
& \int_{\Phi_n ( \Pi_{n-1} (\Omega) \times \mathbb{S}^1)} a (|x|, y, v_\mu) f (\hat{x} \cdot \nabla_x v_\mu ,  \Dx v_\mu , 
 \nabla_y v_\mu ) \, dx dy  \\
& = \int_{\Phi_n ( \Pi_{n-1} (\Omega) \times \mathbb{S}^1)} a (|x|, y, u_0) f (\hat{x} \cdot \nabla_x u_0 ,  \Dx u_0 , 
 \nabla_y u_0 ) \, dx dy
\end{split}
\end{equation}
implies that (up to orthogonal transformations) $u = v_\mu$.  
This question makes sense if the integrals under consideration are finite, so we will require that
\begin{equation} \label{int finite}
\int_{\Phi_n ( \Pi_{n-1} (\Omega) \times \mathbb{S}^1)} a (|x|, y, v_\mu) f (\hat{x} \cdot \nabla_x v_\mu ,  \Dx v_\mu. 
 \nabla_y v_\mu ) \, dx dy < \infty, 
\end{equation}
When \eqref{int finite} is satisfied, we define the set of extremals of \eqref{ineq local intro} as
\[
\mathcal{E} (\mu, \Omega) := \{ u \in W^{1, 1}_{0, \tau} (\Omega): \text{$u$ is $\mu$-distributed and \eqref{equality} holds}  \}.
\]
First of all, note that the following inclusion is always satisfied:
\[
\mathcal{E} (\mu, \Omega) \supset \{ u \in W^{1, 1}_{0, \tau} (\Omega): \, \, u (x, y) = v_\mu (R x, y) 
\text{ for $\mathcal{H}^n$-a.e. $(x, y) \in \Omega$, for some $R \in O (2)$} \},
\]
where $O (2)$ is the set of orthogonal transformations of $\mathbb{R}^2$.
We will say that \textit{rigidity holds} for \eqref{ineq local intro} if also the
opposite inclusion is true, that is if
\begin{equation} \label{rigidity} \tag{$\mathcal{R}$}
\mathcal{E} (\mu, \Omega)  = \{ u \in W^{1, 1}_{0, \tau} (\Omega): u (x, y) = v_\mu (R x, y) 
\text{ for $\mathcal{H}^n$-a.e. $(x, y) \in \Omega$, for some $R \in O (2)$} \}.
\end{equation}

\medskip

As observed by Kawohl, in \cite[pag. 186]{polyaszego51} P\'olya and Szeg\"o
dismiss the study of rigidity (for the Steiner rearrangement) as ``hopeless''.
Indeed, P\'olya's proof of \eqref{pol sze classical} given above
does not provide any information about the functions $u$ satisfying equality.
Moreover, the proof of \eqref{willem} by Smets and Willem (see \cite{SmetsWillem03})
is obtained by combining polarization and approximation arguments
and, also in this case, this does not allow to study rigidity.

To the best of our knowledge, the problem of rigidity for the P\'olya--Szeg\"o inequality
in the context of circular symmetrization was firstly considered by Kawohl.
He showed that if in \eqref{strong assump kawohl} one further assumes that $\Omega$ 
is an annulus, $F$ is continuous and positive, $G$ is increasing and strictly convex, 
and $u$ is analytic, then rigidity holds \cite[Corollary~2.35, iii)]{kawohl_book_85}. 
The proof relies on the smoothness of $u$, and uses the implicit function theorem.

Our goal is to show that rigidity holds also in higher dimensions, under much milder assumptions. To this aim, 
if $\mu \in \mathcal{A} (\Pi_{n-1}(\Omega) \times \R)$, we define the function $\alpha_{\mu}: \Pi_{n-1} (\Omega) \times \R \to [0, \pi]$ as
\begin{equation} \label{def alpha}
\alpha_{\mu} (r, y, t) := \frac{\mu (r, y, t)}{2 r} \quad \text{ for every }  (r, y, t) \in \Pi_{n-1} (\Omega) \times \R.
\end{equation}
For every $(r, y, t) \in \Pi_{n-1} (\Omega) \times \R$, the number $\alpha_{\mu} (r, y, t)$ gives 
half of the angle corresponding to a connected arc of $\partial D (r)$ with length $\mu (r, y, t)$. 
Thanks to Proposition~\ref{derivatives of mu}, whenever $\mu \in \mathcal{A}(\Pi_{n-1}(\Omega) \times \mathbb{R})$, we have 
that $\alpha_\mu \in BV (A \times (-d, + \infty))$ for every open set $A \subset \subset \Pi_{n-1} (\Omega)$ 
and for every $d > 0$. We will make the following assumption:
\begin{equation} \label{set non empty}
\{ 0 < \alpha_\mu < \pi \} \text{ is not $\mathcal{H}^n$-equivalent to the empty set.} 
\end{equation}
This is because, if $\{ 0 < \alpha_\mu < \pi \}$ is empty, then rigidity trivially holds.
Indeed, in this case \textbf{every} $\mu$-distributed function is $\mathcal{H}^n$-equivalent to $v_{\mu}$. 

We can now state our second main result, which gives a sufficient condition for rigidity. 
This is written in terms of the notion of essential connectedness, see Section~\ref{section_preliminaries} and also \cite{CagnettiColomboDePhilippisMaggiSteiner, ccdpmGAUSS}.
Roughly speaking, one can create a counterexample to rigidity 
whenever the set $\{ 0 < \alpha_\mu < \pi \}$
is ``split into two pieces'' by points where
$\alpha_\mu = 0$, $\alpha_\mu = \pi$, or by points where  
the singular part of the distributional derivative $D \alpha_{\mu}$ is concentrated.
Here, $\alpha^\wedge$ and $\alpha^\vee_\mu$ are the approximate liminf and the approximate limsup of $\alpha_\mu$, respectively, 
while $S_{\alpha_\mu} = \{ \alpha^\wedge < \alpha^\vee_\mu \}$ is the singular set of $\alpha_\mu$, see again Section~\ref{section_preliminaries}.
\begin{theorem} \label{thm suff conditions}
Let $n \in \mathbb{N}$ with $n \geq 2$, let $\Omega \subset \R^n$ be open, and let $\mu \in \mathcal{A} (\Pi_{n-1} (\Omega) \times \R)$.
Let $a \in L^{\infty} ((0, \infty) \times \R^{n-2} \times \R)$ with $a > 0$ $\mathcal{H}^n$-a.e.,
let $f \in \mathscr{F}'$, and suppose that \eqref{int finite} and \eqref{set non empty} hold.
Assume, in addition, that the Cantor part $D^c \alpha_\mu$ of $D \alpha_\mu$ 
is concentrated on a Borel set $K \subset \Pi_{n-1} (\Omega) \times \R$ such that
\[
\{ \alpha^\wedge_\mu = 0 \} \cup \{ \alpha^\vee_\mu = \pi \} \cup S_{\alpha_\mu} \cup K 
\text{ does not essentially disconnect } \{ 0 < \alpha_\mu < \pi \}.
\]
Then, \eqref{rigidity} holds.
\end{theorem}

\subsection{Comments on the assumptions of Theorem~\ref{thm suff conditions}} \label{sect comments rigidity}

First of all let us observe that, without assuming strict convexity of the integrand $f$ (i.e. without assuming $f \in \mathscr{F}'$), 
one cannot expect rigidity to hold.

\begin{example}
Let $n=2$, $a \equiv 1$, and $f (\eta, \tau) = |\tau|$. 
Let $\Omega = \{ x \in \R^2 : 1 < |x| < 2 \}$, and let $\phi: (0, +\infty) \times [- \pi, \pi) \to \R^2_0$ be as in Example~\ref{counterex neg funct}. 
Let $O = \left\{ (r, \theta) : 1 < r < 2 \right\}$, and let $u = W \circ \phi^{-1}$,
%
%
where $W: O \to \R$ is defined as
\begin{align*}
W (r, \theta) =
\begin{cases}
2 \theta + \frac{\pi }{2} \quad &\textnormal{if } \theta \in \left[ - \frac{\pi}{4} , 0  \right], \vspace*{0.2cm}\\
- \frac{2}{3} \theta + \frac{\pi}{2}  \quad &\textnormal{if } \theta \in \left[  0 , \frac{3}{4} \pi \right], \vspace*{0.2cm}\\
0\quad &\textnormal{otherwise in }O.
\end{cases}
\end{align*}
Then, we have $\Omega^s = \Omega$ and $v_\mu = T \circ \phi^{-1}$, where 
$T: O \to \R$ is given by
\[
T (r, \theta) = 
\begin{cases}
- | \theta | +  \frac{\pi}{2} \quad &\textnormal{if } | \theta | \leq \frac{\pi}{2}, \vspace*{0.2cm}\\
0\quad &\textnormal{otherwise in }D.
\end{cases}
\]
In this case, 
\[
\int_\Omega | \Dx v_\mu | \, dx = \int_\Omega | \Dx u | \, dx = \pi,
\]
so that rigidity fails.
\end{example}

\begin{remark}
More in general, let $a \equiv 1$ and $f (\eta, \tau, \zeta) = |\tau|$, and suppose 
 that $\Omega$ is bounded.
Then, any smooth function $u \in W^{1, 1}_{0, \tau} (\Omega) \cap W^{1, 1} (\Omega)$ with the property that
\begin{equation} \label{nice property u}
\{ u_0 > t \}_{(r, y)} \text{ is a (possibly empty) connected arc \quad }
\forall \, (r, y, t) \in\Pi_{n-1} (\Omega) \times \R, 
\end{equation}
belongs to $\mathcal{E} (\mu, \Omega)$, and therefore rigidity fails.

Indeed, using property ii) of Proposition~\ref{volpert theorem}), we have 
\begin{align*}
(\partial \Sigma^{u_0})_{(r, y, t)} 
= \partial \left((\Sigma^{u_0})_{(r, y, t)} \right)
= \partial (\{ u_0 > t \}_{(r, y)} ), \quad \text{ for }\mathcal{L}^{n}\text{-a.e. } 
(r, y, t) \in \Pi_{n-1} (\Omega) \times \R.
\end{align*}
%
%
%
Then, by Coarea Formula (see Proposition~\ref{coarea formula}) and Remark~\ref{rem codim 1} we have
\begin{align*}
&\int_{\Omega} | \Dx u |\, dx \, dy 
= \int_{\partial \Sigma^{u_0} \cap (\Omega \times \mathbb{R})} \frac{| \Dx u_0 |}{\sqrt{1 + |\nabla u_0 |^2 }} \, d \mathcal{H}^n (x, y, t) \\
&= \int_{\left( \Pi_{n-1} (\Omega) \times \R\right) \cap \{ \Dx u_0 \neq 0 \}} 
\left(  \int_{(\partial \Sigma^{u_0})_{(r, y, t)}} \, d \mathcal{H}^0 (x) \right) \, dr \, dy \, dt \\
&= \int_{\Pi_{n-1} (\Omega) \times \R} 
\left(  \int_{(\partial \Sigma^{u_0})_{(r, y, t)}} \, d \mathcal{H}^0 (x) \right) \, dr \, dy \, dt  \\
&= \int_{\Pi_{n-1} (\Omega) \times \R} \left(  \int_{(\partial \Sigma^{v_\mu})_{(r, y, t)}} \, d \mathcal{H}^0 (x) \right) \, dr \, dy \, dt \\
&= \int_{\left( \Pi_{n-1} (\Omega) \times \R \right) \cap \{ \Dx v_\mu \neq 0 \}} 
\left(  \int_{(\partial \Sigma^{v_\mu})_{(r, y, t)}} \, d \mathcal{H}^0 (x) \right) \, dr \, dy \, dt \\
&= \int_{\Omega} | \Dx v_\mu |\, dx \, dy,
\end{align*}
where we used the fact that, by definition, 
$v_\mu$ satisfies property \eqref{nice property u}.
In the context of Schwarz rearrangement, a similar phenomenon 
had already been observed by Brothers and Ziemer \cite{brothersziemer}.
\end{remark}

\medskip

Let us show that even asking $f \in \mathscr{F}'$ is not enough to guarantee rigidity.
\begin{example} \label{no rigidity values of alpha}
Let $n = 2$, let $\Omega = D (5)$, let $a \equiv 1$, and let $f \in \mathscr{F}'$ be given by 
\[
f (\eta, \tau) = \eta^2 + \tau^2.
\]
Let \(x' = (2, 0)\), \(x'' = (4, 0)\), \( x''' = (0, 2) \), and let 
\[ 
u (x) = 2 \max \{ 0, 1 - | x  |, 1 - | x - x'' |, 1 - | x - x''' | \}  \quad \text{ for every $x \in \Omega$.}
\]
In this case $\Omega^s = \Omega$ and 
the circular rearrangement of $u$ is given by 
\[
v_{\mu} (x) = 2 \max \{ 0, 1 - | x  |, 1 - | x - x' |, 1 - | x - x'' | \}  \quad \text{ for every $x \in \Omega$,}
\]
see Figure~\ref{figure no rigidity 1}.
\tikzset{every picture/.style={line width=0.75pt}} 
\begin{center}
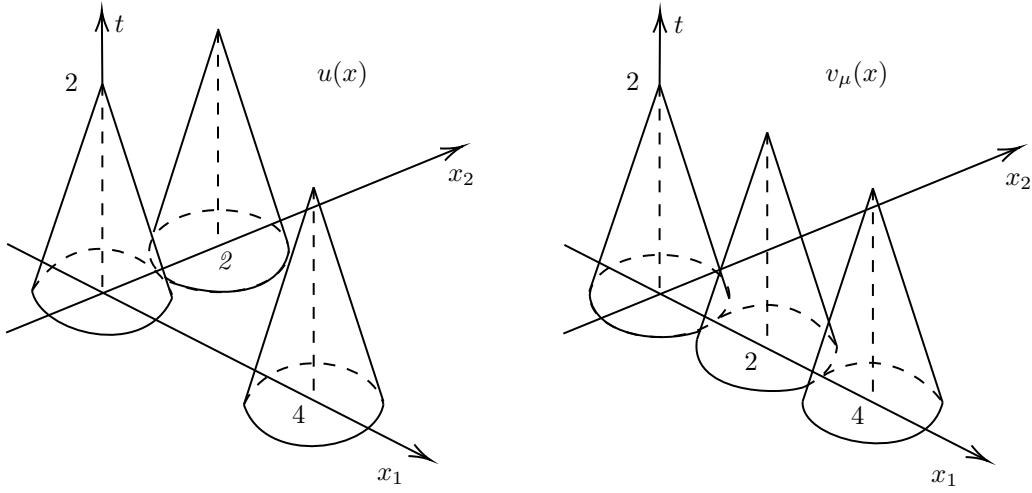
\begin{figure}[!htb]
\begin{tikzpicture}[x=0.75pt,y=0.75pt,yscale=-1,xscale=1]

\draw  [dash pattern={on 4.5pt off 4.5pt}]  (161.51,135.32) -- (161.51,242.82) ;
\draw    (113.25,260.2) -- (340.4,167.46) ;
\draw [shift={(342.25,166.7)}, rotate = 157.79] [color={rgb, 255:red, 0; green, 0; blue, 0 }  ][line width=0.75]    (10.93,-3.29) .. controls (6.95,-1.4) and (3.31,-0.3) .. (0,0) .. controls (3.31,0.3) and (6.95,1.4) .. (10.93,3.29)   ;
\draw    (113.75,215.7) -- (324.47,323.29) ;
\draw [shift={(326.25,324.2)}, rotate = 207.05] [color={rgb, 255:red, 0; green, 0; blue, 0 }  ][line width=0.75]    (10.93,-3.29) .. controls (6.95,-1.4) and (3.31,-0.3) .. (0,0) .. controls (3.31,0.3) and (6.95,1.4) .. (10.93,3.29)   ;
\draw    (161.51,135.32) -- (126.25,239.18) ;
\draw    (161.51,135.32) -- (196.51,242.82) ;
\draw  [dash pattern={on 4.5pt off 4.5pt}] (185.01,219.35) .. controls (184.99,207.97) and (200.65,198.73) .. (219.98,198.7) .. controls (239.31,198.67) and (254.99,207.87) .. (255.01,219.25) .. controls (255.02,230.63) and (239.36,239.87) .. (220.03,239.9) .. controls (200.7,239.93) and (185.02,230.73) .. (185.01,219.35) -- cycle ;
\draw  [dash pattern={on 4.5pt off 4.5pt}]  (267.51,187.32) -- (267.51,294.82) ;
\draw    (267.51,187.32) -- (232.51,296.35) ;
\draw    (267.51,187.32) -- (302.51,294.82) ;
\draw  [dash pattern={on 4.5pt off 4.5pt}]  (219.51,108.82) -- (219.51,216.32) ;
\draw    (219.51,108.82) -- (187.25,209.75) ;
\draw    (219.51,107.82) -- (254.51,215.32) ;
\draw  [dash pattern={on 4.5pt off 4.5pt}] (406.01,241.85) .. controls (405.99,230.47) and (421.65,221.23) .. (440.98,221.2) .. controls (460.31,221.17) and (475.99,230.37) .. (476.01,241.75) .. controls (476.02,253.13) and (460.36,262.37) .. (441.03,262.4) .. controls (421.7,262.43) and (406.02,253.23) .. (406.01,241.85) -- cycle ;
\draw  [dash pattern={on 4.5pt off 4.5pt}]  (441.01,135.82) -- (441.01,243.32) ;
\draw    (392.75,260.7) -- (619.9,167.96) ;
\draw [shift={(621.75,167.2)}, rotate = 157.79] [color={rgb, 255:red, 0; green, 0; blue, 0 }  ][line width=0.75]    (10.93,-3.29) .. controls (6.95,-1.4) and (3.31,-0.3) .. (0,0) .. controls (3.31,0.3) and (6.95,1.4) .. (10.93,3.29)   ;
\draw    (393.25,216.2) -- (603.97,323.79) ;
\draw [shift={(605.75,324.7)}, rotate = 207.05] [color={rgb, 255:red, 0; green, 0; blue, 0 }  ][line width=0.75]    (10.93,-3.29) .. controls (6.95,-1.4) and (3.31,-0.3) .. (0,0) .. controls (3.31,0.3) and (6.95,1.4) .. (10.93,3.29)   ;
\draw    (441.01,135.82) -- (405.75,239.68) ;
\draw    (441.01,135.82) -- (476.01,243.32) ;
\draw  [dash pattern={on 4.5pt off 4.5pt}]  (495.01,159.82) -- (495.01,267.32) ;
\draw    (495.01,159.82) -- (459.75,263.68) ;
\draw    (495.01,159.82) -- (530.01,267.32) ;
\draw  [dash pattern={on 4.5pt off 4.5pt}]  (548.01,187.82) -- (548.01,295.32) ;
\draw    (548.01,187.82) -- (512.75,294.75) ;
\draw    (548.01,187.82) -- (583.01,295.32) ;
\draw    (161.51,135.32) -- (161.26,100.25) ;
\draw [shift={(161.25,98.25)}, rotate = 89.6] [color={rgb, 255:red, 0; green, 0; blue, 0 }  ][line width=0.75]    (10.93,-3.29) .. controls (6.95,-1.4) and (3.31,-0.3) .. (0,0) .. controls (3.31,0.3) and (6.95,1.4) .. (10.93,3.29)   ;
\draw    (441.01,135.82) -- (441.24,100.75) ;
\draw [shift={(441.25,98.75)}, rotate = 90.37] [color={rgb, 255:red, 0; green, 0; blue, 0 }  ][line width=0.75]    (10.93,-3.29) .. controls (6.95,-1.4) and (3.31,-0.3) .. (0,0) .. controls (3.31,0.3) and (6.95,1.4) .. (10.93,3.29)   ;
\draw    (126.25,239.18) .. controls (130.5,260.58) and (182.75,274.25) .. (196.51,242.82) ;
\draw  [dash pattern={on 4.5pt off 4.5pt}]  (126.25,239.18) .. controls (149.75,202.25) and (189.25,221.25) .. (196.51,241.25) ;
\draw [line width=0.75]    (195.25,234.75) .. controls (226.25,245.75) and (251.75,237.75) .. (255.01,219.25) ;
\draw  [dash pattern={on 4.5pt off 4.5pt}]  (232.51,296.35) .. controls (246.75,265.25) and (296.75,272.25) .. (302.51,296.25) ;
\draw    (232.51,296.35) .. controls (242.75,327.75) and (297.75,320.25) .. (302.51,296.25) ;
\draw    (406.01,241.85) .. controls (408.25,261.25) and (440.75,266.75) .. (461.75,259.25) ;
\draw    (459.75,263.68) .. controls (455.75,290.25) and (499.25,292.25) .. (513.5,287.5) ;
\draw  [dash pattern={on 4.5pt off 4.5pt}]  (513.5,287.5) .. controls (522.75,282.75) and (528.75,277.25) .. (530.01,267.32) ;
\draw  [dash pattern={on 4.5pt off 4.5pt}]  (471.5,252.5) .. controls (498.75,238.75) and (518.25,249.75) .. (530.01,267.32) ;
\draw    (512.75,294.75) .. controls (512.7,295.62) and (512.76,296.47) .. (512.92,297.31) .. controls (516.62,316.67) and (574.62,327.43) .. (583.01,295.32) ;
\draw  [dash pattern={on 4.5pt off 4.5pt}]  (523.25,281.75) .. controls (548.75,266.75) and (572.75,280.75) .. (583.01,295.32) ;

\draw (308,332.4) node   [align=left] {\begin{minipage}[lt]{13.06pt}\setlength\topsep{0pt}
$x_1$
\end{minipage}};
\draw (343.6,182.3) node   [align=left] {\begin{minipage}[lt]{13.06pt}\setlength\topsep{0pt}
$x_2$
\end{minipage}};
\draw (586,333.5) node   [align=left] {\begin{minipage}[lt]{13.06pt}\setlength\topsep{0pt}
$x_1$
\end{minipage}};
\draw (623.2,184.3) node   [align=left] {\begin{minipage}[lt]{13.06pt}\setlength\topsep{0pt}
$x_2$
\end{minipage}};
\draw (173.63,105.13) node   [align=left] {\begin{minipage}[lt]{8.67pt}\setlength\topsep{0pt}
$t$
\end{minipage}};
\draw (454.13,105.63) node   [align=left] {\begin{minipage}[lt]{8.67pt}\setlength\topsep{0pt}
$t$
\end{minipage}};
\draw (149.13,134.88) node   [align=left] {\begin{minipage}[lt]{9.69pt}\setlength\topsep{0pt}
$2$
\end{minipage}};
\draw (430.63,134.38) node   [align=left] {\begin{minipage}[lt]{9.69pt}\setlength\topsep{0pt}
$2$
\end{minipage}};
\draw (262.88,301.13) node   [align=left] {\begin{minipage}[lt]{9.35pt}\setlength\topsep{0pt}
$4$
\end{minipage}};
\draw (225.38,224.13) node   [align=left] {\begin{minipage}[lt]{9.35pt}\setlength\topsep{0pt}
2
\end{minipage}};
\draw (489.88,274.63) node   [align=left] {\begin{minipage}[lt]{9.35pt}\setlength\topsep{0pt}
$2$
\end{minipage}};
\draw (543.38,301.63) node   [align=left] {\begin{minipage}[lt]{9.35pt}\setlength\topsep{0pt}
$4$
\end{minipage}};
\draw (276.13,130.63) node   [align=left] {\begin{minipage}[lt]{10.37pt}\setlength\topsep{0pt}
$u (x)$
\end{minipage}};
\draw (531.13,130.63) node   [align=left] {\begin{minipage}[lt]{10.37pt}\setlength\topsep{0pt}
$v_{\mu} (x)$
\end{minipage}};

\end{tikzpicture}

\caption{The function $u$ given in Example~\ref{no rigidity values of alpha} and its circular rearrangement $v_{\mu}$. 
In this case, rigidity fails. This is because the set $\{ 0 < \alpha_{\mu} < \pi \}$ is essentially disconnected, see Figure~\ref{figure distribution}.}
       \label{figure no rigidity 1}
\end{figure}
\end{center}
One can check that $u \in \mathcal{E} (\mu, \Omega)$, since
\[
\int_{\Omega} | \nabla u (x)|^2 \, dx = \int_{\Omega} | \nabla v_{\mu} (x)|^2 \, dx 
= 12 \pi.
\]
Note that $u$ cannot be written as the composition of 
the symmetric function $v_{\mu}$ with an orthogonal transformation, 
so rigidity is violated. 
\end{example}

\begin{remark}
Let us clarify why rigidity is violated in Example~\ref{no rigidity values of alpha}.
Observing that in this case $\Pi_1 (\Omega) = (0, 5)$, 
it will be convenient to visualize the behaviour of the function $\alpha_\mu: (0, 5) \times \R \to [ 0, \pi ]$, 
see Figure~\ref{figure distribution}.
We note that, up to removing the singleton $\{ (3, 0) \}$ (which is $\mathcal{H}^1$-negligible), 
the set $\{ 0 < \alpha_\mu < \pi \}$ is disconnected.
For this reason, it is possible to rotate only the part of the graph of $v_{\mu}$
lying above the annulus $\{ 1 < |x| < 3 \}$, and to obtain functions 
for which the value of the integral is the same.
\begin{center}
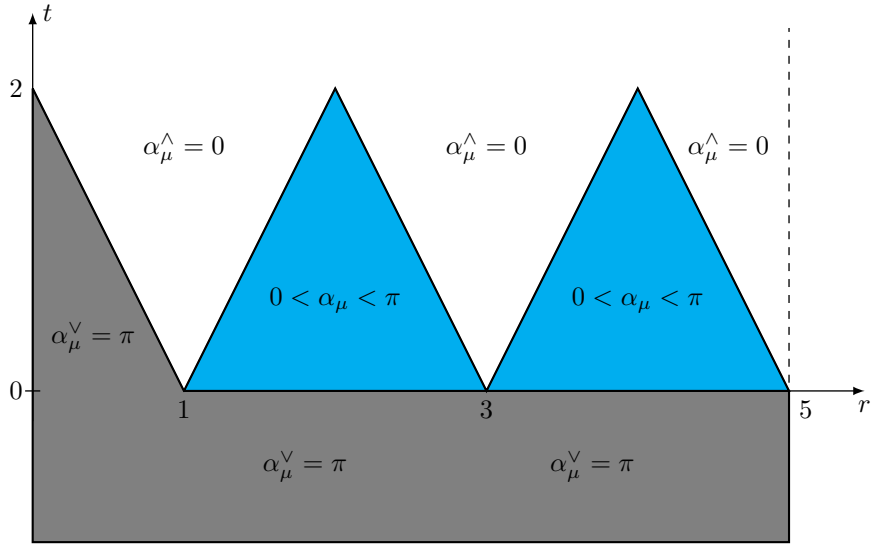
\begin{figure}[!htb]
  \begin{tikzpicture}[>=latex, scale=2]

 \draw[fill=gray, style=thick](0,2)--(1,0)--(5, 0)
--(5, -1)--(0,-1)--(0,2);

 \draw[fill=cyan, style=thick](1,0)--(2, 2)--(3, 0)--(4, 2)--(5, 0)--(1, 0);

   \draw[line width = .5pt] (-.05,0)--(.05,0); 
  \draw(0,0)node[left]{$0$};
  \draw[->, line width = .5pt] (5,0)--(5.5,0); 
  \draw(5.5,0)node[below]{$r$};
 \draw[->, line width = .5pt] (0,-.2)--(0,2.5); 
  \draw(0,2.5)node[right]{$t$};
  \draw(0,2)node[left]{$2$};

  \draw(1,0)node[below]{$1$};

  \draw(3,0)node[below]{$3$};
  
  \draw(5,0)node[below right]{$5$};

  \draw[line width = .5pt] (0,2)--(1,0)--(2, 2)--(3, 0)--(4, 2)--(5, 0); 
   \draw[dashed, line width = .5pt] (1,0)--(5,0); 

  \draw(.4,.35)node{$\alpha^{\vee}_{\mu} = \pi$};
  \draw(1.8,-.5)node{$\alpha^{\vee}_{\mu} = \pi$};
  \draw(3.7,-.5)node{$\alpha^{\vee}_{\mu} = \pi$};

  \draw(1,1.6)node{$\alpha^{\wedge}_{\mu} = 0$};
  \draw(3,1.6)node{$\alpha^{\wedge}_{\mu} = 0$};
  \draw(4.6,1.6)node{$\alpha^{\wedge}_{\mu} =  0$};

  \draw(2,.6)node{$ 0 < \alpha_{\mu} < \pi$};
  \draw(4,.6)node{$ 0 < \alpha_{\mu} < \pi$};

   \draw[dashed, line width = .5pt] (5,-1)--(5, 2.4);

\end{tikzpicture}
\caption{The values of $\alpha_{\mu} (r, t)$ when $(r, t) \in (0, 5) \times \R$, 
in Example~\ref{no rigidity values of alpha}.
Note that the set $\{ 0 < \alpha_{\mu} < \pi \}$ is essentially disconnected, 
since the singleton $\{ (3, 0) \}$ has $\mathcal{H}^1$-measure zero.}
       \label{figure distribution}
\end{figure}
\end{center}
\end{remark}

\begin{example} \label{no rigidity values of alpha 2}
Let us modify Example~\ref{no rigidity values of alpha}, by removing 
one of the cones in the graph of $u$.
More precisely, let $n =2$, and let $\Omega$, $a$, $f$, and $x', x'''$ be as in Example~\ref{no rigidity values of alpha}.
Let now $u: \Omega \to \R$ be given by 
\[ 
u (x) = 2 \max \{ 0, 1 - | x  |, 1 - | x - x''' | \}  \quad \text{ for every } x \in \Omega.
\]
Denoting by $\mu$ the distribution of $u$ (note that this is not the same as the distribution $\mu$
in Example~\ref{no rigidity values of alpha}), the circular rearrangement of $u$
is now given by
\[
v_{\mu} (x) = 2 \max \{ 0, 1 - | x  |, 1 - | x - x' | \}  \quad \text{ for every } x \in \Omega,
\]
see Figure~\ref{figure no rigidity 2}.
\begin{center}
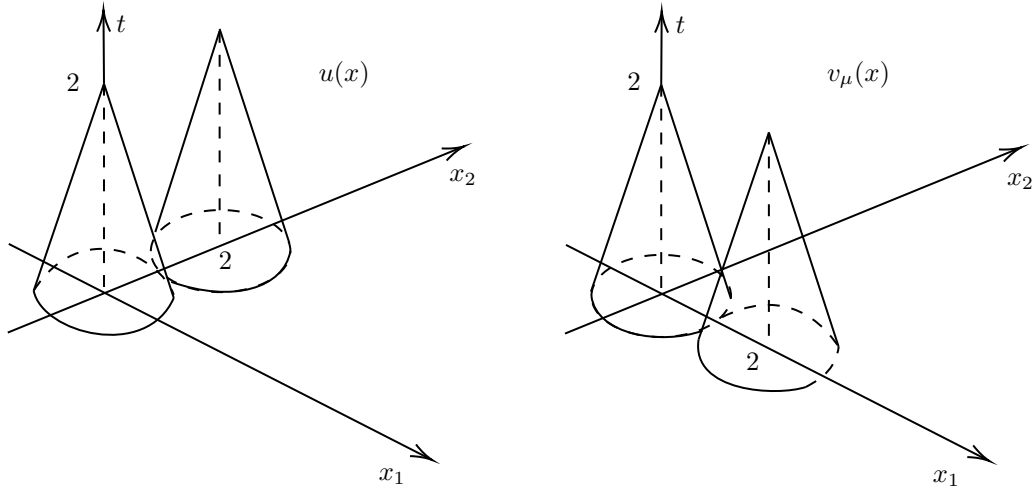
\begin{figure}[!htb]
\begin{tikzpicture}[x=0.75pt,y=0.75pt,yscale=-1,xscale=1]

\draw  [dash pattern={on 4.5pt off 4.5pt}]  (161.51,135.32) -- (161.51,242.82) ;
\draw    (113.25,260.2) -- (340.4,167.46) ;
\draw [shift={(342.25,166.7)}, rotate = 157.79] [color={rgb, 255:red, 0; green, 0; blue, 0 }  ][line width=0.75]    (10.93,-3.29) .. controls (6.95,-1.4) and (3.31,-0.3) .. (0,0) .. controls (3.31,0.3) and (6.95,1.4) .. (10.93,3.29)   ;
\draw    (113.75,215.7) -- (324.47,323.29) ;
\draw [shift={(326.25,324.2)}, rotate = 207.05] [color={rgb, 255:red, 0; green, 0; blue, 0 }  ][line width=0.75]    (10.93,-3.29) .. controls (6.95,-1.4) and (3.31,-0.3) .. (0,0) .. controls (3.31,0.3) and (6.95,1.4) .. (10.93,3.29)   ;
\draw    (161.51,135.32) -- (126.25,239.18) ;
\draw    (161.51,135.32) -- (196.51,242.82) ;
\draw  [dash pattern={on 4.5pt off 4.5pt}] (185.01,219.35) .. controls (184.99,207.97) and (200.65,198.73) .. (219.98,198.7) .. controls (239.31,198.67) and (254.99,207.87) .. (255.01,219.25) .. controls (255.02,230.63) and (239.36,239.87) .. (220.03,239.9) .. controls (200.7,239.93) and (185.02,230.73) .. (185.01,219.35) -- cycle ;




\draw  [dash pattern={on 4.5pt off 4.5pt}]  (219.51,108.82) -- (219.51,216.32) ;

\draw    (219.51,108.82) -- (187.25,209.75) ;

\draw    (219.51,107.82) -- (254.51,215.32) ;
\draw  [dash pattern={on 4.5pt off 4.5pt}] (406.01,241.85) .. controls (405.99,230.47) and (421.65,221.23) .. (440.98,221.2) .. controls (460.31,221.17) and (475.99,230.37) .. (476.01,241.75) .. controls (476.02,253.13) and (460.36,262.37) .. (441.03,262.4) .. controls (421.7,262.43) and (406.02,253.23) .. (406.01,241.85) -- cycle ;

\draw  [dash pattern={on 4.5pt off 4.5pt}]  (441.01,135.82) -- (441.01,243.32) ;

\draw    (392.75,260.7) -- (619.9,167.96) ;

\draw [shift={(621.75,167.2)}, rotate = 157.79] [color={rgb, 255:red, 0; green, 0; blue, 0 }  ][line width=0.75]    (10.93,-3.29) .. controls (6.95,-1.4) and (3.31,-0.3) .. (0,0) .. controls (3.31,0.3) and (6.95,1.4) .. (10.93,3.29)   ;

\draw    (393.25,216.2) -- (603.97,323.79) ;

\draw [shift={(605.75,324.7)}, rotate = 207.05] [color={rgb, 255:red, 0; green, 0; blue, 0 }  ][line width=0.75]    (10.93,-3.29) .. controls (6.95,-1.4) and (3.31,-0.3) .. (0,0) .. controls (3.31,0.3) and (6.95,1.4) .. (10.93,3.29)   ;

\draw    (441.01,135.82) -- (405.75,239.68) ;

\draw    (441.01,135.82) -- (476.01,243.32) ;

\draw  [dash pattern={on 4.5pt off 4.5pt}]  (495.01,159.82) -- (495.01,267.32) ;

\draw    (495.01,159.82) -- (459.75,263.68) ;

\draw    (495.01,159.82) -- (530.01,267.32) ;




\draw    (161.51,135.32) -- (161.26,100.25) ;

\draw [shift={(161.25,98.25)}, rotate = 89.6] [color={rgb, 255:red, 0; green, 0; blue, 0 }  ][line width=0.75]    (10.93,-3.29) .. controls (6.95,-1.4) and (3.31,-0.3) .. (0,0) .. controls (3.31,0.3) and (6.95,1.4) .. (10.93,3.29)   ;

\draw    (441.01,135.82) -- (441.24,100.75) ;

\draw [shift={(441.25,98.75)}, rotate = 90.37] [color={rgb, 255:red, 0; green, 0; blue, 0 }  ][line width=0.75]    (10.93,-3.29) .. controls (6.95,-1.4) and (3.31,-0.3) .. (0,0) .. controls (3.31,0.3) and (6.95,1.4) .. (10.93,3.29)   ;

\draw    (126.25,239.18) .. controls (130.5,260.58) and (182.75,274.25) .. (196.51,242.82) ;
\draw  [dash pattern={on 4.5pt off 4.5pt}]  (126.25,239.18) .. controls (149.75,202.25) and (189.25,221.25) .. (196.51,241.25) ;

\draw [line width=0.75]    (195.25,234.75) .. controls (226.25,245.75) and (251.75,237.75) .. (255.01,219.25) ;



\draw    (406.01,241.85) .. controls (408.25,261.25) and (440.75,266.75) .. (461.75,259.25) ;

\draw    (459.75,263.68) .. controls (455.75,290.25) and (499.25,292.25) .. (513.5,287.5) ;

\draw  [dash pattern={on 4.5pt off 4.5pt}]  (513.5,287.5) .. controls (522.75,282.75) and (528.75,277.25) .. (530.01,267.32) ;

\draw  [dash pattern={on 4.5pt off 4.5pt}]  (471.5,252.5) .. controls (498.75,238.75) and (518.25,249.75) .. (530.01,267.32) ;



\draw (308,332.4) node   [align=left] {\begin{minipage}[lt]{13.06pt}\setlength\topsep{0pt}
$x_1$
\end{minipage}};
\draw (343.6,182.3) node   [align=left] {\begin{minipage}[lt]{13.06pt}\setlength\topsep{0pt}
$x_2$
\end{minipage}};
\draw (586,333.5) node   [align=left] {\begin{minipage}[lt]{13.06pt}\setlength\topsep{0pt}
$x_1$
\end{minipage}};
\draw (623.2,184.3) node   [align=left] {\begin{minipage}[lt]{13.06pt}\setlength\topsep{0pt}
$x_2$
\end{minipage}};
\draw (173.63,105.13) node   [align=left] {\begin{minipage}[lt]{8.67pt}\setlength\topsep{0pt}
$t$
\end{minipage}};
\draw (454.13,105.63) node   [align=left] {\begin{minipage}[lt]{8.67pt}\setlength\topsep{0pt}
$t$
\end{minipage}};
\draw (149.13,134.88) node   [align=left] {\begin{minipage}[lt]{9.69pt}\setlength\topsep{0pt}
$2$
\end{minipage}};
\draw (430.63,134.38) node   [align=left] {\begin{minipage}[lt]{9.69pt}\setlength\topsep{0pt}
$2$
\end{minipage}};
\draw (262.88,301.13) node   [align=left] {\begin{minipage}[lt]{9.35pt}\setlength\topsep{0pt}
\end{minipage}};
\draw (225.38,224.13) node   [align=left] {\begin{minipage}[lt]{9.35pt}\setlength\topsep{0pt}
$2$
\end{minipage}};
\draw (489.88,274.63) node   [align=left] {\begin{minipage}[lt]{9.35pt}\setlength\topsep{0pt}
$2$
\end{minipage}};
\draw (543.38,301.63) node   [align=left] {\begin{minipage}[lt]{9.35pt}\setlength\topsep{0pt}
\end{minipage}};
\draw (276.13,130.63) node   [align=left] {\begin{minipage}[lt]{10.37pt}\setlength\topsep{0pt}
$u (x)$
\end{minipage}};
\draw (531.13,130.63) node   [align=left] {\begin{minipage}[lt]{10.37pt}\setlength\topsep{0pt}
$v_{\mu} (x)$
\end{minipage}};
\end{tikzpicture}
\caption{The function $u$ given in Example~\ref{no rigidity values of alpha 2} and its circular rearrangement $v_{\mu}$. 
In this case rigidity holds, since every extremal is obtained as the composition of $v_\mu$ with a rotation.
Note that now the set $\{ 0 < \alpha_{\mu} < \pi \}$ is essentially connected, see Figure~\ref{figure distribution alpha 2}.}
       \label{figure no rigidity 2}
\end{figure}
\end{center}
Then, 
\[
\int_{\Omega} | \nabla u (x)|^2 \, dx = \int_{\Omega} | \nabla v_{\mu} (x)|^2 \, dx 
= 8 \pi,
\]
so that $u \in \mathcal{E} (\mu, \Omega)$.
However, rigidity is not violated, since $u (x) = v_{\mu} (R x)$ for every $x \in \Omega$,
where $R$ is the clockwise rotation of $\pi/2$ around the origin. 
Note that in this case the set $\{ 0 < \alpha_{\mu} < \pi \}$ is connected, see Figure~\ref{figure distribution alpha 2}.
\begin{center}
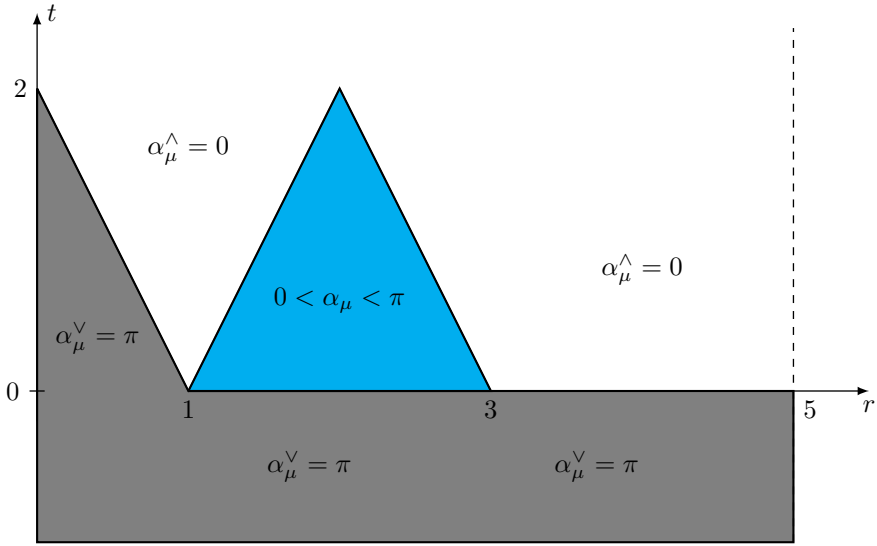
\begin{figure}[!htb]
  \begin{tikzpicture}[>=latex, scale=2]

 \draw[fill=gray, style=thick](0,2)--(1,0)--(5, 0)
--(5, -1)--(0,-1)--(0,2);

 \draw[fill=cyan, style=thick](1,0)--(2, 2)--(3, 0)--(1,0);

   \draw[line width = .5pt] (-.05,0)--(.05,0); 
  \draw(-.05,0)node[left]{$0$};
  \draw[->, line width = .5pt] (5,0)--(5.5,0); 
  \draw(5.5,0)node[below]{$r$};
 \draw[->, line width = .5pt] (0,-.2)--(0,2.5); 
  \draw(0,2.5)node[right]{$t$};
  \draw(0,2)node[left]{$2$};

  \draw(1,0)node[below]{$1$};
  \draw(3,0)node[below]{$3$};
  \draw(5,0)node[below right]{$5$};

  \draw[line width = .5pt] (0,2)--(1,0)--(2, 2)--(3, 0);
   \draw[dashed, line width = .5pt] (1,0)--(5,0); 

  \draw(.4,.35)node{$\alpha^{\vee}_{\mu} = \pi$};
  \draw(1.8,-.5)node{$\alpha^{\vee}_{\mu} = \pi$};
  \draw(3.7,-.5)node{$\alpha^{\vee}_{\mu} = \pi$};

  \draw(1,1.6)node{$\alpha^{\wedge}_{\mu} = 0$};
  \draw(4,0.8)node{$\alpha^{\wedge}_{\mu} = 0$};

  \draw(2,.6)node{$ 0 < \alpha_{\mu} < \pi$};

   \draw[dashed, line width = .5pt] (5,-1)--(5, 2.4);

\end{tikzpicture}
\caption{The values of the function $\alpha_{\mu} (r, t)$ given  
in Example~\ref{no rigidity values of alpha 2}.
In this case, the set $\{ 0 < \alpha_{\mu} < \pi \}$ is essentially connected.}
       \label{figure distribution alpha 2}
\end{figure}
\end{center}
\end{example}

\medskip

The next example shows that even assuming $f \in \mathscr{F}'$ and $\{ 0 < \alpha_{\mu} < \pi \}$ (essentially) connected
might not be enough to guarantee rigidity. This is because also sets where the singular part of $D \alpha_\mu$ is concentrated play an important role.

\begin{example} \label{ex no rig 3}
Let $n=2$, let $\Omega=D(5)$, $0 < \delta \ll 1$, $\tilde{x} = (2 + \delta,0) \in \R^2$, and set:
\[
v_\mu (x) = \max \{ v_1 (x), v_2 (x), v_3 (x)\}, 
\]
where 
\[
v_1 (x) = 2 \max \{ 0, 1 - | x  | \}, 
\quad
v_2 (x) =  \frac{3}{4} \max \{ 0, 2 - | x  | \} \varphi  (\hat{x} \cdot e_1) 
\quad 
v_3 (x) = 2 \max \{ 0,  1 - | x - \tilde{x} | \},
\]
and where $\varphi \in C^\infty_c ((-1, 1])$ 
is a non decreasing function with $0 \leq \varphi \leq 1$ and $\varphi \equiv 1$ in $[0, 1]$,
see Figure~\ref{figure alpha 53}. 

Let now $\gamma \in (0, \pi/4)$, and let $R : \R^2 \to \R^2$ 
be a counterclockwise rotation of an angle $\gamma$.
Then, setting (see Figure~\ref{figure alpha 53})
\[
u (x) = \max \{ v_1 (x), v_2 (x), v_3 (R x)\}, 
\]
one can check that 
\[
\int_{\Omega} |\nabla u|^2 \, dx  = \int_{\Omega} |\nabla v_\mu|^2 \, dx, 
\]
so that rigidity is violated.
 Note that in this case the singular set $S_{\alpha_\mu}$
 essentially disconnects $\{ 0 < \alpha_{\mu} < \pi \}$, see Figure~\ref{figure alpha 23}.
\begin{center}
\begin{figure}[!htb]
\begin{tikzpicture}[x=0.70pt,y=0.70pt,yscale=-1,xscale=1]

\draw  [dash pattern={on 4.5pt off 4.5pt}] (399.81,206.85) .. controls (399.79,195.47) and (415.45,186.23) .. (434.78,186.2) .. controls (454.11,186.17) and (469.79,195.37) .. (469.81,206.75) .. controls (469.82,218.13) and (454.16,227.37) .. (434.83,227.4) .. controls (415.5,227.43) and (399.82,218.23) .. (399.81,206.85) -- cycle ;
\draw  [dash pattern={on 4.5pt off 4.5pt}]  (434.81,100.82) -- (434.81,208.32) ;
\draw    (434.81,100.82) -- (399.55,204.68) ;
\draw  [dash pattern={on 4.5pt off 4.5pt}]  (451.01,150) -- (469.81,208.32) ;
\draw    (434.81,100.82) -- (435.04,65.75) ;
\draw [shift={(435.05,63.75)}, rotate = 90.37] [color={rgb, 255:red, 0; green, 0; blue, 0 }  ][line width=0.75]    (10.93,-3.29) .. controls (6.95,-1.4) and (3.31,-0.3) .. (0,0) .. controls (3.31,0.3) and (6.95,1.4) .. (10.93,3.29)   ;
\draw    (344,203) -- (583,204.49) ;
\draw [shift={(585,204.5)}, rotate = 180.36] [color={rgb, 255:red, 0; green, 0; blue, 0 }  ][line width=0.75]    (10.93,-3.29) .. controls (6.95,-1.4) and (3.31,-0.3) .. (0,0) .. controls (3.31,0.3) and (6.95,1.4) .. (10.93,3.29)   ;
\draw    (434.81,100.82) -- (451.01,150) ;
\draw  [dash pattern={on 4.5pt off 4.5pt}] (357.74,206.91) .. controls (357.72,187.86) and (392.2,172.37) .. (434.76,172.32) .. controls (477.32,172.26) and (511.84,187.65) .. (511.87,206.69) .. controls (511.9,225.74) and (477.41,241.22) .. (434.85,241.28) .. controls (392.29,241.34) and (357.77,225.95) .. (357.74,206.91) -- cycle ;
\draw    (443,158) .. controls (434,170) and (426,203) .. (405,216) ;
\draw  [dash pattern={on 4.5pt off 4.5pt}]  (415,190) .. controls (428,171) and (424.16,153.98) .. (442,146) ;
\draw    (372,226) .. controls (408,252) and (487,241) .. (503,222.5) ;
\draw    (450,147) -- (496.87,185.69) ;
\draw  [dash pattern={on 4.5pt off 4.5pt}] (487.02,203.96) .. controls (487,192.58) and (500.83,183.34) .. (517.9,183.32) .. controls (534.97,183.29) and (548.82,192.5) .. (548.84,203.88) .. controls (548.85,215.25) and (535.03,224.49) .. (517.96,224.52) .. controls (500.89,224.54) and (487.04,215.34) .. (487.02,203.96) -- cycle ;
\draw  [dash pattern={on 4.5pt off 4.5pt}]  (493,182) -- (487.02,203.96) ;
\draw    (522.28,100.11) -- (548.84,203.88) ;
\draw    (522.28,100.11) -- (493,182) ;
\draw    (503,222.5) .. controls (534,229.5) and (547,215.5) .. (548.84,203.88) ;
\draw    (443,158) .. controls (456,156.5) and (450,146.5) .. (451.01,150) ;
\draw    (405,216) -- (372,226) ;
\draw  [dash pattern={on 4.5pt off 4.5pt}]  (415,190) -- (384,180) ;
\draw    (447,157) -- (474,236) ;
\draw  [dash pattern={on 4.5pt off 4.5pt}] (152.81,205.85) .. controls (152.79,194.47) and (168.45,185.23) .. (187.78,185.2) .. controls (207.11,185.17) and (222.79,194.37) .. (222.81,205.75) .. controls (222.82,217.13) and (207.16,226.37) .. (187.83,226.4) .. controls (168.5,226.43) and (152.82,217.23) .. (152.81,205.85) -- cycle ;
\draw  [dash pattern={on 4.5pt off 4.5pt}]  (187.81,99.82) -- (187.81,207.32) ;
\draw    (187.81,99.82) -- (152.55,203.68) ;
\draw  [dash pattern={on 4.5pt off 4.5pt}]  (204.01,149) -- (222.81,207.32) ;
\draw    (187.81,99.82) -- (188.04,64.75) ;
\draw [shift={(188.05,62.75)}, rotate = 90.37] [color={rgb, 255:red, 0; green, 0; blue, 0 }  ][line width=0.75]    (10.93,-3.29) .. controls (6.95,-1.4) and (3.31,-0.3) .. (0,0) .. controls (3.31,0.3) and (6.95,1.4) .. (10.93,3.29)   ;
\draw    (97,202) -- (300,202.99) ;
\draw [shift={(302,203)}, rotate = 180.28] [color={rgb, 255:red, 0; green, 0; blue, 0 }  ][line width=0.75]    (10.93,-3.29) .. controls (6.95,-1.4) and (3.31,-0.3) .. (0,0) .. controls (3.31,0.3) and (6.95,1.4) .. (10.93,3.29)   ;
\draw    (187.81,99.82) -- (204.01,149) ;
\draw  [dash pattern={on 4.5pt off 4.5pt}] (110.74,205.91) .. controls (110.72,186.86) and (145.2,171.37) .. (187.76,171.32) .. controls (230.32,171.26) and (264.84,186.65) .. (264.87,205.69) .. controls (264.9,224.74) and (230.41,240.22) .. (187.85,240.28) .. controls (145.29,240.34) and (110.77,224.95) .. (110.74,205.91) -- cycle ;
\draw    (196,157) .. controls (187,169) and (179,202) .. (158,215) ;
\draw  [dash pattern={on 4.5pt off 4.5pt}]  (168,189) .. controls (181,170) and (177.16,152.98) .. (195,145) ;
\draw    (125,225) .. controls (148,242) and (191,241) .. (203.02,239.96) ;
\draw    (203,146) -- (249.87,184.69) ;
\draw  [dash pattern={on 4.5pt off 4.5pt}] (203.02,241.96) .. controls (203,230.58) and (216.83,221.34) .. (233.9,221.32) .. controls (250.97,221.29) and (264.82,230.5) .. (264.84,241.88) .. controls (264.85,253.25) and (251.03,262.49) .. (233.96,262.52) .. controls (216.89,262.54) and (203.04,253.34) .. (203.02,241.96) -- cycle ;
\draw    (238.28,138.11) -- (264.84,241.88) ;
\draw    (238.28,138.11) -- (203.02,241.96) ;
\draw    (196,157) .. controls (209,155.5) and (203,145.5) .. (204.01,149) ;
\draw    (158,215) -- (125,225) ;
\draw  [dash pattern={on 4.5pt off 4.5pt}]  (168,189) -- (137,179) ;
\draw    (203.02,241.96) .. controls (212,277) and (269,261) .. (264.84,241.88) ;
\draw    (259,220) .. controls (272,208) and (261,190) .. (249.87,184.69) ;

\draw (583.8,220.5) node   [align=left] {\begin{minipage}[lt]{13.06pt}\setlength\topsep{0pt}
$x_1$
\end{minipage}};
\draw (447.93,70.63) node   [align=left] {\begin{minipage}[lt]{8.67pt}\setlength\topsep{0pt}
$t$
\end{minipage}};
\draw (500.93,75.63) node   [align=left] {\begin{minipage}[lt]{10.37pt}\setlength\topsep{0pt}
$v_\mu (x)$
\end{minipage}};
\draw (298.8,217.5) node   [align=left] {\begin{minipage}[lt]{13.06pt}\setlength\topsep{0pt}
$x_1$
\end{minipage}};
\draw (200.93,69.63) node   [align=left] {\begin{minipage}[lt]{8.67pt}\setlength\topsep{0pt}
$t$
\end{minipage}};
\draw (248.93,78.63) node   [align=left] {\begin{minipage}[lt]{10.37pt}\setlength\topsep{0pt}
$u (x)$
\end{minipage}};
\end{tikzpicture}
\caption{The function $u$ given in Example~\ref{ex no rig 3} and its circular rearrangement $v_{\mu}$, 
showing that rigidity fails.
 In this case, the singular set $S_{\alpha_\mu}$ essentially disconnects
 $\{ 0 < \alpha_{\mu} < \pi \}$, see Figure~\ref{figure alpha 23}.
 }  
  \label{figure alpha 53}
\end{figure}
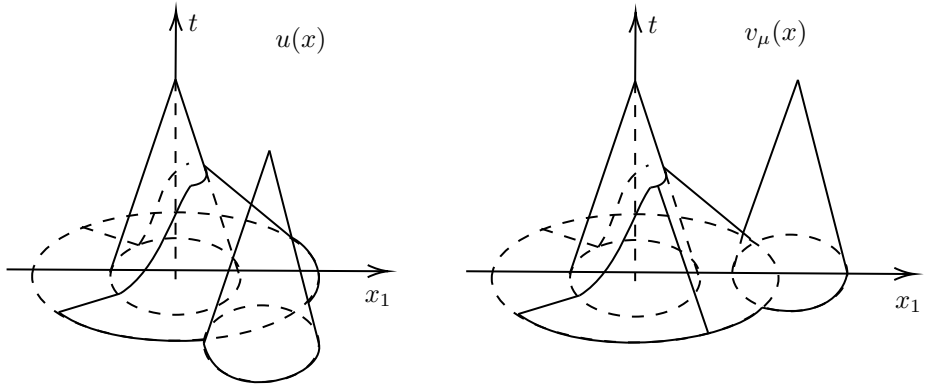
\end{center}

\begin{center}
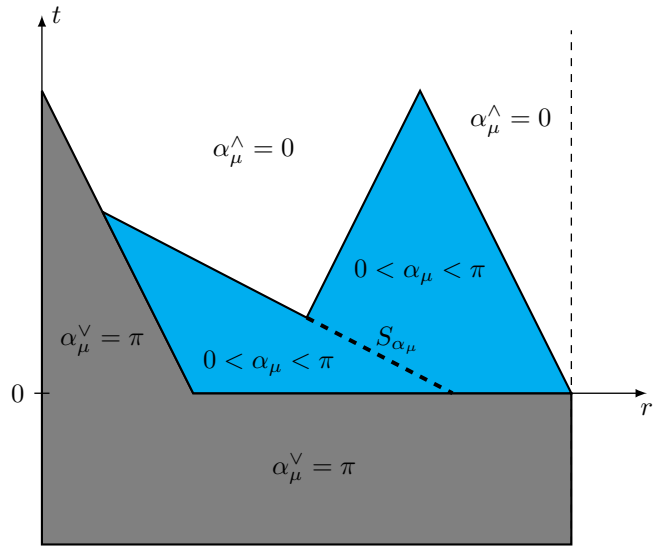
\begin{figure}[!htb]
  \begin{tikzpicture}[>=latex, scale=2]

 \draw[fill=gray, style=thick](0,2)--(1,0)--(7/2, 0)
--(7/2, -1)--(0,-1)--(0,2);

 \draw[fill=cyan, style=thick](2/5,6/5)--(7/4,1/2)--(5/2, 2)--(7/2, 0)--(1, 0)--(2/5,6/5);

\draw[line width = 1.5pt, dashed] (7/4,1/2)--(19/7, 0); 
 \draw(7/4+.1+.5,1/4+0.1)node{$S_{\alpha_\mu}$};

   \draw[line width = .5pt] (-.05,0)--(.05,0); 
  \draw(-.05,0)node[left]{$0$};
  \draw[->, line width = .5pt] (7/2,0)--(7/2+ .5,0); 
  \draw(7/2+ .5,0)node[below]{$r$};
 \draw[->, line width = .5pt] (0,-.2)--(0,2.5); 
  \draw(0,2.5)node[right]{$t$};

%
%


  \draw(.4,.35)node{$\alpha^{\vee}_{\mu} = \pi$};
  \draw(1.8,-.5)node{$\alpha^{\vee}_{\mu} = \pi$};

  \draw(1+.4,1.6)node{$\alpha^{\wedge}_{\mu} = 0$};
  \draw(3+.1,1.8)node{$\alpha^{\wedge}_{\mu} = 0$};

  \draw(2+.5,.8)node{$ 0 < \alpha_{\mu} < \pi$};
  \draw(1.5,.2)node{$ 0 < \alpha_{\mu} < \pi$};

   \draw[dashed, line width = .5pt] (7/2,-1)--(7/2, 2.4);

\end{tikzpicture}
\caption{The values of the function $\alpha_{\mu} (r, t)$ given  
in Example~\ref{ex no rig 3}. The singular set $S_{\alpha_\mu}$ (see bold dashed line)
essentially disconnects $\{ 0 < \alpha_{\mu} < \pi \}$.}
       \label{figure alpha 23}
\end{figure}
\end{center}
\end{example}

\bigskip

The rest of the paper is divided as follows. In Section~\ref{section_preliminaries} we introduce some notions and we recall some basic results of Geometric Measure Theory, while in Section~\ref{subsection_circular in Rn+1} we discuss the circular symmetrization of subgraphs. 
In Section~\ref{sect polya-szego} we prove a general version of the P\'olya--Szeg\"o inequality (see Theorem~\ref{the thing}), 
and we show that this implies Theorem~\ref{thm_Polya-Szego inequality local} and Theorem~\ref{thm_Polya-Szego inequality}.
Finally, Section~\ref{sect rigidity} contains the proof of Theorem~\ref{thm suff conditions}.

\section{Preliminaries and notation} \label{section_preliminaries}

In this section we introduce the notation and the background needed to prove the main results of the paper.
For more details we direct the reader to the monographs 
\cite{AFP, GMSbook1, maggiBOOK, Simon83}. 

\subsection{Basic notation}
Let $N \in \mathbb{N}$ with $N \geq 2$. We decompose $\mathbb{R}^{N}$ as $\mathbb{R}^2 \times\mathbb{R}^{N-2}$, 
and we set $\R^2_0 = \R^2 \setminus {(0, 0)}$. 
Depending on the context, $| \cdot |$ stands for the Euclidean norm in $\mathbb{R}$, $\mathbb{R}^2$, or $\mathbb{R}^N$.
Moreover, for every $x \in \R^2_0$ we set $\hat{x}:= x/|x|$.
If $r > 0$ and $w= (x,z) \in \mathbb{R}^{N}$ with $x\in\mathbb{R}^2$ and $z \in\mathbb{R}^{N-2}$, 
we will denote the open ball of $\mathbb{R}^{N}$ of radius $r$ and centre $w$ by $B_r^N (w)$, or simply $B_r^N$ when $w = 0$.
For every $r > 0$, we set $D(r) = \{ x \in \R^2 :  |x| < r \}$
and $\partial D(r) = \{ x \in \R^2_0 :  |x| = r \}$, while for the unit circle of $\R^2$ we use the standard notation 
$\mathbb{S}^1 = \{ x \in \R^2_0 :  |x| = 1 \}$.
In $\mathbb{S}^1$, we consider the topology induced by the arclength distance $d_{\mathbb{S}^1}: \mathbb{S}^1 \times \mathbb{S}^1 \to [ 0, \pi]$, given by
\[
d_{\mathbb{S}^1} (\omega' , \omega'') := \arccos (\omega' \cdot  \omega''), 
\quad \text{ for every } \omega' , \omega'' \in \mathbb{S}^1.
\] 
This topology coincides with the topology induced in $\mathbb{S}^1$ by the Euclidean topology of $\R^2$.
If $R \subset \mathbb{S}^1$, we set
\[
d_{\mathbb{S}^1} (\omega , R)
:= \inf \{ d_{\mathbb{S}^1} (\omega , \sigma ) : \, \, \sigma \in R \}, 
\quad \text{ for every } \omega \in  \mathbb{S}^1.
\]
In particular, with the usual convention $\inf \emptyset = + \infty$, we have that $d_{\mathbb{S}^1} (\omega , \emptyset) = + \infty$.
Let now $\beta \in (0, \pi)$ and let $x \in \mathbb{S}^1$. We denote by $\mathbf{B}_{\beta} (x)$ be the (relatively) open connected arc of $\mathbb{S}^1$ centred at $x$ and with length $2 \beta$, i.e.:
\[
\mathbf{B}_{\beta} (x):= \{ \sigma \in \mathbb{S}^{1} : d_{\mathbb{S}^1} (\sigma , x) < \beta \}.
\]
For every $k \in \mathbb{N}$ with $1 \leq k \leq N$, we denote the
Hausdorff $k$-dimensional measure in $\R^N$ by $\mathcal{H}^k$, while $\mathcal{L}^N$ stands for the $N$-dimensional Lebesgue measure. 
If $A, B \subset \R^N$, we write $A \subset_{\mathcal{H}^k} B$ when $\mathcal{H}^k (A \setminus B) = 0$ and 
$A =_{\mathcal{H}^k} B$ when $\mathcal{H}^k (A \Delta B) = 0$, where $A \Delta B = (A \setminus B) \cup (B \setminus A)$ is the symmetric difference of $A$ and $B$.

Let $E \subset \R^N$ be a Lebesgue measurable set. We denote by $\chi_E$ its characteristic function. Moreover, 
the upper and lower $N$-dimensional densities of $E$ at $w$ are defined as
\begin{eqnarray*}
  \theta^*(E,w) :=\limsup_{\rho\to 0^+}\frac{\mathcal{H}^N (E\cap B_\rho^N(w))}{\omega_N \,\rho^N}\,,
  \qquad
  \theta_*(E,w) :=\liminf_{\rho\to 0^+}\frac{\mathcal{H}^N(E\cap B_\rho^N(w))}{\omega_N\,\rho^N}\,,
\end{eqnarray*}
respectively, where $\omega_N$ is the $N$-dimensional Lebesgue measure of the unit ball of $\R^N$.
The functions $w \mapsto  \theta^*(E,w)$ and $w \mapsto  \theta_*(E,w)$
are Borel measurable, and they agree $\mathcal{L}^N$-a.e. on $\R^N$. 

Therefore, the $N$-dimensional density of $E$ at $w$
\[
\theta(E,w) := \lim_{\rho\to 0^+}\frac{\mathcal{H}^N(E\cap B_\rho^N(w))}{\omega_N\,\rho^N}\,,
\]
is defined for $\mathcal{L}^N$-a.e. $w\in\R^N$, and $w \mapsto  \theta (E,w)$
is a Borel function on $\R^N$.
Given $t \in [0,1]$, we define the set of points of density $t$ of $E$ as
$$
E^{(t)} :=\{w\in\R^N :\theta(E,w)=t\}.
$$
The set $\partial^{\mathrm{e}} E :=\R^N \setminus(E^{(0)}\cup E^{(1)})$ is called the \textit{essential boundary} of $E$.

Let $f:\mathbb{R}^N \to \mathbb{R}$ be a Lebesgue measurable function. If $M \in \R$, we set 
\[
(f \vee M) (w) = \max\{ f (w), M \} \quad \text{ and } \quad
(f \wedge M) (w) = \min\{ f (w), M \}.
\]
When $M = 0$, we will also use the notation $f_+ = f \vee 0$ 
and $f_- = (-f) \vee 0$. 
We define the {\it approximate upper limit} $f\Upp(w)$ and the {\it approximate lower limit} 
$f\Low(w)$ of $f$ at $w\in\mathbb{R}^N$ as 
\begin{eqnarray}
  \label{def fvee}
  f\Upp(w)=\inf\Big\{t\in\mathbb{R}: w\in \{f>t\}^{(0)}\Big\}\,,
  \\
  \label{def fwedge}
  f\Low (w)=\sup\Big\{t\in\mathbb{R}: w\in \{f>t\}^{(1)}\Big\}\,.
\end{eqnarray}
We observe that $f\Upp$ and $f\Low$ are Borel functions that are defined at 
{\it every} point of $\mathbb{R}^N$, with values in $\mathbb{R}\cup\{\pm\infty\}$.
Moreover, if $f_1: \mathbb{R}^N \to \mathbb{R}$ and $f_2: \mathbb{R}^N \to \mathbb{R}$
are measurable functions satisfying $f_1=f_2$ $\mathcal{L}^N$-a.e. on $\mathbb{R}^N$, 
then $f_1\Upp=f_2\Upp$ and $f_1\Low=f_2\Low$ {\it everywhere} on $\mathbb{R}^N$. 
We define the {\it approximate discontinuity} set $S_f$ of $f$ as $S_f: =\{f\Low<f\Upp\}$.
\noindent
Note that, by the above considerations, it follows that $\mathcal{L}^N(S_f)=0$. 
Although $f^\wedge$ and $f^\vee$ may take infinite values on $S_f$, 
the difference $f^\vee(z)-f^\wedge(z)$ is well defined in $\mathbb{R}\cup\{\pm\infty\}$ for every $w\in S_f$.
Then, we can define the {\it approximate jump} $[f]$ of $f$ as the Borel function $[f]:\mathbb{R}^N\to[0,\infty]$ given by
  \begin{eqnarray*}
    [f](w):=\left\{\begin{array}{l l}
      f\Upp(w)-f\Low(w)\,&\mbox{if $w\in S_f$}\,,
      \vspace{.2cm} \\
      0\,&\mbox{if $w\in \mathbb{R}^N\setminus S_f$}\,.
    \end{array}
    \right .
  \end{eqnarray*}
For $w\in\mathbb{R}^N$ and $\nu\in \mathbb{S}^{N-1}$, we will denote by $H_{w, \nu}^+$ and $H_{w,\nu}^-$ the closed half-spaces
whose boundaries are orthogonal to $\nu$:
\begin{align*}
H_{w,\nu}^+:=\Big\{ w' \in\R^N:\,(w'- w)\cdot\nu\ge 0\Big\},\quad
H_{w,\nu}^-:=\Big\{w'\in\R^N:\,(w'-w)\cdot\nu\le 0\Big\}.
\end{align*}
Let $A\subset\mathbb{R}^N$ be a Lebesgue measurable set. 
We say that $t\in\R\cup\{\pm\infty\}$ is the approximate limit of $f$ at $w$ with respect to $A$, 
and write $t=\aplim (f,A,w)$, if
\begin{eqnarray}
  &&\theta\Big(\{|f-t|>\varepsilon\}\cap A;w\Big)=0\,,\qquad\forall\varepsilon>0\,,\hspace{0.3cm}\qquad (t\in\R)\,,
  \\
  &&\theta\Big(\{f<M\}\cap A;w\Big)=0\,,\qquad\hspace{0.6cm}\forall M>0\,,\qquad (t=+\infty)\,,
  \\
  &&\theta\Big(\{f>-M\}\cap A;w\Big)=0\,,\qquad\hspace{0.3cm}\forall M>0\,,\qquad (t=-\infty)\,.
\end{eqnarray}
We say that $w\in S_f$ is a \textit{jump point} of $f$ if there exists $\nu\in \mathbb{S}^{N-1}$ such that
\[
f^\vee(w)=\aplim(f,H_{w,\nu}^+,x)> f^\wedge(w)=\aplim(f,H_{w,\nu}^-,x)\,.
\]
If this is the case, we say that $\nu_f(w):= \nu$ is the approximate jump direction of $f$ at $w$.
Denoting by $J_f$ the set of approximate jump points of $f$, we have that $J_f\subset S_f$ 
and $\nu_f:J_f\to \mathbb{S}^{N-1}$ is a Borel function.  

Let $w \in S_f^c$. We say that $f$ is \emph{approximately differentiable} at $w$ if $\tilde{f}(w):=f\Low(w)=f\Upp(w)\in \mathbb{R}$
and there exists $\xi \in \mathbb{R}^N$ such that
\begin{align*}
\aplim(g_\xi,\mathbb{R}^N,w)=0,
\end{align*} 
where $g_\xi (w')= (f(w')-\tilde{f}(w)-\xi\cdot(w'-w))/|w'-w|$ for $w' \in \mathbb{R}^N \setminus\{w \}$.
If this is the case, then $\xi$ is uniquely determined, we set $\xi=\nabla f(w)$, and call $\nabla f(w)$ the \emph{approximate differential} of $f$ at $w$. In components, we write $\nabla f(w) = (\nabla_x f (w), \nabla_z f (w))$, where $\nabla_x f (w) =  (\partial_{x_1} f (w), \partial_{x_2} f (w)) \in \R^2$ and 
$\nabla_z f (w) =  (\partial_{z_1} f (w), \ldots, \partial_{z_{N-2}} f (w))\in \mathbb{R}^{N-2}$. Moreover, we denote the set of points of approximate differentiability of $f$ as 
$$
\mathcal{D}_f:=\left\{w\in\mathbb{R}^N:\, f \textnormal{ is approximately differentiable at }w   \right\}.
$$

\subsection{Essential connectedness}
We now introduce the concept of essential connectedness, which will be used to study the rigidity of the P\'olya--Szeg\"o inequality.
\begin{definition}
Let $G \subset \R^N$ be a Borel set, and let $G_+,G_- \subset \R^N$ be Borel sets. We say that 
$\{G_+,G_-\}$ is a non-trivial Borel partition of $G$ modulo $\H^N$ if 
 \[
  \H^N (G_+\cap G_-)=0\,,\qquad  \H^N (G\Delta(G_+\cup G_-))=0\,,\qquad \H^N (G_+)\,\H^N(G_-)>0\,.
  \]
\end{definition}

\begin{definition}\label{definition ess con}
Let $K$ and $G$ be Borel sets in $\R^N$. We say that $K$ essentially disconnects $G$ if there exists a non-trivial Borel partition $\{G_+,G_-\}$ of $G$ modulo $\H^N$ such that
\begin{equation}
  \label{K disconnects G}
  \H^{N-1}\Big(\Big(G^{(1)}\cap\pae  G_+\cap\pae G_-\Big)\setminus K\Big)=0\,.
\end{equation}
Instead, we say that $K$ does not essentially disconnect $G$ if, for every non-trivial Borel partition $\{G_+,G_-\}$ of $G$ modulo $\H^N$,
\begin{equation}
  \label{K does not disconnect G}
  \H^{N-1}\Big(\Big(G^{(1)}\cap\pae  G_+\cap\pae G_-\Big)\setminus K \Big)>0\,.
\end{equation}
Finally, we say that $G$ is essentially connected if $\emptyset$ does not essentially disconnect $G$. 
\end{definition}
%
%
%
\begin{remark}\label{remark essential connected}
  If $\H^N(G\Delta G')=0$ and $\H^{N-1}(K\Delta K')=0$, then $K$ essentially disconnects $G$ if and only if $K'$ essentially disconnects $G'$.
\end{remark}

\subsection{$BV$ functions}
Let $\Omega\subset \mathbb{R}^N$ be open. We denote by $C^0 (\Omega)$ ($C^0_c (\Omega)$)
the space of continuous functions (with compact support) in $\Omega$, while 
$C^0_b (\Omega)$ stands for the space of bounded continuous functions in $\Omega$.
Moreover, $C^1_c(\Omega;\R^N)$ is the space of $\R^N$-valued $C^1$ functions with compact support in $\Omega$.
Let $f\in L^1(\Omega)$. We say that 
$f$ is a function of bounded variation in $\Omega$, and we write $f\in BV(\Omega)$, if
\begin{align}\label{def: total variation of Df, con f in BV}
\sup\left\{\int_\Omega\,f(w)\,\diver\,\varphi(w)\,dw:\,\varphi\in C^1_c(\Omega;\R^N)\,,|\varphi|\le 1\right\}<\infty.
\end{align}
More in general, we say that $f\in BV_{\textnormal{loc}}(\Omega)$ if $f\in BV(\Omega')$ for every open set $\Omega' \subset \subset \Omega$. 
If $f\in BV_{\textnormal{loc}} (\Omega)$, then the distributional derivative $D f$ of $f$ is an $\R^N$-valued Radon measure. 
In this case, when needed we will write $D f = (D_x f, D_z f)$, where $D_x f$ is an $\R^2$-valued Radon measure, and $D_z f$ is an 
 $\R^{N-2}$-valued Radon measure.

Let now $f\in BV(\Omega)$. Then, the total variation of $Df$ is finite in $\Omega$, and its value $|Df|(\Omega)$ coincides with the left--hand side of \eqref{def: total variation of Df, con f in BV}. One can write the  Radon--Nikod\'ym decomposition of $Df$ with respect to $\mathcal{L}^{N}$
as $Df=D^af+D^sf$, where $D^af \ll\mathcal{L}^{N}$, and where $D^s f$ and $D^a f$ are mutually singular. 

It turns out that for $\mathcal{L}^{N}$-a.e. $w\in \Omega$ the function $f$ is approximately differentiable, and its 
approximate differential is the density of $D^a f$ with respect to $\mathcal{L}^{N}$. Thus, we have $D^a f = \nabla f d \mathcal{L}^{N}$
and $\nabla f \in L^1 (\Omega;\R^N)$. 
Moreover, the singular set $S_f$ and the jump set $J_f$ of $f$ are countably $(N-1)$-rectifiable, with $\mathcal{H}^{N-1} (S_f \setminus J_f) = 0$. In addition, for $\mathcal{H}^{N-1}$-a.e. $w \in J_f$ there exists a vector $\nu_f (w) \in \mathbb{S}^{N-1}$ such that we can further decompose 
the singular part of $Df$ as $D^s f= D^j f+ D^c f$, where $D^j f$ and $D^c f$ are mutually singular, 
$D^j f = [f] (w) \nu_f (w)\, d\mathcal{H}^{N-1}\mres J_f$ is called the jump part of $Df$, and 
$D^cf = D^s f - D^j f$ is called the Cantor part of $Df$, and is concentrated on a set $K \subset \Omega$ such that 
$\mathcal{H}^{N} (K) = 0$ and $\mathcal{H}^{N-1} (K) = + \infty$. 
If $f:\Omega\to \mathbb{R}^m$, with $m\in\mathbb{N}$ and $m\geq 2$, we say that $f \in BV(\Omega;\mathbb{R}^m)$ if and only if $f_i\in BV(\Omega)$ for $i=1,\dots,m$, where with $f=(f_1,\dots,f_m)$. Accordingly, we say that $f\in BV_{\textnormal{loc}}(\Omega;\mathbb{R}^m)$ if $f\in BV(\Omega';\mathbb{R}^m)$ for every open set $\Omega'$ compactly contained in $\Omega$. If $f\in BV_{\textnormal{loc}}(\Omega;\mathbb{R}^m)$ the distributional derivative $Df$ of $f$ is an $(m\times N)$-valued Radon measure.

If $a\in \mathbb{R}^m$, and $b\in\mathbb{R}^N$ we denote the tensor product between $a$ and $b$ as $a \otimes b$. This is the $m\times N$ matrix whose components are given by 
$$
(a\otimes b)_{i,j}= a_i b_j \quad i=1,\dots,m,\; j=1,\dots,N.
$$

\subsection{Sets of finite perimeter}\label{section sofp} 
\noindent
Let $E\subset\R^N$ be a Lebesgue measurable set, and let  $O\subset \mathbb{R}^N$ be open. 
We say that  $E$ is a set of finite perimeter in $O$ if 
\begin{align}\label{def: locally finite perimeter of a set in R^k}
P(E;O):= \sup\left\{\int_{E} \textnormal{div}\, \varphi (w) \,dw:\, \varphi \in C^1_c(O;\mathbb{R}^N), \, |\varphi| \leq 1 \right\}<\infty,
\end{align}
and in this case we say that $P(E;O)$ is the perimeter of $E$ in $O$. Note that, if $\mathcal{L}^N(E)<\infty$, then $E$ is a set of finite perimeter in $O$ if and only if $\chi_E\in BV (O)$. In the special case $O = \R^N$ we set $P(E):=P(E;\mathbb{R}^N)$ and when $P(E) < \infty$ we say that $E$ is a set of finite perimeter. We say that $E$ is a set of locally finite perimeter in $O$ if $\chi_E\in BV_{\textnormal{loc}}(O)$.

Let $E$ be a set of locally finite perimeter in $O$. We define the \emph{reduced boundary} $\partial^*E$ of $E$ as the set of those points $w \in O$ 
such that
\begin{align*}
\nu^E(w):=\lim_{\rho\to 0^+}\frac{D\chi_E(B^N_\rho(w))}{|D\chi_E|(B^N_\rho(w))},
\end{align*}
exists and belongs to $\mathbb{S}^{N-1}$. We call $\nu^E:\partial^*E\to \mathbb{S}^{N-1}$ the {\it measure-theoretic inner unit normal} to $E$.
It turns out that $\nu^E$ is a Borel function, and that the distributional derivative $D \chi_E$ of $\chi_E$ is given by
\[
D \chi_E = \nu^E d \mathcal{H}^{N-1} \mres \partial^*E. 
\]
Therefore, if $A \subset O$ is a Borel set, we can define the perimeter of $E$ in $A$ as
$$
P(E;A):=|D\chi_E|(A)=\mathcal{H}^{N-1}(\partial^*E \cap A).
$$
It turns out that 
\begin{equation*}
  \label{inclusioni frontiere}
  (\partial^*E\cap O)  \subset (E^{(1/2)}\cap O) \subset (\partial^{\mathrm{e}} E\cap O)\,.
\end{equation*}
Moreover, {\it Federer's theorem} holds true (see \cite[Theorem 3.61]{AFP} and \cite[Theorem 16.2]{maggiBOOK}):
\[
\mathcal{H}^{N-1}((\partial^{\mathrm{e}} E\cap O)\setminus(\partial^*E\cap O))=0.
\]

\subsection{Sets of finite perimeter in $\mathbb{S}^{1}$}
We now briefly introduce sets of finite perimeter in the unit circle $\mathbb{S}^{1}$.
Then, if $r > 0$ one can argue in a similar way to define sets of finite perimeter on $\partial D (r)$.
We direct the reader to \cite[Chapter~6]{Simon83} for more details (see also \cite{bogelainduzaarfusco2017}). 

We denote by $\Lambda_1 (\R^2)$
and $\Lambda^1 (\R^2)$ the linear spaces of $1$-vectors and  
$1$-covectors in $\R^2$, respectively,
while $\mathcal{D}^1 (\R^2)$ stands for
the set of smooth $1$-forms with compact support in $\R^2$.

A \textit{$1$-dimensional current} in $\R^2$ 
is a continuous linear functional on $\mathcal{D}^1 (\R^2)$.
Instead, $0$-dimensional currents in $\R^2$
are simply defined as the usual distributions in $\R^2$.
The family of $1$-dimensional ($0$-dimensional) currents in $\R^2$
is denoted by $\mathcal{D}_1 (\R^2)$ ($\mathcal{D}_{0}   (\R^2)$). 

We say that $T \in \mathcal{D}_1 (\R^2)$
is an \textit{integer multiplicity rectifiable $1$-current} 
if it can be represented as 
\[
T (\omega) = \int_{M} \langle \omega (x) , \tau (x) \rangle \, \theta (x) \, d \mathcal{H}^{1} (x)
\quad \text{ for every } \omega \in \mathcal{D}^1 (\R^2),
\]
where $M$ is an $\mathcal{H}^1$-measurable countably 
$1$-rectifiable subset of $\R^2$, 
$\theta$ is an $\mathcal{H}^1$-measurable positive integer-valued 
function, $\tau: M \to \Lambda_1 (\R^2)$ is an $\mathcal{H}^1$-measurable 
function such that $\tau (x)$ is a unit vector belonging to the approximate tangent space of $M$ at $x$
for $\mathcal{H}^1$-a.e. $x \in M$,  
and $\langle \cdot , \cdot \rangle$ denotes the usual pairing between  
$\Lambda^1 (\R^2)$ and $\Lambda_1 (\R^2)$.
In the special case when 
\[
T (\omega) = \int_{M} \langle \omega (x) , \tau (x) \rangle \, d \mathcal{H}^{1} (x)
\quad \text{ for every } \omega \in \mathcal{D}^1 (\R^2),
\]
we write $T = [\![ M ]\!]$.
The boundary $\partial T$ of $T$ is then defined as  
the $0$-dimensional current in $\R^2$ such that
\[
\partial T ( \omega ) 
= T (d \omega) \quad \text{ for every } \omega \in C^0_c (\R^2),
\]
while the mass $\mathbf{M} (T)$ of $T$ is given by
\[
\mathbf{M} (T) := \sup \left\{ T (\omega) :  \omega \in \mathcal{D}^1 (\R^2), \, | \omega | \leq 1  \right\}. 
\]
More in general, for any open set $U \subset \R^2$, we set
\[
\mathbf{M}_U (T) := \sup \left\{ T (\omega) :  \omega \in \mathcal{D}^1 (\R^2), \,
 | \omega | \leq 1, \,  \text{supp} \, \omega  \subset U  \right\}. 
\]

Let $A \subset \mathbb{S}^{1}$ be 
an $\mathcal{H}^{1}$-measurable set.
We will say that $A$ is a set of finite perimeter 
 on $\mathbb{S}^{1}$ 
if there exists $Q \in  \mathcal{D}_{0}   (\R^2)$ 
with $\text{supp} \,  Q   \subset \mathbb{S}^{1}$ and 
$$
 Q   = \partial   [\![ A ]\!],
$$
with the property that $\mathbf{M}_U ( Q )  < \infty$
for every $U \subset \subset \R^2$. 
 By the Riesz representation theorem it follows that there exists a 
Radon measure $\mu_Q$ such that
$$
\int_A \diver_{\parallel} \varphi (x) \, d \mathcal{H}^{1} (x)
= \int_{\mathbb{S}^{1}} \varphi (x) \cdot  d  \mu_Q  (x),
$$
for every smooth vector field $\varphi: \mathbb{S}^{1} \to \R^2$ such that $\varphi (x) \cdot \hat{x} = 0$ for every $x \in \mathbb{S}^{1}$, 
where $\diver_{\parallel} \varphi$ stands for the tangential divergence of $\varphi$ on $\mathbb{S}^1$.
%
If $A \subset \mathbb{S}^{1}$ is a set of finite perimeter on the sphere, 
the reduced boundary $\partial^* A$ of $A$ is the set of points $x \in \mathbb{S}^{1}$
such that the limit 
$$
\nu^A (x) := \lim_{\rho \to 0^+} \frac{\mu_Q (\mathbf{B}_{\rho} (x))}{ |\mu_Q| (\mathbf{B}_{\rho} (x))} 
$$
exists, $\nu^A (x) \in T_x \mathbb{S}^{1}$, and  $| \nu^A (x) | = 1$.
The De Giorgi structure theorem holds true also for sets of finite perimeter on the sphere.
In particular, $\partial^* A$ is finite,
$ \mu_Q  = \nu^A \mathcal{H}^{0} \mres \partial^* A$, and 
\begin{equation} \label{div theorem hypersurfaces}
\int_A \diver_{\parallel} \varphi (x) \, d \mathcal{H}^{1} (x)
= \int_{\partial^* A} \varphi (x) \cdot \nu^A (x) \, d \mathcal{H}^{0} (x),
\end{equation}
for every smooth vector field $\varphi: \mathbb{S}^{1} \to \R^2$ such that $\varphi (x) \cdot \hat{x} = 0$ for every $x \in \mathbb{S}^{1}$.

The isoperimetric inequality on the circle states that, if $A \subset \mathbb{S}^{1}$ is a set of finite perimeter on $\mathbb{S}^{1}$
with $\mathcal{H}^{1} (A) = \mathcal{H}^{1} (\mathbf{B}_{\beta} (e_1))$, 
then (see \cite{schmidt})
\begin{equation} \label{isop ineq}
 \mathcal{H}^{0} (\partial^* \mathbf{B}_{\beta} (e_1)) \leq 
\mathcal{H}^{0} (\partial^* A), \quad 
\end{equation}
and 
\begin{equation} \label{isop eq}
\mathcal{H}^{0} (\partial^* \mathbf{B}_{\beta} (e_1)) = 
\mathcal{H}^{0} (\partial^* A) 
\quad \Longleftrightarrow \quad
A =_{\mathcal{H}^1} \mathbf{B}_{\beta} (p) \text{ for some } p \in \mathbb{S}^1.
\end{equation}

\subsection{Circular rearrangement of sets}

If  $E \subset \R^N$ is a Lebesgue measurable set and $(r, z) \in (0, +\infty) \times \R^{N-2}$, we define the slice $E_{(r, z)}$ as the subset of $\R^2_0$ given by
\[
E_{(r, z)} := \{ x \in \partial D(r): (x, z) \in E \}.
\]
We now give the definition of circular symmetrization of a set in $\R^N$. 
\begin{definition}
Let $E \subset \R^N$ be a Lebesgue measurable set. The circular rearrangement of $E$ is the set $E^s$ defined as 
\begin{equation} \label{def Es}
E^s := \{ (x, z) \in \Phi_N (\Pi_{N-1} (E) \times \mathbb{S}^1) : d_{\mathbb{S}^1}  (\hat{x} \cdot e_1) < \beta( | x | , z) \}
 \end{equation}
where 
\begin{equation} \label{def beta}
\beta(r, z) =\frac{\mathcal{H}^1 (E_{(r, z)})}{2 r}, \quad \text{ for every } (r, z) \in \Pi_{N-1} (E).
\end{equation}
\end{definition}

\begin{remark} \label{Es no points -1}
Note that with the definition given above the set $E^s$ does not include points of the half-hyperplane
$H = \{ \hat{x} \cdot e_1 =  -1 \}$, even when the original set $E$ does include some of these points. 
This is not affecting our results, since $H$ is an $\mathcal{H}^n$-negligible set.
The advantage of definition \eqref{def Es} is clarified by the next result.
\end{remark}
\begin{proposition}{\label{prop_Omega^s is open}}
Let $E \subset \mathbb{R}^N$ be a Lebesgue measurable set, and let 
$\beta : \Pi_{N-1} (E) \to [0, \pi]$ be given by \eqref{def beta}.
Then, 

\begin{itemize}

\item[(i)] $E^s$ is open  $\Longleftrightarrow$  $\beta$ is lower semicontinuous;

\vspace{.2cm}

\item[(ii)] $E$ is open $\Longrightarrow$ $E^s$ is open.

\end{itemize}
\end{proposition}

\begin{proof}
We divide the proof into steps.

\vspace{.2cm}

\noindent
\textbf{Step 1.} $E$ open $\Longrightarrow$ $\beta$ lower semicontinuous.
Let $E$ be open and assume, by contradiction, that $\beta$ is not lower semicontinuous. Then, there exist $(r, z) \in \Pi_{N-1} (E)$ and a sequence $\{ (r_h, z_h )\}_{h \in \mathbb{N}} \subset  \Pi_{N-1} (E)$
with $(r_h, z_h ) \to (r, z)$ such that 
\[
\beta (r, z) > \liminf_{h \to + \infty} \beta (r_h, z_h).
\]
Passing to a suitable (not relabelled) subsequence, we can assume that the liminf is a limit and so we have
\[
\beta (r, z) > \lim_{h \to + \infty} \beta (r_{h}, z_{h})
\quad  \Longleftrightarrow \quad
\mathcal{H}^1 (C) > \lim_{h \to + \infty} \mathcal{H}^1 (S_h),
\]
where the sets $C, S_h \subset \mathbb{S}^1$ are defined as:
\[
C:= \frac{1}{ r} E_{(r, z)} \quad \text{ and } \quad S_h:= \frac{1}{ r_h} E_{(r_h, z_h)}, \, \,  \text{ for every }h \in \mathbb{N}. 
\]
Then, there exist $\delta >0$ and $j \in \mathbb{N}$ such that
\begin{align*} 
\mathcal{H}^1 (C) > \mathcal{H}^1 (S_h) + 2 \delta, \quad \text{ for every } h > j,
\end{align*}
and so
\begin{align}\label{dist bigger than delta 2}
\mathcal{H}^1 (C \setminus S_h) \geq \mathcal{H}^1 (C) - \mathcal{H}^1 (S_h) > 2 \delta, \quad \text{ for every } h > j. 
\end{align}
Since $E$ is open and $C= \frac{1}{r} E_{(r, y)}$, we have that $C$ is (nonempty and) relatively open in $\mathbb{S}^1$. Equivalently, $C$ is open in the topology induced in $\mathbb{S}^1$ by
$d_{\mathbb{S}^1}$. 
For every $k \in \mathbb{N}$, we define the set $C^k \subset \mathbb{S}^1$ as 
\[
C^k := \left\{ \omega \in \mathbb{S}^1 : d_{\mathbb{S}^1} (\omega , \mathbb{S}^1 \setminus C) \geq \frac1k \right\}.
\]
For $k \in \mathbb{N}$ sufficiently large, $C^k$ is nonempty and compact in $\mathbb{S}^1$. We also observe that, 
with the usual convention $d_{\mathbb{S}^1} (\omega , \emptyset) = + \infty$, if $C = \mathbb{S}^1$ we have 
$C^k = \mathbb{S}^1$. 
Moreover, $C^{k_1} \subset C^{k_2} \subset C$ whenever $k_1 < k_2$, and 
\[
C = \bigcup_{k \in \mathbb{N}} C^k.
\]
Therefore, there exists $k_\delta \in \mathbb{N}$ such that
\[
\mathcal{H}^1 (C \setminus C^{k_\delta}) < \delta.
\]
Last inequality, together with \eqref{dist bigger than delta 2}, implies that
\[
\mathcal{H}^1 (C^{k_\delta} \setminus S_h) > \delta, \quad \text{ for every } h > j.
\]
Then, for every $h > j$ there exists $\omega_h \in C^{k_\delta} \setminus S_h$.
 By compactness of $C^{k_\delta} $, there exists $\overline{\omega} \in C^{k_\delta}$ such that, up to subsequences, 
\[
\omega_h \to \overline{\omega}.
\]
Since $\overline{\omega} \in  C^{k_\delta} \subset C$, by definition of $C$ we have that 
$(r \overline{\omega}, z) \in E$.
Thus, recalling the definition of $S_h$, the sequence $\{ (r_h \omega_h, z_h) \}_{h > j}$ is such that
\[
(r_h \omega_h, z_h) \notin E \quad \text{ for every } h > j,
\]
and 
\[
(r_h \omega_h, z_h) \to (r \overline{\omega}, z) \in E.
\]
But this is impossible, since $E$ is open.

\vspace{.2cm}

\noindent
\textbf{Step 2.} $\beta$ lower semicontinuous $\Longrightarrow$  $E^s$ open. 
Let $\beta$ be lower semicontinuous and suppose, by contradiction, that $E^s$ is not open. Then, there exist $(x,z)\in E^s$ and a sequence $\{ (x_h,z_h) \}_{h} \subset \Phi_N (\Pi_{N-1} (E) \times \mathbb{S}^1)$ such that
$(x_h,z_h)\to (x,z)$ and $(x_h,z_h)\notin E^s$ for all $h\in \mathbb{N}$. Then, we have 
\begin{align*}
\beta(|x|,z) > d_{\mathbb{S}^1} (\hat{x}, e_1)
\quad \text{ and } \quad d_{\mathbb{S}^1} ( \hat{x}_h,  e_1) \geq \beta(|x_h|,z_h) \quad \forall\, h\in\mathbb{N}.
\end{align*}
Using that fact that $\beta$ is lower semicontinuous, we obtain 
\begin{align*}
\beta(|x|,z) > d_{\mathbb{S}^1} (\hat{x},  e_1)
= \lim_{h \to + \infty} d_{\mathbb{S}^1} (\hat{x}_h,  e_1)
\geq \liminf_{h\to\infty}\beta(|x_h|,z_h) \geq \beta(|x|,z),
\end{align*}
which is impossible.

%

\vspace{.2cm}

\noindent
\textbf{Step 3.} We conclude. Property (i) follows by combining Step~1 (applied to the set $E^s$) and Step~2.
Let us now show (ii). If $E$ is open, thanks to Step~1, the function $\beta$ is lower semicontinuous and then, by Step~2, $E^s$ is open.
\end{proof}

\section{Circular rearrangement of subgraphs}\label{subsection_circular in Rn+1}

\noindent
In this section and in the rest of the paper, we assume $n \in \mathbb{N}$ with $n \geq 2$ and $\Omega \subset \mathbb{R}^n$ open.
We decompose $\R^{n+1}$ as $\R^{n+1} = \R^2 \times \R^{n-2} \times \R$, and we label points of $\R^{n+1}$
as $(x, y, t)$, with $x \in \R^2$, $y \in \R^{n-2}$ and $t \in \R$. If $E \subset \R^{n+1}$ is a set of locally finite perimeter, 
for every $(x, y, t) \in \partial^* E \cap (\R^2_0 \times \R^{n-2} \times \R)$ we decompose the measure theoretic inner unit normal to $E$ at $(x, y, t)$ as $\nu^E (x, y, t) = ( \nu^E_x (x, y, t), \nu^E_y,  (x, y, t) , \nu^E_t (x, y, t))$, where $\nu^E_x (x, y, t) \in \R^2$, $\nu^E_y (x, y, t) \in \R^{n-2}$, and $\nu^E_t (x, y, t) \in \R$. Moreover, we further decompose $\nu^E_x (x, y, t)$ as
\[
\nu^E_x (x, y, t) = (\hat{x} \cdot \nu^E_x (x, y, t) ) \hat{x} + \nu^{E}_{\!  x {\scriptscriptstyle\parallel}}  (x, y, t) x_\parallel, 
\]
where $\hat{x}$ and $x_{\parallel}$ are defined in \eqref{basis of R2}, and $\nu^{E}_{\!  x {\scriptscriptstyle\parallel}} (x, y, t) = \nu^E_x (x, y, t) \cdot x_\parallel$.

\subsection{Subgraphs}
Let now $u:\Omega\to \mathbb{R}$ be a Lebesgue measurable function. 
We denote by $\Sigma^u\subset \mathbb{R}^{n+1}$ the subgraph of $u$, defined as
\begin{align*}
\Sigma^u = \{ (x, y, t) \in \mathbb{R}^2 \times\mathbb{R}^{n-2} \times \mathbb{R} :\, (x, y) \in \Omega \text{ and } u (x, y) > t  \}
\end{align*}
The following result follows from \cite[Chapter 4, Section~1.5, Theorem~1--Theorem~5]{GMSbook1} (see also \cite[Proposition~3.4]{ccdpmGAUSS}, and \cite[Theorem 5.2]{Perugini}).
\begin{proposition}  \label{giaquinta}
Let $U \subset \mathbb{R}^n$ be open and bounded, and let $u \in L^1 (U)$. 
Then, $u \in BV (U)$ if and only if $\Sigma^u$ is a set of finite perimeter in $U \times \mathbb{R}$.
In this case, we have:
\[
\partial^* \Sigma^{u} \cap ( (S_u)^c \times \mathbb{R})
=_{\mathcal{H}^{n}} 
\{ (x,y,t) \in (S_u)^c \times \mathbb{R} : u^{\wedge} (x, y) = u^{\vee} (x, y) = t \},
\]
with
\begin{equation} \label{normal1}
\nu^{\Sigma^{u}} (x, y,  u (x, y)) = \Bigg( \frac{\nabla_x u(x, y)}{\sqrt{1 + | \nabla u(x, y)|^{2}}} ,
\frac{\nabla_y u (x, y)}{\sqrt{1 + | \nabla u (x, y)|^{2}}}, 
\frac{-1}{\sqrt{1 + | \nabla u (x, y)|^{2}}} \Bigg)
\end{equation}
for $\mathcal{H}^{n}$-a.e. $(x, y) \in \mathcal{D}_u$, and
\[
\nu^{\Sigma^{u}} (x, y, u (x, y)) = \Bigg( \frac{d D^c u}{d | D^c u|} (x, y), 0 \Bigg)
\quad \text{ for $| D^c u|$-a.e. $(x, y) \in U$.}
 \] 
Moreover, 
\[
\partial^* \Sigma^{u} \cap ( S_u \times \mathbb{R})
=_{\mathcal{H}^{n}} 
\{ (x,y,t) \in S_u \times \mathbb{R} : u^{\wedge} (x, y) < t < u^{\vee} (x, y) \},
\]
and 
\[
\nu^{\Sigma^{u}} (x, y, t) = \left( \nu_u (x, y), 0 \right),
 \]  
for  $\mathcal{H}^{n-1}$-a.e. $(x, y) \in S_u$ and for every $ t \in (u^{\wedge} (x, y),  u^{\vee} (x, y))$.
Finally, if $B \subset U$ is a Borel set, the perimeter 
of $\Sigma^{u}$ in $B \times \mathbb{R}$ is given by
\begin{equation} \label{formula per subgraph}
P (\Sigma^{u} ; B \times \mathbb{R}) = \int_B \sqrt{1 + | \nabla u|^2} \, dx \, dy 
+| D^s u | (B). 
\end{equation}
\end{proposition}

\noindent
We will also need the following statement.
\begin{proposition} \label{prop_no flat parts means being Sobolev}
Let $U \subset \mathbb{R}^n$ be open and bounded, let $ u \in BV (U)$ and let $A \subset U$ be a Borel set. Then, the following are equivalent:

\begin{itemize}

\item[(i)] $\displaystyle \mathcal{H}^n 
\Big( \{ (x, y, t) \in \partial^* \Sigma^{u} : \nu_t^{\Sigma^{u}} (x, y, t) = 0 \} 
\cap \left( A \times \mathbb{R}   \right) \Big) = 0$;

\vspace{.2cm}

\item[(ii)] $\displaystyle P ( \Sigma^{u} ; B \times \mathbb{R}) = 0$ for every Borel set $B \subset A$ with $\mathcal{H}^n (B) = 0$.
\end{itemize}
Moreover, if $A$ is open, (i) and (ii) are also equivalent to:

\begin{itemize}

\item[(iii)] $\displaystyle u \in W^{1, 1} (A)$.

\end{itemize}
\end{proposition}

\begin{proof}
The equivalence between (i) and (ii) follows directly from \cite[Lemma 4.1]{ChlebikCianchiFuscoAnnals05}.
Let us now assume that $A$ is open, and let us show that (ii) is equivalent to (iii).
\noindent
Thanks to formula \eqref{formula per subgraph}, we have
that for every Borel set $B \subset A$ with $\mathcal{H}^n (B) = 0$
\[
P (\Sigma^{u} ; B \times \mathbb{R}) = \int_B \sqrt{1 + | \nabla u|^2} \, dx \, dy 
+| D^s u | (B) = | D^s u | (B).
\]
Therefore,
\[
\text{ (ii) }\, \Longleftrightarrow \, 
| D^s u | (B) = 0 \text{ for every Borel set $B \subset A$ with $\mathcal{H}^n (B) = 0$}
\, \Longleftrightarrow \, u \in W^{1, 1} (A).
\]
\end{proof}
\noindent
We now introduce a class of functions that extends Definition~\ref{def w1p0tau}.
\begin{definition} \label{def BV0tau}
Let $n \in \mathbb{N}$ with $n \geq 2$, let $\Omega \subset \R^n$ be open, let $u:\Omega \to \R$, and let  
$u_0$ be given by \eqref{def u0}.
We say that $u \in BV_{0, \tau} (\Omega)$ if the following conditions are satisfied:
\begin{itemize}

\item[(a)] $u_0 \in BV (\Phi_n (A \times \mathbb{S}^1) )$ for every 
open set $A \subset (0, \infty) \times \mathbb{R}^{n-2}$ with $A \subset \subset \Pi_{n-1} (\Omega)$,

\vspace{.2cm}

\item[(b)] $u \geq 0$ \quad  $\mathcal{L}^n$-a.e. in $\Omega \setminus 
\Phi_n(\Pi_{n-1}^{\textnormal{a}}(\Omega) \times\mathbb{S}^1)$,

\end{itemize}
where $\Phi_n$, $\Pi_{n-1} (\Omega)$, and $\Pi_{n-1}^{\textnormal{a}}(\Omega)$ are defined by \eqref{def_PiN}, \eqref{def_Pi_a}, and \eqref{def_PhiN}, respectively.
\end{definition}
\begin{remark} \label{rem holder}
We have $W^{1, p}_{0, \tau} (\Omega) \subset 
W^{1, 1}_{0, \tau} (\Omega) \subset BV_{0, \tau} (\Omega)$
for every $p \in [1, + \infty)$, where the first inclusion follows from 
the H\"older inequality on compact sets.
\end{remark}
\begin{remark} \label{u bv finite perimeter}
Thanks to Proposition~\ref{giaquinta}, if $u \in BV_{0, \tau} (\Omega)$ then $\Sigma^{u_0}$ is a set of finite perimeter 
in $\Phi_n (A \times \mathbb{S}^1) \times \R = \Phi_{n+1} ((A \times \R) \times \mathbb{S}^1)$, for every open set $A \subset \subset \Pi_{n-1} (\Omega)$.
\end{remark}
\noindent
Adapting \cite[Proposition~6.1]{CagnettiPeruginiStoger} to the case of subgraphs of functions belonging to $BV_{0,\tau}(\Omega)$, we obtain the following.
\begin{proposition}[Coarea Formula] \label{coarea formula}
Let $\Omega \subset \mathbb{R}^n$ be open, let 
$u \in BV_{0, \tau} (\Omega)$, and
let $g:\R^{n+1} \rightarrow [0,\infty]$ be a Borel function. 
Then, 
\begin{align*}
&\int_{\partial^* \Sigma^{u_0}} g(x, y, t) | \nu^{\Sigma^{u_0}}_{\!  x {\scriptscriptstyle\parallel}} (x, y, t) |
\, d\mathcal{H}^{n} (x, y, t) \\
&= \int_{(0, \infty) \times \mathbb{R}^{n-2} \times \mathbb{R}} \left( \int_{(\partial^* \Sigma^{u_0})_{(r, y, t) }} g(x, y, t) \, d\mathcal{H}^{0}(x) \right) dr \, dy \,  dt.
\end{align*}
\end{proposition}
\begin{remark} \label{rem parallel normal if u sobolev}
From Proposition~\ref{giaquinta} it follows that when $u \in BV_{0, \tau} (\Omega)$,
the subgraph $\Sigma^{u_0}$ is a set of locally finite perimeter in $\Phi_{n+1}((\Pi_{n-1}(\Omega)\times \mathbb{R})\times \mathbb{S}^1)$. 
Moreover, if $u \in W^{1, 1}_{0, \tau} (\Omega)$
\[
\nu^{\Sigma^{u_0}}_{\!  x {\scriptscriptstyle\parallel}} (x, y, u_0 (x, y))
= \displaystyle \frac{\Dx u_0 (x, y)}{\sqrt{1 + |\nabla u_0 (x, y)|^2}}, \quad \text{ for $\mathcal{H}^n$-a.e. } (x,y) \in 
\mathcal{D}_{u_0} \cap \Phi_n (\Pi_{n-1}(\Omega) \times \mathbb{S}^1).
\]
\end{remark}

\begin{remark} 
Combining Proposition~\ref{coarea formula} and Remark~\ref{rem parallel normal if u sobolev} we obtain that
if $u \in W^{1, 1}_{0, \tau} (\Omega)$, then
\begin{equation} \label{only tangential non zero}
\mathcal{H}^n \left(  \partial^* \Sigma^{u_0} \cap \Phi_{n+1} (R \times \mathbb{S}^{1}) \cap (\{ \Dx u_0 \neq 0\} \times \mathbb{R}) \right) = 0,
\end{equation}
for every Borel set $R \subset \Pi_{n-1} (\Omega) \times \R$ with $\mathcal{H}^n (R) = 0$.
Indeed,
\begin{align*}
&\mathcal{H}^n \left(  \partial^* \Sigma^{u_0} \cap \Phi_{n+1} (R \times \mathbb{S}^{1}) \cap (\{ \Dx u_0 \neq 0\} \times \mathbb{R}) \right) \\
  &= \int_{R}  \left(
\int_{( \partial^{*} \Sigma^{u_0} \cap (\{ \Dx u_0 \neq 0\} \times \mathbb{R}) )_{(r,y, t)}}
  \frac{\chi_{\{ \Dx u_0  \neq 0 \}} (x, y)}{ | \Dx u_0  (x, y)|}  \sqrt{1 + |\nabla u_0|^2} \, d \mathcal{H}^0 (x)  \right) \, dr  \, dy \, dt = 0,
\end{align*}
where in the last equality we used the fact that $\mathcal{H}^n (R) = 0$.
\end{remark}

\begin{proposition}\label{prop_circular rear and circular sym are connected}
Let $n \in \mathbb{N}$ with $n \geq 2$, and let $\Omega \subset \mathbb{R}^n$ be open. 
Let $u: \Omega \to \R$ be a Lebesgue measurable function such that $u_0 \in L^1_{\textnormal{loc}} (\Phi_{n}  (\Pi_{n-1} (\Omega)  \times \mathbb{S}^{1}))$, where $u_0$ is given by \eqref{def u0}.
Let $\mu$, $v_\mu$, and $\alpha_\mu$ be given by \eqref{mu distr}, \eqref{def vmu} and \eqref{def alpha}, respectively.
%
%
%
Then, $\mu \in L^1 (A \times (-d, + \infty ))$ for every open set $A \subset \subset \Pi_{n-1} (\Omega)$ and for every $d > 0$. 
\end{proposition}

\begin{proof}
Let $A \subset \subset \Pi_{n-1} (\Omega)$ be open and let $d > 0$.
We have
\begin{align*}
&\| \mu \|_{L^1 (A \times (0, + \infty ))} 
= \int_0^{\infty} \left( \int_{A   } \mu (r, y, t ) \, dr d y  \right) dt  
= \int_0^{\infty} \left( \int_{A} \int_{\partial D (r)}  \chi_{\{u_0 >  t \}} (x, y) \, d \mathcal{H}^1 (x) \, dr d y  \right) dt  \\
&= \int_0^{\infty} \left( \int_{ \Phi_n (A \times \mathbb{S}^{1})  } \chi_{\{u_0 >  t \}} (x, y) \, dx d y  \right) dt
= \| (u_0)_+ \|_{L^1 (\Phi_n (A \times \mathbb{S}^{1}))} < + \infty. 
\end{align*}
On the other hand,
 \begin{align*}
\| \mu \|_{L^1 (A \times (-d , 0 ))} 
&= \int_{-d}^{0} \left( \int_{A   } \mu (r, y, t ) \, dr d y  \right) dt
= \int_{-d}^{0} \left( \int_{A} \int_{\partial D (r)}  \chi_{\{u_0 >  t \}} (x, y) \, d \mathcal{H}^1 (x) \, dr d y  \right) dt  \\
&= \int_{-d}^{0} \left( \int_{ \Phi_n (A \times \mathbb{S}^{1})  } \chi_{\{u_0 >  t \}} (x, y) \, dx d y  \right) dt 
\leq \mathcal{H}^n \left( \Phi_n (A \times \mathbb{S}^{1}) \right) d
< + \infty.
 \end{align*}
Therefore, $\mu \in L^1 (A \times (-d, + \infty ))$.
\end{proof}

Let us now show that the function $\mu$ is well defined.
\begin{proposition} \label{prop several useful things}
Let $n \in \mathbb{N}$ with $n \geq 2$, and let $\Omega \subset \mathbb{R}^n$ be open. 
Let $u_1, u_2: \Omega \to \R$ be Lebesgue measurable functions such that $u_{1,0}, u_{2,0} \in L^1_{\textnormal{loc}} (\Phi_{n}  (\Pi_{n-1} (\Omega)  \times \mathbb{S}^{1}))$, where $u_{1,0}$ and $u_{2,0}$ are given by \eqref{def u0} with $u_1$ and $u_2$ in place of $u$, respectively.
Let $\mu_{1}$ and $\mu_{2}$ be defined by \eqref{mu distr} with $u_1$ and $u_2$ in place of $u$, respectively. 
Then, the following are equivalent:
\begin{enumerate}

\item[(i)] $u_1 = u_2$ $\mathcal{H}^n$-a.e. in $\Omega$;

\vspace{.2cm}

\item[(ii)] $\Sigma^{u_{1,0}} =_{\mathcal{H}^{n+1}} \Sigma^{u_{2,0}}$ in $\Phi_{n}  (\Pi_{n-1} (\Omega)  \times \mathbb{S}^{1})) \times \R$.

\end{enumerate}
Moreover, both \textnormal{(i)} and  \textnormal{(ii)} imply 
\begin{enumerate}

\vspace{.2cm}

\item[(iii)] \( \mu_{1} = \mu_{2} \) $\mathcal{H}^n$-a.e. in $\Pi_{n-1} (\Omega) \times \R$.

\end{enumerate}
%
%
%
%
%
%
%
\end{proposition}

\begin{proof}
We start by showing that
for every open set $A \subset \subset \Pi_{n-1} (\Omega)$ 
\begin{equation} \label{two equalities}
\begin{split} 
&\| u_1 - u_2 \|_{L^1 (\Omega \cap \Phi_{n}  (A  \times \mathbb{S}^{1}))}  \\
&= \mathcal{H}^{n+1} \Big( (\Sigma^{u_{1,0}} \Delta \Sigma^{u_{2,0}}) \cap \big( \Phi_{n+1}  ((A \times \R) \times \mathbb{S}^{1})) \big) \Big)  \\
&= \int_{A \times \R }
\mathcal{H}^1 \Big( (\partial D (r) \cap \{ u_{1,0} (\cdot, y) > t \} )\Delta (\partial D (r) \cap \{ u_{2,0} (\cdot, y) > t \}) \Big) \, dr \, dy \, dt. 
\end{split}
\end{equation}
Indeed, let $A \subset \subset \Pi_{n-1} (\Omega)$ be open. Then,
\begin{align*}
&\| u_1 - u_2 \|_{L^1 (\Omega \cap \Phi_{n}  (A  \times \mathbb{S}^{1}))}
= \| u_{1,0} - u_{2,0} \|_{L^1 (\Phi_{n}  (A  \times \mathbb{S}^{1}))} \\
&= \int_{\Phi_{n}  (A  \times \mathbb{S}^{1})} |u_{1,0} (x, y) - u_{2,0} (x, y) |  \, dx \, dy \\ 
&= \int_{\Phi_{n}  (A  \times \mathbb{S}^{1})} \left( \int_{\R} ( \chi_{[u_{1,0} (x, y), u_{2,0} (x, y))} (t)
+  \chi_{[u_{2,0} (x, y), u_{1,0} (x, y))} (t)  ) \, dt \right) \, dx \, dy \\
&= \int_{\Phi_{n}  (A  \times \mathbb{S}^{1})} \left( \int_{\R} ( 
\chi_{ \{ u_{2,0} > t \} \setminus \{ u_{1,0} > t \} } (x, y)
+ \chi_{ \{ u_{1,0} > t \} \setminus \{ u_{2,0} > t \} } (x, y) 
) 
\, dt \right) \, dx \, dy \\
&= \int_{\Phi_{n}  (A  \times \mathbb{S}^{1})} \left( \int_{\R}  ( 
\chi_{\Sigma^{u_{2,0}} \setminus \Sigma^{u_{1,0}}} (x, y, t)
+  \chi_{\Sigma^{u_{1,0}} \setminus \Sigma^{u_{2,0}}} (x, y, t) ) \, dt \right) \, dx \, dy \\
&= \mathcal{H}^{n+1} \Big( (\Sigma^{u_{1,0}} \Delta \Sigma^{u_{2,0}}) \cap \big( \Phi_{n+1}  ((A \times \R) \times \mathbb{S}^{1})) \big) \Big).
\end{align*}
On the other hand,
\begin{align*}
&\mathcal{H}^{n+1} \Big( (\Sigma^{u_{1,0}} \Delta \Sigma^{u_{2,0}}) \cap \big( \Phi_{n+1}  ((A \times \R) \times \mathbb{S}^{1})) \big) \Big) \\
&= \int_{\Phi_{n+1}  ( (A \times \R)  \times \mathbb{S}^{1})} 
\chi_{ \{ u_{1,0} > t \} \Delta \{ u_{2,0} > t \}} (x, y) \, dt  \, dx \, dy \\
&= \int_{A \times \R } \left( \int_{\partial D (r)} \chi_{ \{ u_{1,0} > t \} \Delta \{ u_{2,0} > t \}} (x, y) \, d \mathcal{H}^1 (x) \right) dr  \, dy \, dt \\
&=\int_{A \times \R }
\mathcal{H}^1 \Big( (\partial D (r) \cap \{ u_{1,0} (\cdot, y) > t \} )\Delta (\partial D (r) \cap \{ u_{2,0} (\cdot, y) > t \}) \Big) \, dr \, dy \, dt,
\end{align*}
and this shows \eqref{two equalities}.
Since $A \subset \subset \Pi_{n-1} (\Omega)$ was arbitrary, from the first equality in \eqref{two equalities}, it follows that 
(i) $\Longleftrightarrow$ (ii).

Suppose now that (ii) holds.
Then, since $A \subset \subset \Pi_{n-1} (\Omega)$ was arbitrary, from the second equality in \eqref{two equalities} we obtain that
\[
\partial D (r) \cap \{ u_{1,0} (\cdot, y) > t \} =_{\mathcal{H}^1} \partial D (r) \cap \{ u_{2,0} (\cdot, y) > t \}, \qquad \text{ for $\mathcal{H}^n$-a.e. } (r, y, t) \in \Pi_{n-1} (\Omega) \times \R.
\]
As a consequence, 
\[
\mu_{u_1} (r, y, t) = \mathcal{H}^1 (\partial D (r) \cap \{ u_{1,0} (\cdot, y) > t \}) =  \mathcal{H}^1 (\partial D (r) \cap \{ u_{2,0} (\cdot, y) > t \}) = \mu_{u_2} (r, y, t), 
\]
for $\mathcal{H}^n$-a.e. $(r, y, t) \in \Pi_{n-1} (\Omega) \times \R$, and this shows (iii).
\end{proof}

The next result will be used to prove the P\'olya--Szeg\"o inequality. 
\begin{proposition} \label{ssssd}
Let $n \in \mathbb{N}$ with $n \geq 2$, and let $\Omega \subset \mathbb{R}^n$ be open. 
Let $u \in BV_{0, \tau} (\Omega)$, and let $\mu$ and $\alpha_\mu$ be given by \eqref{mu distr} and
\eqref{def alpha}, respectively. 
Then, 
\begin{equation} \label{only 0<alpha<pi counts}
\begin{split} 
(r, y, t) \in \{ \alpha_\mu = 0 \}^{(1)} 
&\quad \Longrightarrow \quad
(x, y, t) \in (\Sigma^{u_0})^{(0)} \quad \text{ for every } x \in \partial D (r); \\
(r, y, t) \in \{ \alpha_\mu = \pi \}^{(1)} 
&\quad \Longrightarrow \quad
(x, y, t) \in (\Sigma^{u_0})^{(1)} \quad \text{ for every } x \in \partial D (r).
\end{split}
\end{equation}
\end{proposition}
\begin{proof}
We will only prove the first implication, since the second one is analogous. In the following, we will use the fact that
\begin{align}
(x' ,  y' ,  t') \in B^{n+1}_\rho ( x , y, t) 
\quad \Longrightarrow \quad
(|x'|, y', t') \in B^{n}_\rho ( |x| , y, t) \quad \text{ for every } \rho > 0. \label{rt1} 
\end{align}
Let $(r, y, t) \in \{ \alpha_\mu = 0 \}^{(1)}$, and let $x \in \partial D (r)$.
Let now $\rho \in (0, |x|)$. Thanks to \eqref{rt1}, we have 
 \begin{align*}
&\mathcal{H}^{n+1} \left( \Sigma^{u_0} \cap B^{n+1}_\rho ( x , y, t)  \right)
=  \int_{\mathbb{R}^{n+1}} \chi_{\Sigma^{u_0} \cap B^{n+1}_\rho ( x , y, t)} (x' ,  y' ,  t')
\, d x' \, d y' \, d t' \\
&= \int_{\Pi_{n-1} (\Omega) \times \R} 
\left( 
 \int_{ \partial D (r')} \chi_{\Sigma^{u_0} \cap B^{n+1}_\rho ( x , y, t)} (x' ,  y' ,  t')  \chi_{B^{2}_\rho ( x )} ( x') \, d \mathcal{H}^1 (x')  
 \right) 
 \, d r' \, d y'  \, d t' \\
&\leq  \int_{\Pi_{n-1} (\Omega) \times \R} \chi_{B^{n}_\rho ( |x| , y, t)} (r' ,  y' ,  t') 
\left( 
 \int_{ \partial D (r')} \chi_{\Sigma^{u_0}} (x' ,  y' ,  t') 
 \chi_{B^{2}_\rho ( x )} ( x') 
 \, d \mathcal{H}^1 (x')  
 \right) 
 \, d r' \, d y'  \, d t' \\
&\leq \int_{\Pi_{n-1} (\Omega) \times \R} \chi_{\{ \alpha_\mu > 0 \} \cap B^{n}_\rho ( |x| , y, t)} (r' ,  y' ,  t')
\left( 
 \int_{ \partial D (r')}  
 \chi_{B^{2}_\rho ( x )} ( x') 
 \, d \mathcal{H}^1 (x')  
 \right) 
 \, d r' \, d y'  \, d t' \\
 &=  \int_{\Pi_{n-1} (\Omega) \times \R} \mathcal{H}^1 (\partial D (r') \cap B^{2}_\rho ( x ))
 \chi_{\{ \alpha_\mu > 0 \} \cap B^{n}_\rho ( |x| , y, t)} (r' ,  y' ,  t') 
 \, d r' \, d y'  \, d t'.
 \end{align*}
Note now that in the last integral only values of $r' \in (|x|-\rho, |x|+\rho)$ give a contribution. Therefore, for $\rho$ sufficiently small, 
we have
\[
\mathcal{H}^1 (\partial D (r') \cap B^{2}_\rho ( x )) 
\leq C \rho, 
\]
for some $C > 0$ independent of $\rho$.
 Thus, 
\begin{align*}
0 &\leq \limsup_{\rho \to 0^+} \frac{\mathcal{H}^{n+1} \left( \Sigma^{u_0} \cap B^{n+1}_\rho ( x , y, t)  \right)}
{\mathcal{H}^{n+1} \left(  B^{n+1}_\rho ( r \omega , y, t)  \right)}
\leq \frac{C}{\omega_{n+1}} \limsup_{\rho \to 0^+}  \frac{\mathcal{H}^n (\{ \alpha_\mu > 0 \} \cap B^{n}_\rho ( |x| , y, t))}{\rho^n} = 0,
 \end{align*}
 where we used the fact that $(|x|, y, t) \in \{ \alpha_\mu = 0 \}^{(1)} = \{ \alpha_\mu > 0 \}^{(0)}$. 
  \end{proof}

\noindent
We now state a useful result, which will be used extensively in the following, and can be obtained by combining together \cite[Theorem~6.2]{CagnettiPeruginiStoger} and Proposition~\ref{giaquinta} (see also \cite{Volpert}).
\begin{proposition}[Vol'pert] \label{volpert theorem}
Let $n \in \mathbb{N}$ with $n \geq 2$, and let $\Omega \subset \mathbb{R}^n$ be open. 
Let $u \in BV_{0, \tau} (\Omega)$, and let $\mu$ and $\alpha_\mu$ be given by \eqref{mu distr} and
\eqref{def alpha}, respectively.
Let $A \subset \subset \Pi_{n-1} (\Omega)$ be open.
Then, there exists a Borel set 
$G_{u_0} \subset \{ 0 < \alpha_\mu   < \pi \} \cap (A \times \R)$ 
with $\mathcal{H}^n ( (\{ 0 < \alpha_\mu   < \pi \} \cap (A \times \R)) \setminus G_{u_0}) = 0$ such that the following properties hold. For every $(r, y, t) \in G_{u_0}$:
\begin{itemize}
\item[(i)] $(\Sigma^{u_0})_{(r,y, t)}$ is a set of finite perimeter in $\partial D (r)$;

\vspace{.2cm}

\item[(ii)] $\partial^* \left( (\Sigma^{u_0})_{(r,y, t)} \right) = ( \partial^{*} \Sigma^{u_0})_{(r,y, t)}$;

\vspace{.2cm}

\item[(iii)] for \textbf{every} $\omega \in \mathbb{S}^{1}$ 
such that $( r \omega , y, t) \in   (\partial^{*} \Sigma^{u_0} )_{(r,y, t)} \cap \partial^* ( (\Sigma^{u_0})_{(r,y, t)} )$:

\vspace{.2cm}

\begin{itemize}
\item[(iiia)]  $| \nu^{\Sigma^{u_0}}_{\!  x {\scriptscriptstyle\parallel}} (r \omega, y, t)| 
= \displaystyle \frac{|\Dx u_0 (r \omega, y)|}{\sqrt{1 +  |\nabla u_0 (r \omega, y)|^2}} >0$;
\vspace{.3cm}
\item[(iiib)]$ \displaystyle 
\nu^{(\Sigma^{u_0})_{(r,y, t)}} (r \omega, y, t)
= \frac{\nu^{\Sigma^{u_0}}_{\!  x {\scriptscriptstyle\parallel}} (r \omega,y, t)}{| \nu^{\Sigma^{u_0}}_{\!  x {\scriptscriptstyle\parallel}} (r \omega,y, t)|} = \displaystyle \frac{\Dx u_0 (r \omega, y)}{|\Dx u_0 (r \omega, y)|}$;  
\vspace{.2cm}
\end{itemize}

\vspace{.1cm}
\end{itemize}
\end{proposition}

\begin{remark}
Note that the set $G_{u_0}$ in general depends on $A$.
\end{remark}

\begin{proof}
The proposition is a direct consequence of the results contained in \cite[Section 2.5]{GMSbook1}.
\end{proof}

\begin{remark} \label{rem codim 1}
The previous proposition implies that 
\begin{equation} \label{measure 0 for the bad set}
\mathcal{H}^n (B_{\Sigma^{u_0}}) = 0,
\end{equation}
where
\[
B_{\Sigma^{u_0}} := \{ 0 < \alpha_\mu  < \pi \} \cap \{ (r, y, t) \in A  \times \mathbb{R} : \, \exists \, \omega \in \mathbb{S}^1 : 
 ( r \omega , y, t) \in \partial^{*} \Sigma^{u_0} \text{ and }
\nu^{\Sigma^{u_0}}_{\!  x {\scriptscriptstyle\parallel}} ( r \omega , y, t)=0 \}.
\]
\end{remark}

\begin{remark} \label{remark us constant functions new}
From Remark~\ref{rem codim 1} it follows that if $(x, y) \in \mathcal{D}_{u_0} \cap \Phi_{n} (A \times \mathbb{S}^1)$ is such that $(x, y, u (x, y)) \in \partial^* \Sigma^{u_0}$ and
\[
\Dx u_0 (x, y)  = 0,
\]
then $(|x|, y, u_0 (x, y)) \notin G_{u_0}$.
\end{remark}

\begin{remark} \label{if higher dim slices}
If one considers the spherical symmetrization of $u$, 
in which the level sets of $u$ are rearranged using spheres of $\R^n$ (rather than circles of $\R^2$), 
a result similar to Proposition~\ref{volpert theorem} holds. 
However, in that case the analogue of identity \eqref{measure 0 for the bad set} is false in general.
This is due to the fact that for the spherical symmetrization property (iii) of Proposition~\ref{volpert theorem} is satisfied in each slice only \textbf{up to an $\mathcal{H}^{n-2}$-negligible set}. 
%
%
\end{remark}
Let us now show a consequence of \eqref{measure 0 for the bad set}.
\begin{proposition} \label{prop flat and perimeter}
Let $n \in \mathbb{N}$ with $n \geq 2$, and let $\Omega \subset \mathbb{R}^n$ be open.
Let $u \in BV_{0, \tau} (\Omega)$, and let $B \subset \subset \Pi_{n-1}(\Omega)\times\mathbb{R}$ be a Borel set.
Then the following statements are equivalent:
\begin{itemize}
\item[(i)] $\mathcal{H}^{n}\left(\Big\{(x, y, t) \in \partial^* \Sigma^{u_0} \cap \Phi_{n+1}
(B \times \mathbb{S}^{1}): \nu^{\Sigma^{u_0}}_{\!  x {\scriptscriptstyle\parallel}} (x, y, t)=0  \Big\}  \right)=0$;
%
\vspace{.2cm}
\item[(ii)] $P(\Sigma^{u_0};\Phi_{n+1}( B' \times \mathbb{S}^{1}))=0$ 
for every Borel set $B' \subset B$, such that $\mathcal{H}^{n} (B')=0$.
\end{itemize}
\end{proposition}
\begin{proof}
We show the two implications.

\vspace{.2cm}
\noindent 
\textbf{Step 1:} (i) $\Longrightarrow$ (ii).
Suppose (i) is satisfied, and let $B' \subset B$ be a Borel set such that $\mathcal{H}^{n} (B')=0$. Then, 
thanks to Proposition~\ref{coarea formula},
\begin{align*}
&P(\Sigma^{u_0};\Phi_{n+1}( B' \times \mathbb{S}^{1})) \\
&= \int_{\partial^*\Sigma^{u_0} \cap \Phi_{n+1}( B' \times \mathbb{S}^{1}) \cap \{ \nu^{\Sigma^{u_0}}_{\!  x {\scriptscriptstyle\parallel}} = 0 \} } 1 \, d \mathcal{H}^{n} (x, y, t) + \int_{\partial^*\Sigma^{u_0} \cap \Phi_{n+1}( B' \times \mathbb{S}^{1}) \cap \{ \nu^{\Sigma^{u_0}}_{\!  x {\scriptscriptstyle\parallel}} \neq 0 \} } 1 \, d \mathcal{H}^{n} (x, y, t) \\
&\leq \mathcal{H}^{n}\left(\Big\{(x, y, t) \in \partial^* \Sigma^{u_0} \cap \Phi_{n+1}
(B \times \mathbb{S}^{1}): \nu^{\Sigma^{u_0}}_{\!  x {\scriptscriptstyle\parallel}} (x, y, t)=0  \Big\}  \right) \\
&\hspace{.2cm}+ \int_{B'} \int_{(\partial^* \Sigma^{u_0} )_{(r, y, t)}} \frac{1}{|\nu^{\Sigma^{u_0}}_{\!  x {\scriptscriptstyle\parallel}} (x, y, t)|} \, d \mathcal{H}^0 (x) \, dr \, dy \, dt = 0, 
\end{align*}
where we used (i) and the fact that $\mathcal{H}^{n} (B') = 0$.

\vspace{.2cm}
\noindent 
\textbf{Step 2:} (ii) $\Longrightarrow$ (i). Suppose (ii) is satisfied. Then, since 
\begin{align*}
&\mathcal{H}^{n}\left(\Big\{(x, y, t) \in \partial^* \Sigma^{u_0} \cap \Phi_{n+1}
( (B \cap \{ \alpha_\mu = 0\}  )\times \mathbb{S}^{1})  \Big\} \right) \\
&\stackrel{(ii)}{=}
\mathcal{H}^{n}\left(\Big\{(x, y, t) \in \partial^* \Sigma^{u_0} \cap \Phi_{n+1}
((B \cap \{ \alpha_\mu = 0\}^{(1)}  ) \times \mathbb{S}^{1})  \Big\} \right)
\stackrel{\eqref{only 0<alpha<pi counts}}{=} 0.
\end{align*}
Similarly, we have 
\begin{align*}
\mathcal{H}^{n}\left(\Big\{(x, y, t) \in \partial^* \Sigma^{u_0} \cap \Phi_{n+1}
( (B \cap \{ \alpha_\mu = \pi \}  )\times \mathbb{S}^{1})  \Big\} \right) = 0.
\end{align*}
Let now $A \subset \subset \Pi_{n-1}(\Omega)$ be an open set such that 
$B \subset A \times \R$, and let $B_{\Sigma^{u_0}}$ be the set defined in Remark~\ref{rem codim 1}.
 Then, thanks to \eqref{measure 0 for the bad set}, we have
\begin{align*}
&\mathcal{H}^{n}\left(\Big\{(x, y, t) \in \partial^* \Sigma^{u_0} \cap \Phi_{n+1}
(B \times \mathbb{S}^{1}): \nu^{\Sigma^{u_0}}_{\!  x {\scriptscriptstyle\parallel}} (x, y, t)=0  \Big\} \right) \\
&= \mathcal{H}^{n}\left(\Big\{(x, y, t) \in \partial^* \Sigma^{u_0} \cap \Phi_{n+1}
( (B\cap \{ 0 < \alpha_\mu < \pi \}) \times \mathbb{S}^{1}): \nu^{\Sigma^{u_0}}_{\!  x {\scriptscriptstyle\parallel}} (x, y, t)=0  \Big\} \right) \\
&\,\,\,\, =  \mathcal{H}^{n}\left(\Big\{(x, y, t) \in \partial^* \Sigma^{u_0} \cap \Phi_{n+1}
(( B \cap (A \times \R) \cap \{ 0 < \alpha_\mu < \pi \}) \times \mathbb{S}^{1}): \nu^{\Sigma^{u_0}}_{\!  x {\scriptscriptstyle\parallel}} (x, y, t)=0  \Big\} \right)\\
&\,\,\,\,  = P(\Sigma^{u_0};\Phi_{n+1}( B \cap B_{\Sigma^{u_0}} \times \mathbb{S}^{1})) 
= 0, 
\end{align*}
 where we used (ii) with $B' = B \cap B_{\Sigma^{u_0}}$. 
\end{proof}

\begin{remark} \label{spherical difficult}
One cannot expect a result analogous to Proposition~\ref{prop flat and perimeter} to hold for the spherical rearrangement of $u$.
Indeed, to prove the implication (ii) $\Longrightarrow$ (i) above we have used identity \eqref{measure 0 for the bad set},
which fails for the spherical symmetrization (see Remark~\ref{if higher dim slices}).
\end{remark}
\subsection{Further properties of $\mu$, $v_\mu$, and $\Sigma^{v_\mu}$}
In this section we show some important properties of the distribution $\mu$, and how these affect the function $v_\mu$ 
and the set $\Sigma^{v_\mu}$. 
We start by showing that the circular rearrangement preserves the $L^p$ norm, provided condition (b) of Definition~\ref{def w1p0tau} is satisfied. 
\begin{proposition} \label{prop Lp norm}
Let $n \in \mathbb{N}$ with $n \geq 2$, let $\Omega \subset \mathbb{R}^n$ be open, and let $p \in [1, \infty)$. 
Let $u: \Omega \to \R$ be a Lebesgue measurable function such that $u_0 \in L^p_{\textnormal{loc}} (\Phi_{n}  (\Pi_{n-1} (\Omega)  \times \mathbb{S}^{1}))$, 
where $u_0$ is given by \eqref{def u0}. Let $\mu$ and $v_{\mu}$ be defined by \eqref{mu distr} and \eqref{def vmu}, respectively. 
Then, for every open set $A \subset \subset \Pi_{n-1} (\Omega)$ we have $v_{\mu} \in L^p (\Phi_n (A \times \mathbb{S}^1) )$, with \( \| v_{\mu} \|_{L^p (\Phi_n (A \times \mathbb{S}^1) )}
= \| u_0 \|_{L^p (\Phi_n (A \times \mathbb{S}^1) )} \). 
\end{proposition}

\begin{proof}
Let $t \geq 0$.
We have 
\begin{align*}
&\mathcal{H}^n (\{ (v_\mu)_+ >  t \} \cap \Phi_n (A \times \mathbb{S}^1))) 
= \mathcal{H}^n (\{ v_\mu  >   t \} \cap \Phi_n (A \times \mathbb{S}^1))) \\
&= \int_{A}  \mathcal{H}^1 (\{ v_\mu (\cdot, y)  >  t \} \cap \partial D (r)) \, dr \, dy 
= \int_{A} \mu (r, y, t) \, dr \, dy \\
&= \mathcal{H}^n (\{ (u_0)_+ >  t \} \cap \Phi_n (A \times \mathbb{S}^1))),
\end{align*}
where the last equality follows by applying backward the previous equalities to the function $u$. 
Similarly,   
\begin{align*}
&\mathcal{H}^n (\{ (v_\mu)_- \geq  t \} \cap \Phi_n (A \times \mathbb{S}^1)) 
= \mathcal{H}^n (\{ v_\mu  \leq -  t \} \cap \Phi_n (A \times \mathbb{S}^1)) \\
&= \int_{A} \mathcal{H}^1 (\{ v_\mu (\cdot, y) \leq -  t \} \cap \partial D (r)) \, dr \, dy 
= \int_{A} \left[  2 \pi r - \mathcal{H}^1 (\{ v_\mu (\cdot, y)  > - t \} \cap \partial D (r)) \right]\, dr \, dy \\
&= \int_{A} \left[  2 \pi r - \mu (r, y, -t) \right]\, dr \, dy
= \mathcal{H}^n (\{ (u_0)_- \geq  t \} \cap \Phi_n (A \times \mathbb{S}^1))).
\end{align*}
Thanks to the Layer-cake formula, we then obtain 
\begin{align*}
&\| v_\mu \|_{L^p (\Phi_n (A \times \mathbb{S}^1))} 
= \| (v_\mu)_+ \|_{L^p (\Phi_n (A \times \mathbb{S}^1))} + \| (v_\mu)_- \|_{L^p (\Phi_n (A \times \mathbb{S}^1))} \\
&= \| (u_0)_+ \|_{L^p (\Phi_n (A \times \mathbb{S}^1))} + \| (u_0)_- \|_{L^p (\Phi_n (A \times \mathbb{S}^1))} 
= \| u_0 \|_{L^p (\Phi_n (A \times \mathbb{S}^1))}.
\end{align*}
\end{proof}
In the next proposition we show the connection between the circular rearrangement of a function 
and the circular symmetrization of its subgraph.

\begin{proposition}\label{prop_circular rear and circular sym are connected bis}
Let $n \in \mathbb{N}$ with $n \geq 2$, and let $\Omega \subset \mathbb{R}^n$ be open. 
Let $u: \Omega \to \R$ be a Lebesgue measurable function such that $u_0 \in L^1_{\textnormal{loc}} (\Phi_{n}  (\Pi_{n-1} (\Omega)  \times \mathbb{S}^{1}))$, 
where $u_0$ is given by \eqref{def u0}.
Let $\mu$, $v_\mu$, and $\alpha_\mu$ be given by \eqref{mu distr}, \eqref{def vmu} and \eqref{def alpha}, respectively, 
and let $F_\mu \subset \R^{n+1}$ be defined as
\begin{equation} \label{def Fmu}
F_\mu := \{ (x, y, t) \in \Phi_{n}  (\Pi_{n-1} (\Omega)  \times \mathbb{S}^{1}) \times \R : d_{\mathbb{S}^1}  (\hat{x} , e_1) < \alpha_\mu ( | x |, y, t) \}. 
\end{equation}
Then,
\begin{equation} \label{equivalence}
\Sigma^{v_\mu} =_{\mathcal{H}^{n+1}} F_\mu.
\end{equation}
\end{proposition}

\begin{proof}
Let $(x, y, t) \in \Phi_{n}  (\Pi_{n-1} (\Omega)  \times \mathbb{S}^{1}) \times \R$. 
Since the hyperplane $\{ (x_1, x_2, y_1, \ldots, y_{n-2}, t) \in \R^{n+1} : x_2 = 0 \}$
is $\mathcal{H}^{n+1}$-negligible, we can assume \( 0< d_{\mathbb{S}^1}  (\hat{x} , e_1) < \pi \). 
By definition of $v_\mu$, we have 
\[
v_\mu (x, y) = \inf A (x, y) \quad \text{ where } \quad A (x, y):= \left\{ s \in \mathbb{R} : \alpha_\mu (|x|, y, s) \leq   d_{\mathbb{S}^1}  (\hat{x} , e_1)   \right\}.
\]
Thanks to Remark~\ref{mu(t) right-cont}, the function $s \mapsto \mu (|x|, y, s)$ is non-increasing and right-continuous, 
and so is $s \mapsto \alpha_\mu (|x|, y, s)$. 
Therefore, since \( 0< d_{\mathbb{S}^1}  (\hat{x} , e_1) < \pi \), we have that $A (x, y)$ is a closed half-line which is unbounded from above and bounded from below, and
\[
A (x, y) = [ v_{\mu} (x, y), + \infty) \quad \text{ and } \quad \alpha_\mu (|x|, y, v_{\mu} (x, y)) \leq   d_{\mathbb{S}^1}  (\hat{x} \cdot e_1).
\]
Thus, 
\[
v_{\mu} (x, y) > t 
\quad \Longleftrightarrow \quad
t \notin A (x, y)
\quad \Longleftrightarrow \quad
\alpha_\mu (|x|, y, t) >   d_{\mathbb{S}^1}  (\hat{x} \cdot e_1).
\]
From this, it follows that 
\[
(x, y, t) \in \Sigma^{v_{\mu}}
\quad \Longleftrightarrow \quad
v_{\mu} (x, y) > t 
\quad \Longleftrightarrow \quad
\alpha_\mu (|x|, y, t) >   d_{\mathbb{S}^1}  (\hat{x} \cdot e_1)
\quad \Longleftrightarrow \quad
(x, y, t) \in F_{\mu}.
\]

%
\end{proof}

The next result should be compared with Proposition~\ref{prop several useful things}.
\begin{proposition} \label{prop several useful things 2}
Let $n \in \mathbb{N}$ with $n \geq 2$, and let $\Omega \subset \mathbb{R}^n$ be open. 
Let $u_1, u_2: \Omega \to \R$ be Lebesgue measurable functions such that $u_{1,0}, u_{2,0} \in L^1_{\textnormal{loc}} (\Phi_{n}  (\Pi_{n-1} (\Omega)  \times \mathbb{S}^{1}))$, where $u_{1,0}$ and $u_{2,0}$ are given by \eqref{def u0} with $u_1$ and $u_2$ in place of $u$, respectively.
Let $\mu_{1}$ and $\mu_{2}$ be defined by \eqref{mu distr} with $u_1$ and $u_2$ in place of $u$, respectively. 
Let now $v_{\mu_1}, v_{\mu_2}$ and $F_{\mu_1}, F_{\mu_2}$ be defined by \eqref{def vmu} and by \eqref{def Fmu}, 
with $\mu_{1}$ and $\mu_{2}$ in place of $\mu$, respectively. Then, the following are equivalent:
\begin{enumerate}

\item[(a)] \( \mu_{1} = \mu_{2} \) $\mathcal{H}^n$-a.e. in $\Pi_{n-1} (\Omega) \times \R$;

\vspace{.2cm}

\item[(b)] $v_{\mu_{1}} = v_{\mu_{2}}$ $\mathcal{H}^n$-a.e. in $\Phi_{n}  (\Pi_{n-1} (\Omega)  \times \mathbb{S}^{1}))$;

\vspace{.2cm}

\item[(c)] $\Sigma^{v_{\mu_{1}}} =_{\mathcal{H}^{n+1}} \Sigma^{v_{\mu_{2}}}$ in $\Phi_{n}  (\Pi_{n-1} (\Omega)  \times \mathbb{S}^{1})) \times \R$;

\vspace{.2cm}

\item[(d)] $F_{\mu_{1}} =_{\mathcal{H}^{n+1}} F_{\mu_{2}}$ in $\Phi_{n}  (\Pi_{n-1} (\Omega)  \times \mathbb{S}^{1})) \times \R$.

\end{enumerate}

\end{proposition}

\begin{proof}
Let $A \subset \subset \Pi_{n-1} (\Omega)$ be open.
Applying \eqref{two equalities} to the two functions $v_{\mu_1}$ and $v_{\mu_2}$, we obtain
\begin{equation} \label{two equalities for v}
\begin{split} 
&\| v_{\mu_1} - v_{\mu_2} \|_{L^1 (\Phi_{n}  (A  \times \mathbb{S}^{1}))}  \\
&= \mathcal{H}^{n+1} \Big( (\Sigma^{v_{\mu_1}} \Delta \Sigma^{v_{\mu_2}}) \cap \big( \Phi_{n+1}  ((A \times \R) \times \mathbb{S}^{1})) \big) \Big) 
 \\
&= \int_{A \times \R }
\mathcal{H}^1 \Big( (\partial D (r) \cap \{ v_{\mu_1} (\cdot, y) > t \} )\Delta (\partial D (r) \cap \{ v_{\mu_2} (\cdot, y) > t \}) \Big) \, dr \, dy \, dt \\
&= \int_{A \times \R }
\mathcal{H}^1 ( \mathbf{B}_{\alpha_{\mu_1}} ( r e_1) \Delta \mathbf{B}_{\alpha_{\mu_2}} ( r e_1) ) \, dr \, dy \, dt
= \| \mu_1 - \mu_2 \|_{L^1 (A  \times \mathbb{R})}.
\end{split}
\end{equation}
This shows that $\mu_1 - \mu_2 \in L^1 (A  \times \mathbb{R})$ (even though, thanks to Proposition~\ref{prop_circular rear and circular sym are connected}, in general we only have $\mu_1, \mu_2 \in L^1 (A  \times (-d, + \infty))$ for every $d > 0$).
Since $A \subset \subset \Pi_{n-1} (\Omega)$ is arbitrary, from \eqref{two equalities for v} it follows that (a), (b), and (c) are equivalent.
Finally, from \eqref{equivalence} we conclude that (c) is equivalent to (d). 
\end{proof}
\begin{proposition} \label{prop v_mu zero}
Let $n \in \mathbb{N}$ with $n \geq 2$, let $\Omega \subset \mathbb{R}^n$ be open. 
Let $u: \Omega \to \R$ be a measurable function satisfying property (b) of Definition~\ref{def w1p0tau}, 
such that $u_{0} \in L^1_{\textnormal{loc}} (\Phi_{n}  (\Pi_{n-1} (\Omega)  \times \mathbb{S}^{1}))$, where $u_{0}$ is given by \eqref{def u0}.
Let $\mu$ and $v_{\mu}$ be defined by \eqref{mu distr} and \eqref{def vmu}, respectively. 
Then, 
\begin{equation} \label{us is zero a.e. outside}
v_{\mu} (x, y) = 0 \quad \text{ for $\mathcal{H}^n$-a.e. } (x, y) \in \Phi_n (\Pi_{n-1} (\Omega) \times \mathbb{S}^1) \setminus \Omega^s.
\end{equation}
As a consequence, 
\begin{equation} \label{v_mu is zero a.e. outside}
(v_{\mu}   \! \mid_{\Omega^s} )_0 = v_{\mu},
\end{equation}
where $(v_{\mu}  \! \mid_{\Omega^s} )_0$ is given by \eqref{def u0} with $v_\mu$ in place of $u$.
\end{proposition}

\begin{proof}
Note that, by definition of $\Pi_{n-1}^{\textnormal{a}}(\Omega)$ (see \eqref{def_Pi_a}), we have that 
\[
\mathcal{L}^n (\Phi_n (\Pi_{n-1}^{\textnormal{a}}(\Omega) \times \mathbb{S}^1) \setminus \Omega^s) = 0.
\]
Therefore, to show \eqref{us is zero a.e. outside} it will be enough to prove that
\begin{equation*} 
v_{\mu} (x, y) = 0 \quad \text{ for $\mathcal{H}^n$-a.e. } (x, y) \in \Phi_n ((\Pi_{n-1} (\Omega) \setminus \Pi_{n-1}^{\textnormal{a}}(\Omega) )\times \mathbb{S}^1) \setminus \Omega^s.
\end{equation*}
Let now $(x, y) \in \Phi_n ((\Pi_{n-1} (\Omega) \setminus \Pi_{n-1}^{\textnormal{a}}(\Omega) )\times \mathbb{S}^1) \setminus \Omega^s$.
Since the hyperplane \( \{ (x, y) \in \R^n : x_2 = 0\} \) is $\mathcal{H}^n$-negligible, we can assume that $0< d_{\mathbb{S}^1} (\hat{x} , e_1)  < \pi$.
Then, we have $(|x|, y) \notin \Pi_{n-1}^{\textnormal{a}}(\Omega)$ and
\[
2 \pi |x| > 2 |x|  d_{\mathbb{S}^1} (\hat{x} , e_1) \geq \mathcal{H}^1 ( (\Omega)_{(|x|, y)}).
\]
Thanks to Proposition~\ref{prop several useful things} and Proposition~\ref{prop several useful things 2}, if we modify $u$ on a set of $\mathcal{H}^n$-measure zero, this will only affect $v_\mu$ on a $\mathcal{H}^n$-negligible set. 
Therefore, thanks to property (b) of Definition~\ref{def w1p0tau}, we can assume that
\begin{equation} \label{frte}
u \geq 0 \quad  \text{ everywhere in $\Omega \setminus 
\Phi_n(\Pi_{n-1}^{\textnormal{a}}(\Omega) \times\mathbb{S}^1)$.}
\end{equation}
Then, since $u_0 = 0$ outside of $\Omega$, for every $t > 0$ we have 
\begin{align*}
\mu (|x|, y, t) &= \mathcal{H}^1 ( \{ u_0 (\cdot, y)> t \} \cap \partial D (|x|)) \\
&= \mathcal{H}^1 ( \{ u (\cdot, y) > t \} \cap  \Omega \cap \partial D (|x|)) \\
&\leq \mathcal{H}^1 ( (\Omega)_{(|x|, y)}) 
\leq 2 |x|  d_{\mathbb{S}^1} (\hat{x} , e_1). 
\end{align*}
On the other hand, thanks to \eqref{frte}, for every $t < 0$ we have
\[
\mu (|x|, y, t) 
= \mathcal{H}^1 ( \{ u_0 (\cdot, y) > t \} \cap \partial D (|x|)) = 2 \pi |x|  > 2 |x| d_{\mathbb{S}^1} (\hat{x} , e_1).
\]
Thus, recalling \eqref{def vmu},
 \[
v_{\mu} (x, y) = \inf \left\{ t \in \mathbb{R} : \mu (|x|, y, t) \leq 2 |x|  d_{\mathbb{S}^1} (\hat{x} , e_1)  \right\} = 0. 
\]
\end{proof}
The next result gives a refinement of Proposition~\ref{ssssd} in the special case in which one considers the function
$v_\mu$, and will be used in the proof of The P\'olya--Szeg\"o inequality.
\begin{proposition} \label{ssssd2}
Let $n \in \mathbb{N}$ with $n \geq 2$, and let $\Omega \subset \mathbb{R}^n$ be open. 
Let $u \in BV_{0, \tau} (\Omega)$, and let $\mu$, $v_{\mu}$, and $\alpha_\mu$  be given by \eqref{mu distr}, \eqref{def vmu}, and \eqref{def alpha}, respectively. 
Then, 
\begin{align*}
(r, y, t) \in \{ \alpha_\mu^{\vee} = 0 \} 
&\quad \Longrightarrow \quad
(r \omega, y, t) \in (\Sigma^{v_\mu})^{(0)} \text{ for every } \omega \in \mathbb{S}^1 \setminus \{ e_1 \}; \\
(r, y, t) \in \{ \alpha_\mu^{\wedge} = \pi \} 
&\quad \Longrightarrow \quad
(r \omega, y, t) \in (\Sigma^{v_\mu})^{(1)} \text{ for every } \omega \in \mathbb{S}^1 \setminus \{ - e_1 \}.
\end{align*}
Moreover, 
\begin{align} \label{nec for reduced bdry 1}
\mathcal{H}^n \left(  \partial^* \Sigma^{v_\mu} \cap \left( \Phi_{n+1} \left(
\{ \alpha_\mu^{\vee} = 0 \}  \times \mathbb{S}^{1}
\right) \right)
\right) = 0,
\end{align}
and 
\begin{align}  \label{nec for reduced bdry 2}
\mathcal{H}^n \left(  \partial^* \Sigma^{v_\mu} \cap \left( \Phi_{n+1} \left(
\{ \alpha_\mu^{\wedge} = \pi \}  \times \mathbb{S}^{1}
\right) \right)
\right) = 0.
\end{align}
As a consequence, for every Borel set $B \subset \Pi_{n-1} (\Omega) \times \R$ we have
\begin{equation} \label{the only parts that count}
\begin{split} 
&\mathcal{H}^n \left(  \partial^* \Sigma^{v_\mu} \cap \left( \Phi_{n+1} \left(
B  \times \mathbb{S}^{1}
\right) \right)
\right)   \\
&= \mathcal{H}^n \left(  \partial^* \Sigma^{v_\mu} \cap \left( \Phi_{n+1} \left(
(B \cap \{ \alpha_\mu^{\vee} > 0 \} \cap \{ \alpha_\mu^{\wedge} < \pi \})
  \times \mathbb{S}^{1}
\right) \right)
\right).  
\end{split}
\end{equation}
\end{proposition}

\begin{proof}
We will just show the first implication and \eqref{nec for reduced bdry 1}, since the second implcation and \eqref{nec for reduced bdry 2} can be proved in a similar way.
We proceed by steps.

\vspace{.2cm}

\noindent
\textbf{Step 1:} We show that
\begin{align*}
(r, y, t) \in \{ \alpha_\mu^{\vee} = 0 \} 
\quad \Longrightarrow \quad
(r \omega, y, t) \in (\Sigma^{v_\mu})^{(0)} \text{ for every } \omega \in \mathbb{S}^1 \setminus \{ e_1 \}.
\end{align*}
Let $(r, y, t) \in \{ \alpha_\mu^{\vee} = 0 \}$, 
and let $\omega \in \mathbb{S}^1 \setminus \{ e_1 \}$. We set 
\begin{equation} \label{ineq with delta}
\delta:= \arccos (\omega \cdot e_1) > 0.
\end{equation}
Note that
\begin{align}
(x' ,  y' ,  t') \in \Sigma^{v_\mu} \cap B^{n+1}_\rho ( r \omega , y, t) 
\quad \Longrightarrow \quad (|x'| ,  y' ,  t') \in \{ \alpha_\mu > \delta/2 \}
\quad \text{ if } 0< \rho \ll 1.
 \label{rt2}
\end{align}
Indeed, if $(x' ,  y' ,  t') \in \Sigma^{v_\mu} \cap B^{n+1}_\rho ( r \omega , y, t)$, 
we have $(x' ,  y' ,  t') = ( r \omega , y, t) + \rho (\overline{x}, \overline{y}, \overline{t})$,
for some $(\overline{x}, \overline{y}, \overline{t}) \in B^{n+1}_1$. Therefore, thanks to \eqref{ineq with delta}, for $\rho$ sufficiently small
\begin{align*}
&\alpha_\mu ( |x'|  ,  y' ,  t' ) > \arccos \left( \frac{x'}{|x'|} \cdot e_1\right) = \arccos \left(\frac{(r \omega + \rho \overline{x}) \cdot e_1 }{|r \omega + \rho \overline{x}|} \right) > \frac{\delta}{2},
\end{align*}
so that \eqref{rt2} holds.
Then, using the same argument of the proof of Proposition~\ref{ssssd}, we obtain that for $0< \rho \ll 1$
\begin{align*}
&\mathcal{H}^{n+1} \left( \Sigma^{v_\mu} \cap B^{n+1}_\rho ( r \omega , y, t)  \right)
=  \int_{\mathbb{R}^{n+1}} \chi_{\Sigma^{v_\mu} \cap B^{n+1}_\rho ( r \omega , y, t)} (x' ,  y' ,  t')
\, d x' \, d y' \, d t' \\
&= \int_{\Pi_{n-1} (\Omega) \times \R} 
\left( 
 \int_{ \partial D (r')} \chi_{B_\rho^2 (r\omega)}(x') \chi_{\Sigma^{v_\mu} \cap B^{n+1}_\rho ( r \omega , y, t)} (x' ,  y' ,  t') \, d \mathcal{H}^1 (x')  
 \right) 
 \, d r' \, d y'  \, d t' \\
 &\leq  \int_{\Pi_{n-1} (\Omega) \times \R} 
\left( 
 \int_{ \partial D (r')} \chi_{B_\rho^2 (r\omega)}(x') \chi_{\{ \alpha_\mu > \delta/2 \} \cap B^{n}_\rho ( r , y, t)} (r',  y' ,  t') \, d \mathcal{H}^1 (x')  
 \right) 
 \, d r' \, d y'  \, d t' \\
 &\leq \int_{\Pi_{n-1} (\Omega) \times \R} \mathcal{H}^1 \left( \partial D (r') \cap B_\rho^2 (r\omega)   \right)
 \chi_{\{ \alpha_\mu > \delta/2 \} \cap B^{n}_\rho ( r , y, t)} (r',  y' ,  t') \, d \mathcal{H}^1 (x')  
  \, d r' \, d y'  \, d t' \\
 &\leq C \rho \mathcal{H}^n (\{ \alpha_\mu > \delta/2 \} \cap B^{n}_\rho ( r , y, t)),
 \end{align*}
 for some $C > 0$ independent of $\rho$.
Therefore, 
\[
0 \leq \limsup_{\rho \to 0^+} \frac{\mathcal{H}^{n+1} \left( \Sigma^{v_\mu} \cap B^{n+1}_\rho ( r \omega , y, t)  \right)}
{\mathcal{H}^{n+1} \left(  B^{n+1}_\rho ( r \omega , y, t)  \right)} 
\leq  \frac{C}{\omega_{n+1}} \limsup_{\rho \to 0^+}
\frac{\mathcal{H}^n (\{ \alpha_\mu > \delta/2 \} \cap B^{n}_\rho ( r , y, t))}{ \rho^{n}} = 0, 
\]
where we used the fact that $\alpha^{\vee} ( r , y, t) = 0$.

\vspace{.2cm}

\noindent
\textbf{Step 2:} We show \eqref{nec for reduced bdry 1}.
Thanks to \eqref{only 0<alpha<pi counts} and Step 1, if $(x, y, t) \in \partial^* \Sigma^{v_\mu}$ and $(|x|, y, t) \in \{ \alpha_\mu^{\vee} = 0 \} $, with $x = (x_1, x_2)$, then
\[
x_1 =  |x|, \, \, x_2 = 0,  \quad \text{ and } \quad
(x_1, y, t) \in \big(\Pi_{n-1} (\Omega) \times \R \big) \cap \big( \{ \alpha_\mu^{\vee} = 0 \} \setminus \{ \alpha_\mu = 0 \}^{(1)}\big).
\]
Therefore,
\begin{align*}
&\mathcal{H}^n \left(  \partial^* \Sigma^{v_\mu} \cap \left( \Phi_{n+1} \left(
\{ \alpha_\mu^{\vee} = 0 \}  \times \mathbb{S}^{1}
\right) \right)
\right) \leq 
 \int_{ \big(\Pi_{n-1} (\Omega) \times \R \big) \cap \big( \{ \alpha_\mu^{\vee} = 0 \} \setminus \{ \alpha_\mu = 0 \}^{(1)}\big)}  
 1 \, d \mathcal{H}^n (x_1, y, t) \\
  &\leq \mathcal{H}^n \big( \{ \alpha_\mu^{\vee} = 0 \} \setminus \{ \alpha_\mu = 0 \}^{(1)}\big)
  = \mathcal{H}^n \big( \{ \alpha_\mu^{\vee} = 0 \} \setminus \{ \alpha_\mu^{\vee} = 0 \}^{(1)}\big) = 0.
\end{align*}
\end{proof}
In the next proposition we give the explicit expression of the distributional derivatives 
of $\mu$ and $\xi_\mu$, see \cite[Lemma~6.7]{CagnettiPeruginiStoger} and \cite[Lemma 3.6]{PeruginiCircolare} for related results.
\begin{proposition} \label{derivatives of mu}
Let $n \in \mathbb{N}$ with $n \geq 2$, and let $\Omega \subset \mathbb{R}^n$ be open. 
Let $u \in BV_{0, \tau} (\Omega)$, and let $\mu$ and $\alpha_\mu$ be given by \eqref{mu distr} and
\eqref{def alpha}, respectively, and let $\xi_\mu: \Pi_{n-1} (\Omega) \times \R \to [0, 2 \pi]$ be given by
\begin{equation} \label{def xi}
\xi_{\mu} (r, y, t) := \frac{\mu (r, y, t)}{r} = 2 \alpha_{\mu} (r, y, t), \quad \text{ for every }  (r, y, t) \in \Pi_{n-1} (\Omega) \times \R.
\end{equation}
Let $A \subset \subset \Pi_{n-1} (\Omega)$ be open and let $d > 0$.
Then, $\Sigma^{u_0}$ is a set of finite perimeter in $\Phi_{n+1} ((A \times \R) \times \mathbb{S}^{1})$, 
and $\mu, \xi_{\mu}, \alpha_\mu \in BV (A \times (-d, + \infty))$.
Moreover, for every Borel set $B \subset A \times (-d, + \infty)$ 
and for every bounded Borel function $\varphi: B \to \mathbb{R}$ we have
\begin{align}
&\int_{B} \varphi (r, y, t) \, d D_y \mu (r, y, t)
= \int_{\partial^* \Sigma^{u_0} \cap \left( \Phi_{n+1} (B \times \mathbb{S}^{1}) \right)}  
\varphi (|x| , y, t) \, \nu_y^{\Sigma^{u_0}} (x, y, t) \, d \mathcal{H}^n (x, y, t), \label{Dy mu} \\
&\int_{B} \varphi (r, y, t) \, d D_t \mu (r, y, t)
= \int_{\partial^* \Sigma^{u_0} \cap \left( \Phi_{n+1} (B \times \mathbb{S}^{1}) \right)}  
\varphi (|x| , y, t) \, \nu_t^{\Sigma^{u_0}} (x, y, t) \, d \mathcal{H}^n (x, y, t), \label{Dt mu} \\
&\int_{B} \varphi (r, y, t) r d D_r \xi_\mu (r, y, t) 
\hspace{-.05cm}= 
\hspace{-.1cm} \int_{\partial^* \Sigma^{u_0} \cap \left( \Phi_{n+1} (B \times \mathbb{S}^{1}) \right)}  
\varphi (|x| , y, t) \, \hat{x} \hspace{-.05cm} \cdot \hspace{-.05cm} \nu^{\Sigma^{u_0}}_x (x, y, t) \, d \mathcal{H}^n (x, y, t). \label{rDr xi}
\end{align} 
\end{proposition}

\begin{remark} \label{alpha = 0 and alpha = pi}
Applying the previous result to the $\mu$-distributed function $v_\mu$, thanks to Proposition~\ref{ssssd2}, 
we obtain that
\begin{align*}
D_y \mu &= D_y \mu \mres \{ \alpha^{\vee} > 0 \} \cap \{ \alpha^{\wedge} < \pi \}; \\
D_t \mu &= D_t \mu \mres \{ \alpha^{\vee} > 0 \} \cap \{ \alpha^{\wedge} < \pi \}; \\
r D_r \xi_\mu &= r D_r \xi_\mu \mres \{ \alpha^{\vee} > 0 \} \cap \{ \alpha^{\wedge} < \pi \}.
\end{align*}
\end{remark}

\begin{proof}[Proof of Proposition~\ref{derivatives of mu}]
By definition of $BV_{0, \tau} (\Omega)$ we have that $u_0 \in BV (\Phi_n (A \times \mathbb{S}^{1}))$
and so, thanks to Proposition~\ref{giaquinta}, $\Sigma^{u_0}$ is a set of finite perimeter in 
$\Phi_n (A \times \mathbb{S}^{1}) \times \mathbb{R} = \Phi_{n+1} ((A \times \R) \times \mathbb{S}^{1})$.
Moreover, \( \mu \in L^1 (A \times (-d, + \infty)) \) thanks to Proposition~\ref{prop_circular rear and circular sym are connected}.
We now divide the remaining part of the proof in several steps.

\vspace{.2cm}

\noindent
\textbf{Step 1:} We show that formulas 
\eqref{Dy mu} and \eqref{Dt mu}
hold in the special case $B=A \times (-d, + \infty)$ and $\varphi \in C^1_c (A \times (-d, + \infty))$. 
Concerning \eqref{Dt mu}, we have
\begin{align*}
& \int_{A \times (-d, + \infty)} \varphi (r, y, t) \, d D_t \mu (r, y, t)  
= - \int_{A \times (-d, + \infty)}  \mu (r, y, t ) \frac{\partial \varphi}{\partial t} (r, y, t) \, dr dy dt \\
&= - \int_{\Phi_{n+1} ((A \times \R) \times \mathbb{S}^{1})} \! \! \chi_{\Sigma^{u_0}} (x, y, t) 
\frac{\partial \varphi}{\partial t} (|x| , y, t)\, dx \, dy \, dt \\
&=  \int_{\Phi_{n+1} ((A \times \R) \times \mathbb{S}^{1})}
\varphi (|x| , y,  t) \, d D_t \chi_{\Sigma^{u_0}} (x, y, t) \\
&=  \int_{\partial^* \Sigma^{u_0} \cap \left( 
\Phi_{n+1} ((A \times (-d, + \infty)) \times \mathbb{S}^{1}) 
\right)}  
\varphi (|x| , y, t) \, \nu_t^{\Sigma^{u_0}} (x, y, t) \, d \mathcal{H}^n (x, y, t).
\end{align*}
One can argue in a similar way for identity \eqref{Dy mu}.

\vspace{.2cm}

\noindent
\textbf{Step 2:} We show that for every $\varphi \in C^1_c (A \times (-d, + \infty))$
\begin{equation} \label{rDr xi 2}
\begin{split} 
&\int_{A \times (-d, + \infty)} \varphi (r, y, t) \, d D_r \mu (r, y, t) 
= \int_{A \times (-d, + \infty)}  \varphi (r, y, t) \xi_{\mu} (r, y, t) \, dr \, dy \, dt  \\
&\hspace{.2cm}+
\int_{\partial^* \Sigma^{u_0} \cap 
\Phi_{n+1} ((A \times (-d, + \infty)) \times \mathbb{S}^{1}) }   
\varphi (|x| , y, t) \, \hat{x} \cdot \, \nu^{\Sigma^{u_0}}_x (x, y, t) \, d \mathcal{H}^n (x, y, t).
\end{split}
\end{equation}
Let $\varphi \in C^1_c (A \times (-d, + \infty))$. Then, a direct calculation gives
\[
\text{div}_{(x, y, t)} \big( \varphi (|x|, y, t) \hat{x} \big)
= \frac{\partial \varphi}{\partial r} (| x | , y, t) + \frac{1}{|x|} \varphi (|x|, y, t), 
\qquad \forall \, (x, y, t)\in \Phi_{n+1} ((A \times \R) \times \mathbb{S}^{1}),
\]
where $\text{div}_{(x, y, t)}$ denotes the divergence 
of a vector field in $\mathbb{R}^{n+1}$ with respect to the variables $(x, y, t)$.
Therefore, 
\begin{align*}
& \int_{A \times (-d, + \infty)} \varphi (r, y, t) \, d D_r \mu (r, y, t) 
=- \int_{A \times (-d, + \infty)}  \mu (r, y, t ) \frac{\partial \varphi}{\partial r} (r, y, t) \, dr \, dy \, dt   \\
&= - \int_{A \times \R}  \mu (r, y, t ) \frac{\partial \varphi}{\partial r} (r, y, t) \, dr \, dy \, dt 
= - \int_{\Phi_{n+1} ((A \times \R) \times \mathbb{S}^{1})} \! \! \chi_{\Sigma^{u_0}} (x, y, t) 
\frac{\partial \varphi}{\partial r} (|x| , y, t)\, dx \, dy \, dt  \\
&= - \int_{\Phi_{n+1} ((A \times \R) \times \mathbb{S}^{1})} \! \! \chi_{\Sigma^{u_0}} (x, y, t)  
\left(  \text{div}_{(x, y, t)} \Big( \varphi (|x|, y, t) \hat{x} \Big)
-  \frac{1}{|x|} \varphi (|x|, y, t) \right) \, dx \, dy \, dt  \\
&= \int_{\Phi_{n+1} ((A \times \R) \times \mathbb{S}^{1})} 
\hspace{-.2cm} \varphi (|x| , y, t) \, \hat{x} \cdot \, d D \chi_{\Sigma^{u_0}} (x, y, t) 
+  \int_{\Phi_{n+1} ((A \times \R) \times \mathbb{S}^{1})} 
\hspace{-.2cm} \chi_{\Sigma^{u_0}} (x, y, t) 
\frac{1}{|x|} \varphi (|x|, y, t)  \, dx \, dy \, dt  \\
&= \int_{\partial^* \Sigma^{u_0} \cap \left( 
\Phi_{n+1} ((A \times (-d, + \infty)) \times \mathbb{S}^{1}) 
\right)}     
\hspace{-.2cm} \varphi (|x| , y, t) \, \hat{x} \cdot \, \nu^{\Sigma^{u_0}}_x (x, y, t) \, d \mathcal{H}^n (x, y, t)  \\
&+\int_{A \times (-d, + \infty)} \xi_{\mu} (r, y, t) \varphi (r, y, t) \, dr \, dy \, dt,
\end{align*}
which gives \eqref{rDr xi 2}. 

\vspace{.2cm}

\noindent
\textbf{Step 3:} We show that $D_y \mu$, $D_t \mu$ and $D_r \mu$ are bounded Radon measures on $A \times (-d, + \infty)$.
From Step 2, we know that \eqref{Dt mu} holds with $B=A \times (-d, + \infty)$ and $\varphi \in C^1_c (A \times (-d, + \infty))$. 
Taking the supremum over all $\varphi \in C^1_c (A \times (-d, + \infty))$ with $| \varphi | \leq 1$
we obtain that $D_t \mu$ is a bounded Radon measure on $A \times (-d, + \infty)$, with
\begin{align*}
| D_t \mu | (A \times (-d, + \infty))  
&\leq \mathcal{H}^n 
\left( \partial^* \Sigma^{u_0} \cap \Phi_{n+1} ((A \times \R) \times \mathbb{S}^{1}) \right) 
< \infty.
\end{align*}
One can argue in a similar way for $D_y \mu$. 

We are left to show that $D_r \mu$ 
is a bounded Radon measure on $A \times (-d, + \infty)$. 
Let $c > 0 $ be such that $r \geq c$ for every $(r, y) \in A$. Thanks to \eqref{rDr xi 2},
for every $\varphi \in C^1_c (A \times (-d, + \infty))$ with $| \varphi | \leq 1$ 
\begin{align*}
&\int_{A \times (-d, + \infty)} \varphi (r, y, t) \, d D_r \mu (r, y, t) 
= \int_{A \times (-d, + \infty)}  \varphi (r, y, t) \frac{\mu (r, y, t)}{r} \, dr \, dy \, dt  \\
&\hspace{.2cm}+
\int_{\partial^* \Sigma^{u_0} \cap 
\Phi_{n+1} ((A \times (-d, + \infty)) \times \mathbb{S}^{1}) }   
\varphi (|x| , y, t) \, \hat{x} \cdot \, \nu^{\Sigma^{u_0}}_x (x, y, t) \, d \mathcal{H}^n (x, y, t) \\
&\leq \frac{1}{c} \| \mu \|_{L^1 \left( A \times (-d, + \infty)  \right) } 
+ \mathcal{H}^n \left(  \partial^* \Sigma^{u_0} \cap 
\Phi_{n+1} ((A \times (-d, + \infty)) \times \mathbb{S}^{1})   \right)  < \infty.
\end{align*}
Taking the supremum over all $\varphi \in C^1_c (A \times (-d, + \infty))$ with $| \varphi | \leq 1$ 
we conclude.

\vspace{.2cm}

\noindent
\textbf{Step 4:} We show that $\xi_{\mu}, \alpha_\mu \in BV (A \times (-d, + \infty))$.
From Step 3, we have $\mu \in BV (A \times (-d, + \infty))$. Observe now that also the map $(r, y, t) \mapsto 1/r$
belongs to $BV (A \times (-d, + \infty))$. Therefore, from \cite[Example~3.97]{AFP} it follows that both 
 $\xi_{\mu}  = \mu /r$ and $\alpha_{\mu} = \mu/ (2r)$ belong to $BV (A \times (-d, + \infty))$.

\vspace{.2cm}

\noindent
\textbf{Step 5:} We show formulas \eqref{Dy mu} and \eqref{Dt mu}. 
We will only give the proof of \eqref{Dt mu}, because the proof of \eqref{Dy mu} is similar. Using the fact that every function in $C^0_b (A \times (-d, + \infty))$ can be approximated uniformly on compact subsets of $A \times (-d, + \infty)$
by functions in $C^1_c (A \times (-d, + \infty))$, and the fact that $D_t \mu$ is a bounded Radon measure 
on $A \times (-d, + \infty)$, we have that \eqref{Dt mu} holds for every $\varphi \in C^0_b (A \times (-d, + \infty))$.

Let now $\varphi : A \times (-d, + \infty) \to \mathbb{R}$
be a bounded Borel function, and let $\lambda$ be the 
bounded Radon measure on $A \times (-d, + \infty)$ defined by
\[
\lambda (B) := | D_t \mu| (B) + \mathcal{H}^n  \left( \partial^* \Sigma^{u_0} \cap \left( \Phi_{n+1} (B \times \mathbb{S}^{1}) \right) \right), \quad \text{ for every Borel set $B \subset A \times (-d, + \infty)$.}
\]
 By Lusin Theorem, for every $h \in \mathbb{N}$ there exists $\varphi_h \in C^0_b (A \times (-d, + \infty))$ with $\| \varphi_h \|_{L^{\infty} (A \times (-d, + \infty))} \leq \| \varphi \|_{L^{\infty} (A \times (-d, + \infty))}$ such that
\[
\lambda \left( \{ (r, y, t) \in A \times (-d, + \infty) : \varphi (r, y, t) \neq \varphi_h (r, y, t) \}  \right) < \frac{1}{h}.
\]
For each $h \in \mathbb{N}$ we can apply 
\eqref{Dt mu} to $\varphi_h$, so that 
\begin{align*} 
&\int_{A \times (-d, + \infty)} \varphi_h (r, y, t) \, d D_t \mu (r, y, t) \\
&= \int_{\partial^* \Sigma^{u_0} \cap \left( \Phi_{n+1} ((A \times (-d, + \infty)) \times \mathbb{S}^{1}) \right)}  
\varphi_h (|x| , y, t) \, \nu_t^{\Sigma^{u_0}} (x, y, t) \, d \mathcal{H}^n (x, y, t).
\end{align*} 
Using this identity, we have
\begin{align*}
&\left| \int_{A \times (-d, + \infty)} \varphi (r, y, t) \, d D_t \mu (r, y, t) \right. \\
&\hspace{1cm} \left. - \int_{\partial^* \Sigma^{u_0} \cap \left( \Phi_{n+1} ((A \times (-d, + \infty)) \times \mathbb{S}^{1}) \right)}  
\varphi (|x| , y, t) \, \nu_t^{\Sigma^{u_0}} (x, y, t) \, d \mathcal{H}^n (x, y, t) \right| \\
&= \left| \int_{A \times (-d, + \infty)} \big( \varphi (r, y, t) - \varphi_h (r, y, t) \big) \, d D_t \mu (r, y, t) \right. \\
&\hspace{1cm} + \left. \int_{\partial^* \Sigma^{u_0} \cap \left( \Phi_{n+1} ((A \times (-d, + \infty)) \times \mathbb{S}^{1}) \right)}  
\big( \varphi_h (|x| , y, t) - \varphi (|x| , y, t) \big) \, \nu_t^{\Sigma^{u_0}} (x, y, t) \, d \mathcal{H}^n (x, y, t) \right| \\
&\leq \int_{A \times (-d, + \infty)} \left| \varphi (r, y, t) - \varphi_h (r, y, t) \right| \, d \left| D_t \mu \right| (r, y, t) \\
&+ \int_{\partial^* \Sigma^{u_0} \cap \left( \Phi_{n+1} ((A \times (-d, + \infty)) \times \mathbb{S}^{1}) \right)}  
\left|  \varphi_h (|x| , y, t) - \varphi (|x| , y, t) \right| \, d \mathcal{H}^n (x, y, t) 
\leq \frac{2}{h} \| \varphi \|_{L^{\infty} (A \times (-d, + \infty))}.
\end{align*}
Passing to the limit as $h \to \infty$ we obtain \eqref{Dt mu}.

\vspace{.2cm}

\noindent
\textbf{Step 6:} We show \eqref{rDr xi}.
Arguing as in Step 5 it follows that \eqref{rDr xi 2} holds
whenever $\varphi$ is a bounded Borel function on $A \times (-d, + \infty)$. 
Observe now that by \cite[Example~3.97]{AFP} we have
\[
D_r \mu = D_r (r \xi_\mu) = r D_r \xi_\mu + \xi_\mu \, dr \, dy \, dt.
\]
Comparing last identity with \eqref{rDr xi 2} we conclude.
\end{proof}
Before stating the next result, we recall that $\nabla \mu = (\partial_r \mu, \nabla_{y} \mu, \partial_t \mu)$ and $\nabla \xi = (\partial_r \xi, \nabla_{y} \xi, \partial_t \xi)$ denote the absolutely continuous parts of the Radon measures $D \mu$ and $D \xi$.
\begin{proposition} \label{abs cont parts}
Let $n \in \mathbb{N}$ with $n \geq 2$, and let $\Omega \subset \mathbb{R}^n$ be open. 
Let $u \in BV_{0, \tau} (\Omega)$, and let $\mu$, $\alpha_\mu$ and $\xi_\mu$ be given by \eqref{mu distr},
\eqref{def alpha}, and \eqref{def xi}, respectively.
Then, $\nabla \mu = (\partial_r \mu, \nabla_{y} \mu, \partial_t \mu)$ and $\nabla \xi = (\partial_r \xi, \nabla_{y} \xi, \partial_t \xi)$ are concentrated on $\{ 0 < \alpha_\mu < \pi \}$. Moreover,
\begin{align}
\partial_t \mu (r, y, t)
&= - \int_{(\partial^*\Sigma^{u_0})_{(r, y, t)}} \frac{1}{|\Dx u_0 (x, y) |} \, d \mathcal{H}^0 (x) \label{partial t mu} \\[8pt]
\nabla_y \mu (r, y, t)
&= \int_{(\partial^*\Sigma^{u_0})_{(r, y, t)}} \frac{\nabla_y u_0 (x, y)}{|\Dx u_0 (x, y) |} \, d \mathcal{H}^0 (x), \label{partial y mu}
\end{align}
and 
\begin{align}
r \partial_r \xi_{\mu} (r, y, t)
&= \int_{(\partial^*\Sigma^{u_0})_{(r, y, t)}} \frac{\hat{x} \cdot \nabla_x u_0 (x, y)}{|\Dx u_0 (x, y) |} \, d \mathcal{H}^0 (x), \label{r partial r xi}  \\[8pt]
r \nabla_y \xi_{\mu} (r, y, t)  
&= \int_{(\partial^*\Sigma^{u_0})_{(r, y, t)}} \frac{\nabla_y u_0 (x, y)}{|\Dx u_0 (x, y) |} \, d \mathcal{H}^0 (x), \nonumber \\[8pt]
r \partial_t \xi_{\mu} (r, y, t)
&= - \int_{(\partial^*\Sigma^{u_0})_{(r, y, t)}} \frac{1}{|\Dx u_0 (x, y) |} \, d \mathcal{H}^0 (x),  \nonumber
\end{align}
for $\mathcal{H}^n$-a.e. $(r, y, t) \in \{ 0 < \alpha_\mu < \pi \}$. 
\end{proposition}
\begin{proof}
We will only show the statement for 
$\partial_t \mu$, since the other identities can be proven in a similar way. 
Let $A \subset \subset \Pi_{n-1} (\Omega)$ be open and let $d > 0$.
Let now $G_{u_0}$ be given by Proposition~\ref{volpert theorem} for the set $A \times (-d, +\infty)$,  
and let $q \in C^0_c (A \times (-d, +\infty))$. Then, thanks to Remark~\ref{alpha = 0 and alpha = pi}
\begin{align*}
&\int_{\big(  A \times (-d, +\infty) \big)}  q (r, y, t) \, d D_t \mu (r, y, t) \\
&=\int_{\big(  A \times (-d, +\infty) \big) \cap \{ \alpha^{\vee} > 0 \} \cap \{ \alpha^{\wedge} < \pi \}}  q (r, y, t) \, d D_t \mu (r, y, t) \\
&=\int_{\big(  A \times (-d, +\infty) \big) \cap \{ \alpha^{\vee} > 0 \} \cap \{ \alpha^{\wedge} < \pi \} \cap G_{u_0}}  q (r, y, t) \, d D_t \mu (r, y, t) \\
&+\int_{\Big( \big(  A \times (-d, +\infty) \big) \cap \{ \alpha^{\vee} > 0 \} \cap \{ \alpha^{\wedge} < \pi \} \Big) \setminus G_{u_0}}  q (r, y, t) \, d D_t \mu (r, y, t).
\end{align*}
Note now that the last integral does not contain any contribution coming from the absolutely continuous part of $D_t \mu$, since 
the region of integration is $\mathcal{H}^n$-negligible:
\[
\mathcal{H}^n \left( \Big( \big(  A \times (-d, +\infty) \big) \cap \{ \alpha^{\vee} > 0 \} \cap \{ \alpha^{\wedge} < \pi \} \Big) \setminus G_{u_0} \right) = 0.
\]
For the remaining integral, setting $R:= (  A \times (-d, +\infty) ) \cap \{ \alpha^{\vee} > 0 \} \cap \{ \alpha^{\wedge} < \pi \} \cap G_{u_0}$, 
thanks to Proposition~\ref{giaquinta}, \eqref{Dt mu}, and Remark~\ref{remark us constant functions new}, we have 
\begin{align*}
&\int_{R}  q (r, y, t) \, d D_t \mu (r, y, t) 
= \int_{\partial^* \Sigma^{u_0} \cap \left( \Phi_{n+1} (R \times \mathbb{S}^{1})  \right)} 
q (|x|, y, t) \nu_t^{\Sigma^{u_0}} (x, y, t) \, d \mathcal{H}^n (x, y, t)  \\
&= -\int_{\partial^* \Sigma^{u_0} \cap \{ \Dx u_0   \neq 0 \} \cap \left( \Phi_{n+1} (R \times \mathbb{S}^{1})  \right)} 
 \, \frac{q (|x|, y, t) }{\sqrt{1 + |\nabla u_0 (x, y)|^2}} \, d \mathcal{H}^n (x, y, t) \\
 &= - \int_{(A \times (-d, +\infty)) \cap \{ \alpha^{\vee} > 0 \} \cap \{ \alpha^{\wedge} < \pi \}  \cap G_{u_0}} q (r, y, t) \left( \int_{(\partial^*\Sigma^{u_0})_{(r, y, t)} }  
 \, \frac{1}{| \Dx u_0 (x, y)|} \, d \mathcal{H}^{0} (x) \right)\, dr \, dy \, dt \\
 &= - \int_{(A \times (-d, +\infty)) \cap \{ 0< \alpha_\mu  < \pi \}  \cap G_{u_0}} q (r, y, t) \left( \int_{(\partial^*\Sigma^{u_0})_{(r, y, t)} }  
 \, \frac{1}{| \Dx u_0 (x, y)|} \, d \mathcal{H}^{0} (x) \right)\, dr \, dy \, dt,
\end{align*}
where we also used Proposition~\ref{coarea formula} and identity (iiia) of Proposition~\ref{volpert theorem}.
Since $q$ is arbitrary, this shows that the density $\partial_t \mu$ 
of the absolutely continuous part of the measure $D_t \mu$ satisfies
\begin{equation} \label{abs cont part}
\partial_t \mu (r, y, t) = -  \int_{(\partial^*\Sigma^{u_0})_{(r, y, t)} }  
 \, \frac{1}{| \Dx u_0 (x, y)|} \, d \mathcal{H}^{0} (x),
\end{equation}
for $\mathcal{H}^n$-a.e. $(r, y, t) \in A \times (-d, +\infty) \cap \{ 0 < \alpha_\mu < \pi \} \cap G_{u_0}$. 
Recalling now that 
\[
\mathcal{H}^n \Big( ( A \times (-d, +\infty) ) \cap \{ 0 < \alpha_\mu < \pi \}
\setminus G_{u_0}
  \Big) = 0, 
\]
we obtain that \eqref{abs cont part} is satisfied for $\mathcal{H}^n$-a.e. $(r, y, t) \in ( A \times (-d, +\infty) ) \cap \{ 0 < \alpha_\mu < \pi \}$.
Since $A$ and $d$ are arbitrary, the conclusion follows.
\end{proof}

The next result shows a perimeter inequality under circular symmetrization of subgraphs.
\begin{proposition}\label{prop_circular perimeter inequality for subgraphs}
Let $n \in \mathbb{N}$ with $n \geq 2$, and let $\Omega \subset \mathbb{R}^n$ be open. 
Let $u \in BV_{0, \tau} (\Omega)$, let $\mu$ be given by \eqref{mu distr},
and let $v_{\mu}$ be defined by \eqref{def vmu}.
Then, $v_{\mu}  \! \mid_{\Omega^s} \in BV_{0, \tau} (\Omega^s)$. Moreover, $\Sigma^{v_\mu}$
is a set of locally finite perimeter in $\Phi_{n+1}  ((\Pi_{n-1} (\Omega) \times \R) \times \mathbb{S}^{1})$ and
\begin{equation} \label{per ineq cyl subgraph}
P (\Sigma^{v_\mu}; \Phi_{n+1} (B \times \mathbb{S}^{1})) 
\leq P (\Sigma^{u_0}; \Phi_{n+1}  (B \times \mathbb{S}^{1})),
\end{equation}
for every Borel set $B \subset \Pi_{n-1} (\Omega) \times \R$. 
\end{proposition}
\begin{remark}
Thanks to Remark~\ref{u bv finite perimeter}, we have that 
$\Sigma^{v_\mu}$ and $\Sigma^{u_0}$ are sets of finite perimeter in 
$\Phi_{n+1} ((A \times \R) \times \mathbb{S}^1)$, for every open set $A \subset \subset \Pi_{n-1} (\Omega)$.
Therefore,  
\eqref{per ineq cyl subgraph} should be interpreted as an inequality between extended real numbers.
\end{remark}
\begin{proof}
Let $A \subset \subset \Pi_{n-1} (\Omega)$ be open.
We will show that $v_{\mu} \in BV (\Phi_n (A \times \mathbb{S}^1) )$.
By assumption, we already know that $u_0 \in BV (\Phi_n (A \times \mathbb{S}^1) )$.
Thanks to Proposition~\ref{giaquinta}, $\Sigma^{u_0}$
is a set of finite perimeter in $\Phi_n (A \times \mathbb{S}^1) \times \mathbb{R}$.
Let now $\left(\Sigma^{u_0} \right)^s$ and $F_\mu$ be defined by \eqref{def Es} and \eqref{def Fmu}, respectively.
We have 
\[
\left(\Sigma^{u_0} \right)^s = F_\mu =_{\mathcal{H}^{n+1}} \Sigma^{v_\mu}, 
\]
where the last equality follows from \eqref{equivalence}. 
We would like to apply Theorem~\ref{per ineq circ} to the sets $\Sigma^{v_\mu}$ and $\Sigma^{u_0}$, but this is not possible, 
since they do not have finite ($n$-dimensional) Lebesgue measure.
Therefore, for each $d > 0$ if we set 
\[
\Sigma^{u_0}_{d} = \Sigma^{u_0} \cap \{ (x, y, t) \in \R^{n+1} : t > - d \} 
\quad \text{ and }\quad
\Sigma^{v_\mu}_d = \Sigma^{v_\mu} \cap \{ (x, y, t) \in \R^{n+1} : t > - d \},
\]
so that
\[
\left(\Sigma^{u_0}_{d} \right)^s =_{\mathcal{H}^{n+1}} \Sigma^{v_\mu}_d. 
\]
From Proposition~\ref{prop_circular rear and circular sym are connected} it follows that \( \mu \in L^1 (A \times (-d, + \infty)) \), 
so that $\mathcal{L}^{n+1} (\Sigma^{u_0}_{d}) = \mathcal{L}^{n+1} (\Sigma^{v_\mu}_{d}) < + \infty$.
Thanks to \cite[Proposition~2.16]{maggiBOOK}, there exists a sequence $d_k \to + \infty$ such that for each $k \in \mathbb{N}$
\begin{equation} \label{d_k does not contribute}
\mathcal{H}^n (\partial^* \Sigma^{u_0}_{d_k} \cap \{ (x, y, t) \in \R^{n+1} : t = - d_k \} ) 
= \mathcal{H}^n (\partial^* \Sigma^{v_\mu}_{d_k} \cap \{ (x, y, t) \in \R^{n+1} : t = - d_k \} ) = 0.
\end{equation}
If we set $\tilde A:= A \times \R$, we have 
$\Phi_n (A \times \mathbb{S}^1) \times \mathbb{R} = \Phi_{n+1}  (\tilde A \times \mathbb{S}^{1})$. 
Then, thanks to Theorem~\ref{per ineq circ}, $\Sigma^{v_\mu}_{d_k}$ is a set of 
finite perimeter in $\Phi_{n+1}  (\tilde A \times \mathbb{S}^{1})$ and
\begin{equation} \label{afdsfz}
P (\Sigma^{v_\mu}_{d_k}; \Phi_{n+1} (\tilde A \times \mathbb{S}^{1})) 
\leq P (\Sigma^{u_0}_{d_k}; \Phi_{n+1}  (\tilde A \times \mathbb{S}^{1})).
\end{equation}
Thanks to \cite[formula~(16.10)]{maggiBOOK} and recalling \eqref{d_k does not contribute}, we have 
\begin{align*}
&P (\Sigma^{u_0}_{d_k}; \Phi_{n+1} (\tilde A \times \mathbb{S}^{1})) 
= P (\Sigma^{u_0}; \Phi_{n+1} (\tilde A \times \mathbb{S}^{1}) \cap \{ t > - d_k \}) 
+ P (\{ t > - d_k \}; (\Sigma^{u_0})^{(1)}) \\
&= P (\Sigma^{u_0}; \Phi_{n+1} (\tilde A \times \mathbb{S}^{1}) \cap \{ t > - d_k \}) 
+ \mathcal{L}^n (\{ u_0 > - d_k \} \cap \Phi_n (A \times \mathbb{S}^1)) \\
&= P (\Sigma^{u_0}; \Phi_{n+1} (\tilde A \times \mathbb{S}^{1}) \cap \{ t > - d_k \}) 
+ \mathcal{L}^n (\{ v_\mu > - d_k \} \cap \Phi_n (A \times \mathbb{S}^1)).
\end{align*}
Similarly, we have 
\begin{align*}
P (\Sigma^{v_\mu}_{d_k}; \Phi_{n+1} (\tilde A \times \mathbb{S}^{1})) 
= P (\Sigma^{v_\mu}; \Phi_{n+1} (\tilde A \times \mathbb{S}^{1}) \cap \{ t > - d_k \}) 
+ \mathcal{L}^n (\{ v_\mu > - d_k \} \cap \Phi_n (A \times \mathbb{S}^1)).
\end{align*}
Therefore, thanks to \eqref{afdsfz},
\begin{equation*} 
P (\Sigma^{v_\mu}; \Phi_{n+1} (\tilde A \times \mathbb{S}^{1}) \cap \{ t > - d_k \}) 
\leq P (\Sigma^{u_0}; \Phi_{n+1} (\tilde A \times \mathbb{S}^{1}) \cap \{ t > - d_k \}), \quad \forall \, k \in \mathbb{N}. 
\end{equation*}
Passing to the limit as $k \to \infty$, we have 
\begin{equation} \label{afdsfer}
P (\Sigma^{v_\mu}; \Phi_{n+1} (\tilde A \times \mathbb{S}^{1})) 
\leq P (\Sigma^{u_0}; \Phi_{n+1}  (\tilde A \times \mathbb{S}^{1})).
\end{equation}
Therefore, $\Sigma^{v_\mu}$ is a set of finite perimeter in $\Phi_{n+1} (\tilde A \times \mathbb{S}^{1})$.
Thanks to Proposition~\ref{prop Lp norm}, we have $v_{\mu} \in L^1 (\Phi_n (A \times \mathbb{S}^1) )$. 
Then, from Proposition~\ref{giaquinta} it follows that $v_\mu \in BV (\Phi_n (A \times \mathbb{S}^1) )$.
From \eqref{afdsfer}, since $A \subset \subset \Pi_{n-1} (\Omega)$ was arbitrary, inequality \eqref{per ineq cyl subgraph} follows. 
Finally, thanks to \eqref{v_mu is zero a.e. outside} and using again the arbitrariness of $A$, we conclude that 
$v_{\mu} \! \mid_{\Omega^s} \in BV_{0, \tau} (\Omega^s)$.
\end{proof}
 \begin{remark} \label{this is now the magic property}
Let $n \in \mathbb{N}$ with $n \geq 2$, and let $\Omega \subset \mathbb{R}^n$ be open. 
Let $u \in BV_{0, \tau} (\Omega)$, and let $\mu$ be given by \eqref{mu distr}.
Then, thanks to \cite[Proposition 1.2]{PeruginiCircolare} applied to $\Sigma^{v_\mu}$, we have that for every $(r, y, t) \in \Pi_{n-1} (\Omega) \times \R$
such that $( \partial^* \Sigma^{v_\mu} )_{(r, y, t)} \neq \emptyset$, the functions
\begin{align*}
x \mapsto \hat{x} \cdot \nu^{ \Sigma^{v_\mu}}_x (x, y, t), 
\quad
x \mapsto | \nu^{ \Sigma^{v_\mu}}_{\!  x {\scriptscriptstyle\parallel}} (x, y, t)|, 
\quad 
x \mapsto \nu^{ \Sigma^{v_\mu}}_y (x, y, t),
\quad 
x \mapsto \nu^{ \Sigma^{v_\mu}}_t (x, y, t),
\end{align*}
are constant in $( \partial^* \Sigma^{v_\mu} )_{(r, y, t)}$. Thanks to \eqref{normal1}, this means that 
\begin{align*}
x \mapsto \hat{x} \cdot \nabla_x v_\mu (x, y, t), 
\quad
x \mapsto | \Dx v_\mu (x, y, t)|, 
\quad 
x \mapsto \nabla_y v_\mu (x, y, t),
\quad 
x \mapsto |\nabla v_\mu | (x, y, t),
\end{align*}
are constant in $( \partial^* \Sigma^{v_\mu} )_{(r, y, t)}$. 
\end{remark}

\begin{remark} \label{nice derivatives}
Suppose that $u \in BV_{0, \tau} (\Omega)$ and let $\mu$ be given by \eqref{mu distr}.
Thanks to Proposition~\ref{prop_circular perimeter inequality for subgraphs}, we also have $v_{\mu}  \mid_{\Omega^s} \in BV_{0, \tau} (\Omega^s)$.
Then, applying Proposition~\ref{abs cont parts} to $v_{\mu}$ and 
taking into account Remark~\ref{this is now the magic property},
we have 
\begin{align} 
\partial_t \mu (r, y, t) 
&= - \frac{2}{|\Dx v_\mu (x, y)|}, \label{Dt mu for us} \\[8pt]
\nabla_y \mu (r,y, t) 
&= 2  \frac{\nabla_y v_\mu (x,y)}{|\Dx v_\mu (x, y)|}, \label{r^(n-1)Dxi for us} \\[8pt]
r \partial_r \xi (r,y, t) 
&= 2  \frac{\widehat{x} \cdot \nabla_x v_\mu (x,y)}{|\Dx v_\mu (x, y)|}, \label{r^(n-1)Dxi for us} 
\end{align}
for $\mathcal{H}^n$-a.e. $(r, y, t) \in \Pi_{n-1} (\Omega) \times \mathbb{R}$. 
\end{remark}

\noindent
We are now ready to show that the circular rearrangement of a function in $W^{1,1}_{0, \tau} (\Omega)$
also belongs to $W^{1,1}_{0, \tau} (\Omega)$.
\begin{proposition} \label{ us sobolev}
Let $n \in \mathbb{N}$ with $n \geq 2$, and let $\Omega \subset \mathbb{R}^n$ be open.
Let $u \in W^{1,1}_{0, \tau} (\Omega)$, let $\mu$ be given by \eqref{mu distr},
and let $v_{\mu}$ be defined by \eqref{def vmu}. Then, $ v_\mu \! \mid_{\Omega^s} \in W^{1,1}_{0, \tau} (\Omega^s)$.
\end{proposition}

\begin{proof}
Let $A \subset \subset \Pi_{n-1} (\Omega)$ be open.
Thanks to Proposition~\ref{prop_circular perimeter inequality for subgraphs}, $v_{\mu} \in BV (\Phi_n (A \times \mathbb{S}^1) )$, 
and $\Sigma^{u_0}$ and $\Sigma^{v_{\mu}}$ are sets of finite perimeter 
in $\Phi_n (A \times \mathbb{S}^1) \times \mathbb{R}$. Let us first show that 
\begin{equation} \label{dkue}
\mathcal{H}^n (B)= 0, \quad \text{ where } \quad B := \{ (x, y, t) \in \partial^* \Sigma^{v_{\mu}} : \nu_t^{\Sigma^{v_{\mu}}} (x, y, t) = 0 \} 
\cap  ( \Phi_n (A \times \mathbb{S}^1) \times \mathbb{R} ).
\end{equation}
First of all note that, by definition of $\Phi_n$ and $\Phi_{n+1}$, we have
\[
B := \{ (x, y, t) \in \partial^* \Sigma^{v_{\mu}} : \nu_t^{\Sigma^{v_{\mu}}} (x, y, t) = 0 \} 
\cap ( \Phi_{n+1} ( (A \times \R) \times \mathbb{S}^{1}) ).
\]
Let now $S: \R^2_0 \times \R^{n-2} \times \R \to (0, + \infty) \times \R^{n-2} \times \R$ be defined as $S (x, y, t) := (|x|, y, t)$. 
Thanks to Remark~\ref{this is now the magic property}, we have
\begin{align*}
B = \partial^* \Sigma^{v_\mu} \cap ( \Phi_{n+1} ( S(B) \times \mathbb{S}^{1}) ).
\end{align*}
Since $S$ is continuous, $B$ is Borel, and $D_t \mu$ is a finite Radon measure on $(0, + \infty) \times \R^{n-2} \times \R$, the set $S (B)$ is 
$D_t \mu$-measurable. 
Moreover, since $D_t \mu$ is a finite Radon measure, formula \eqref{Dt mu} holds also for bounded $D_t \mu$-measurable functions. In particular, we have 
 \begin{equation} \label{dtmu on measurable}
 \begin{split}
 &\int_{\partial^* \Sigma^{v_\mu} \cap ( \Phi_{n+1} ( S (B) \times \mathbb{S}^{1}) )}  
 \, \nu_t^{\Sigma^{v_\mu}} (x, y, t)\, d \mathcal{H}^n (x, y, t) \\
& = D_t \mu  (S (B))  
 = \int_{\partial^* \Sigma^{u_0} \cap ( \Phi_{n+1} ( S (B) \times \mathbb{S}^{1}) )}  
 \, \nu_t^{\Sigma^{u_0}} (x, y, t)\, d \mathcal{H}^n (x, y, t).
\end{split}
\end{equation}
Now, since $\nu_t^{\Sigma^{v_{\mu}}} (x, y, t) = 0$ for every $(x, y, t) \in B$, thanks to \eqref{dtmu on measurable}
\begin{align*}
0&= \int_{B}  
 \, \nu_t^{\Sigma^{v_\mu}} (x, y, t)\, d \mathcal{H}^n (x, y, t) 
 = \int_{\partial^* \Sigma^{v_\mu} \cap ( \Phi_{n+1} ( S (B) \times \mathbb{S}^{1}) )}  
 \, \nu_t^{\Sigma^{v_\mu}} (x, y, t)\, d \mathcal{H}^n (x, y, t) \\
&= \int_{\partial^* \Sigma^{u_0} \cap ( \Phi_{n+1} ( S (B) \times \mathbb{S}^{1}) )}  
 \, \nu_t^{\Sigma^{u_0}} (x, y, t)\, d \mathcal{H}^n (x, y, t) \\
&= - \int_{\partial^* \Sigma^{u_0} \cap ( \Phi_{n+1} ( S (B) \times \mathbb{S}^{1}) )}  
 \, \frac{1}{\sqrt{1 + |\nabla u_0 (x, y)|^2}} \, d \mathcal{H}^n (x, y, t),
\end{align*}
where the last equality follows from Proposition~\ref{giaquinta} and from the fact that $u_0 \in W^{1,1} (\Phi_n (A \times \mathbb{S}^1) )$.
From the above chain of equalities we infer that 
\[
\mathcal{H}^n \left(  \partial^* \Sigma^{u_0} \cap ( \Phi_{n+1} (  S (B) \times \mathbb{S}^{1}) )  \right) = 0.
\]
Then, thanks to \eqref{per ineq cyl subgraph},
\begin{align*}
0 &= \mathcal{H}^n \left(  \partial^* \Sigma^{u_0} 
\cap ( \Phi_{n+1} ( S (B) \times \mathbb{S}^{1}) )  \right) 
= P ( \Sigma^{u_0}; \Phi_{n+1}  (S (B) \times \mathbb{S}^{1})) \\
& \geq P ( \Sigma^{v_\mu}; \Phi_{n+1} (S (B) \times \mathbb{S}^{1})) 
= \mathcal{H}^n \left(  \partial^* \Sigma^{v_\mu} 
\cap ( \Phi_{n+1} ( S (B) \times \mathbb{S}^{1}) )  \right) = \mathcal{H}^n (B),
\end{align*}
which shows \eqref{dkue}.

Once \eqref{dkue} is satisfied, thanks to Proposition~\ref{prop_no flat parts means being Sobolev} 
it follows that $v_\mu \in W^{1, 1} (\Phi_n (A \times \mathbb{S}^1) )$.
Using the arbitrariness of $A$ and thanks to \eqref{v_mu is zero a.e. outside}, we have
$ v_\mu \! \mid_{\Omega^s} \in W^{1,1}_{0, \tau} (\Omega^s)$.
\end{proof}

\noindent
The next proposition will be useful to understand
rigidity of the P\'olya--Szeg\"o inequality. 
\begin{proposition} \label{flat parts for u and vmu}
Let $n \in \mathbb{N}$ with $n \geq 2$, and let $\Omega \subset \mathbb{R}^n$ be open. 
Let $u \in W^{1,1}_{0, \tau} (\Omega)$, and let $\mu$ be given by \eqref{mu distr}.
Let $v_{\mu}$ be defined by \eqref{def vmu} and let
$B \subset \subset \Pi_{n-1} (\Omega) \times \R$ be a Borel set.
Then, 
\begin{equation} \label{no flat parts for u}
\mathcal{H}^{n}\left(\Big\{(x, y, t) \in \partial^* \Sigma^{u_0} \cap \Phi_{n+1}
(B \times \mathbb{S}^{1}): \nu^{\Sigma^{u_0}}_{\!  x {\scriptscriptstyle\parallel}} (x, y, t)=0  \Big\}  \right)=0
\end{equation}
if and only if
\begin{equation} \label{no flat parts for u^s}
\mathcal{H}^{n}\left(\Big\{(x, y, t) \in \partial^* \Sigma^{v_\mu} \cap \Phi_{n+1}
(B \times \mathbb{S}^{1}): \nu^{\Sigma^{v_\mu}}_{\!  x {\scriptscriptstyle\parallel}} (x, y, t)=0  \Big\}  \right)=0.
\end{equation}
\end{proposition}


\begin{proof}
Applying formula \eqref{Dt mu}
with $\varphi = \chi_{B'}$ first to $\Sigma^{u_0}$ and then to 
$\Sigma^{v_\mu}$, taking into account Proposition~\ref{giaquinta} and Proposition~\ref{ us sobolev}, we have 
\begin{equation} \label{amazing}
\begin{split}
 &- \int_{\partial^* \Sigma^{u_0} \cap \left( \Phi_{n+1} (B' \times \mathbb{S}^{1})  \right)}  
 \, \frac{1}{\sqrt{1 + |\nabla u_0 (x, y)|^2}} \, d \mathcal{H}^n (x, y, t) \\[3pt]
& = D_t \mu  (B') \\[3pt]
&= - \int_{\partial^* \Sigma^{v_\mu} \cap \left( \Phi_{n+1} (B' \times \mathbb{S}^{1})  \right)}  
 \, \frac{1}{\sqrt{1 + |\nabla v_\mu (x, y)|^2}} \, d \mathcal{H}^n (x, y, t),
 \end{split}
 \end{equation}
 for every Borel set $B' \subset B$.
 Then, thanks to Proposition~\ref{prop flat and perimeter}, we have
 \begin{align*}
\eqref{no flat parts for u} 
&\quad \Longleftrightarrow \quad
 P(\Sigma^{u_0}; \Phi_{n+1}( B' \times \mathbb{S}^{1}))=0 \quad \forall \, \text{ Borel set $B' \subset B$ with $\mathcal{H}^n(B')=0$} \\
&\quad \Longleftrightarrow \quad
 \mathcal{H}^n ( \partial^* \Sigma^{u_0} \cap \Phi_{n+1}( B' \times \mathbb{S}^{1}))=0 \quad \forall \, \text{ Borel set $B' \subset B$ with $\mathcal{H}^n(B')=0$} \\
& \stackrel{\eqref{amazing}}{\quad \Longleftrightarrow \quad}
 D_t \mu  (B') = 0 \quad \forall \, \text{ Borel set $B' \subset B$ with $\mathcal{H}^n(B')=0$} \\
 & \stackrel{\eqref{amazing}}{\quad \Longleftrightarrow \quad}
  \mathcal{H}^n ( \partial^* \Sigma^{v_\mu} \cap \Phi_{n+1}( B' \times \mathbb{S}^{1}))=0 \quad \forall \, \text{ Borel set $B' \subset B$ with $\mathcal{H}^n(B')=0$} \\
  &\quad \Longleftrightarrow \quad
 P(\Sigma^{v_\mu}; \Phi_{n+1}( B' \times \mathbb{S}^{1}))=0 \quad \forall \, \text{ Borel set $B' \subset B$ with $\mathcal{H}^n(B')=0$} \\
 &\quad \Longleftrightarrow \quad \eqref{no flat parts for u^s}.
 \end{align*}
\end{proof}

\section{Proof of the P\'olya--Szeg\"o inequality} \label{sect polya-szego}

In this section we prove a general version of the P\'olya--Szeg\"o inequality (see Theorem~\ref{the thing}), 
and then we show how this implies Theorem~\ref{thm_Polya-Szego inequality local} and Theorem~\ref{thm_Polya-Szego inequality}.
We start with some preliminary results that concerning functions $f$ belonging to $\mathscr{F}$ and $\mathscr{F}'$.
\begin{lemma} \label{lemma equivalent increasing}
Let $\mathscr{F}$ and $\mathscr{F}'$ be given by Definition~\ref{def_admissable integrands}, and let $f \in \mathscr{F}$. Then, 
\[
\tau \longmapsto f (\eta, \tau, \zeta) \text{ is increasing in }[0, \infty) 
\qquad 
\text{ for every } 
(\eta, \zeta) \in \R \times \R^{n-2}.
\]
If, in addition, $f \in \mathscr{F}'$, then
\[
\tau \longmapsto f (\eta, \tau, \zeta) \text{ is \textbf{strictly} increasing in }[0, \infty)
\qquad 
\text{ for every } 
(\eta, \zeta) \in \R \times \R^{n-2}.
\]
\end{lemma}
\begin{proof}
Suppose that $f \in \mathscr{F}$, let $(\eta, \zeta) \in \R \times \R^{n-2}$ be fixed, 
and let $g:\R \to [0, +\infty)$ be defined as $g (\tau):= f (\eta, \tau, \zeta)$. We want to show 
that $g$ is increasing. First of all, note that since $f$ is convex we have that $g$ is convex.
Let now $0 \leq \tau_1 < \tau_2$. 
Then, we can write $\tau_1 = \lambda (- \tau_2) + (1-\lambda) \tau_2$, 
where $\lambda := \frac{\tau_2 - \tau_1}{2 \tau_2} \in (0,1/2)$.
Therefore, by convexity of $g$
\[
g (\tau_1) = g (\lambda (-\tau_2) + (1-\lambda) \tau_2) \leq \lambda g (-\tau_2) + (1- \lambda) g (\tau_2) = g (\tau_2),
\]
which shows the first part of the statement.  
If, in addition, $f \in \mathscr{F}'$, then the function $g$ is strictly convex and so
the inequality above is strict.
\end{proof}
\noindent
If $f \in \mathscr{F}$, we recall that the Legendre transform $f^*$ of $f$ is defined as:
\[
f^* (w) := \sup_{(\eta, \zeta, \tau) \in \mathbb{R}^n} (   (\eta, \zeta, \tau) \cdot w - f (\eta, \zeta, \tau)), \quad \text{ for every } w \in  \R^{n}.
\]
Since $f$ is convex and finite everywhere, $f^*$ is convex with $- \infty < f^* (w) \leq + \infty$ for every $w \in  \R^{n}$. 
Moreover, for every $(\eta, \zeta, \tau) \in  \R^{n}$ we have
 \begin{equation} \label{leg transf twice}
 \begin{split}
f (\eta, \zeta, \tau) 
&= (f^*)^* (\eta, \zeta, \tau) 
= \sup_{w \in  \R^{n}} ( w \cdot (\eta, \zeta, \tau) - f^* (w)) \\
&= \sup_{h \in \mathbb{N}} ( w_h \cdot (\eta, \zeta, \tau) - f^* (w_h)),
\end{split}
\end{equation}
where $\{ w_h \}_{h \in \mathbb{N}}$ is a countable dense subset of the set
$\{ w \in \R^n: f^* (w) < + \infty \}$.

The next result is a variant of \cite[Lemma~2.35]{AFP}, and will be used 
in the proof of Theorem~\ref{the thing}. 
\begin{lemma} \label{thanks for this lemma}
Let $n \in \mathbb{N}$ with $n \geq 2$, and let $\Omega \subset \mathbb{R}^n$ be open. 
Let $u \in BV_{0, \tau} (\Omega)$, let $A \subset \subset \Pi_{n-1} (\Omega)$ be open, 
and let $B \subset A \times \mathbb{R}$ be a Borel set. 
For every $h \in \mathbb{N}$, let 
$\varphi_h: \Phi_{n+1} (B \times \mathbb{S}^{1}) \to \mathbb{R}$ be
a Borel function such that 
\begin{equation} \label{constant along slices}
x \longmapsto \varphi_h (x, y, t) \text{ is constant in } (\partial^* \Sigma^{u_0})_{(r, y, t)} \quad 
\text{ for every } (r, y, t) \in B.
\end{equation}
Suppose, in addition, that there exists $\tilde{h} \in \mathbb{N}$ such that 
\[
\int_{\partial^* \Sigma^{u_0} \cap \Phi_{n+1} (B \times \mathbb{S}^{1})} | \varphi_{\tilde{h}} (x, y, t) | \, d \mathcal{H}^n (x, y, t) < + \infty.
\]
Then, 
\begin{align*}
&\int_{\partial^* \Sigma^{u_0} \cap \Phi_{n+1} (B \times \mathbb{S}^{1})}  
\, \, \sup_{h \in \mathbb{N}} 
  \varphi_h (x, y, t) \, d \mathcal{H}^n (x, y, t) \\
 &= \sup_{H} \left\{\sum_{h\in H} 
 \int_{\partial^* \Sigma^{u_0} \cap \Phi_{n+1} (B_h \times \mathbb{S}^{1})} 
 \varphi_h (x, y, t) \, d \mathcal{H}^n (x, y, t) \right\},
\end{align*}
where the supremum in the right hand side ranges over all finite sets $H\subset \mathbb{N}$ 
and over all pairwise disjoint Borel partitions $\{ B_h \}_{h\in H}$ of $B$.
\end{lemma}

\begin{proof}
In order to prove the lemma, we first need to introduce some tools.
We define the map $S: \R^2_0 \times \R^{n-2} \times \R \to (0, + \infty) \times \R^{n-2} \times \R$ as
\begin{equation} \label{def map S}
S (x, y, t) := (|x|, y, t), \qquad \text{ for every } (x, y, t)  \in \R^2_0 \times \R^{n-2} \times \R.
\end{equation}
Let $\lambda$ be the finite Radon measure on $\R^2_0 \times \R^{n-2} \times \R$ given by 
\[
\lambda (R) := \mathcal{H}^n \left(R \cap  \partial^* \Sigma^{u_0} \cap \Phi_{n+1} ((A \times \R) \times \mathbb{S}^{1}) \right), 
\quad \forall \text{ Borel set } R \subset \R^2_0 \times \R^{n-2} \times \R,
\]
and let $\sigma: = S_{\#} \lambda$, where $S_{\#} \lambda$ denotes the push--forward measure of $\lambda$ through $S$.
We observe that both $\lambda$ and $\sigma$ are finite Borel measures. 
Finally, we set for every $h \in \{ 1, \ldots, k\}$,
\begin{equation} \label{def Ah}
A_h := \left\{ 
(x, y, t) \in \partial^* \Sigma^{u_0} \cap \Phi_{n+1} (B \times \mathbb{S}^{1}) : 
\varphi_h (x, y, t) = \max_{1 \leq i \leq k}  \varphi_i (x, y, t) \right\}.
\end{equation}
We now divide the proof into steps.

\vspace{.2cm}

\noindent
\textbf{Step 1:} We show that the statement follows if we prove that for every $k \in \mathbb{N}$ with $k \geq \tilde{h}$, we have
\begin{equation} \label{before k goes to infinity}
\begin{split}
&\int_{\partial^* \Sigma^{u_0} \cap \Phi_{n+1} (B \times \mathbb{S}^{1}) }  
\, \, \max_{1 \leq h \leq k} 
  \varphi_h (x, y, t) \, d \mathcal{H}^n (x, y, t)  \\
 &= \sup \left\{\sum_{h = 1}^k 
 \int_{\partial^* \Sigma^{u_0} \cap \Phi_{n+1} (B_h \times \mathbb{S}^{1})} 
 \varphi_h (x, y, t) \, d \mathcal{H}^n (x, y, t) \right\}, 
\end{split}
\end{equation}
where the last supremum ranges over all the $k$-ples $B_1, \ldots, B_k$
of pairwise disjoint Borel partitions of $B$. 

To this aim, for every $k\ge \tilde h$, let us define the function $M_k: \Phi_{n+1} (B \times \mathbb{S}^{1}) \to \mathbb{R}$ as 
\[
M_k (x, y, t) :=\max_{1\le h\le k}\varphi_h (x, y, t).
\]
Then, we have $M_k \ge \varphi_{\tilde h}$ and $(M_k)_-\le (\varphi_{\tilde h})_-$. Moreover, by the assumption on $\varphi_{\tilde h}$ it follows that 
\[
\int_{\partial^* \Sigma^{u_0} \cap \Phi_{n+1} (B \times \mathbb{S}^{1})} \!\!\! (M_k)_- (x, y, t)  \, d \mathcal{H}^n (x, y, t)
\leq
\int_{\partial^* \Sigma^{u_0} \cap \Phi_{n+1} (B \times \mathbb{S}^{1})} \!\!\! (\varphi_{\tilde{h}})_- (x, y, t)  \, d \mathcal{H}^n (x, y, t) < \infty.
\] 
Thus, the integral
\[
\int_{\partial^* \Sigma^{u_0} \cap \Phi_{n+1} (B \times \mathbb{S}^{1})} M_k (x, y, t)  \, d \mathcal{H}^n (x, y, t),
\]
is well defined, and its value belongs to $(-\infty, + \infty]$. 

We also have that $M_k-\varphi_{\tilde h}\ge 0$, 
and $M_k-\varphi_{\tilde h} \uparrow \sup_{h\in\mathbb N}\varphi_h-\varphi_{\tilde h}$ as $k\to\infty$.
Therefore, by the Monotone Convergence Theorem applied to $M_k-\varphi_{\tilde h}$, we have
\begin{align*}
&\lim_{k \to + \infty} \int_{\partial^* \Sigma^{u_0} \cap \Phi_{n+1} (B \times \mathbb{S}^{1})} 
\max_{1\le h\le k}\varphi_h (x, y, t)
\, d \mathcal{H}^n (x, y, t) \\
  &= \int_{\partial^* \Sigma^{u_0} \cap \Phi_{n+1} (B \times \mathbb{S}^{1})}  
\, \,  \sup_{h \in \mathbb{N}} 
  \varphi_h (x, y, t)  \, d \mathcal{H}^n (x, y, t),
\end{align*}
so the claim follows.

\vspace{.2cm}

\noindent
\textbf{Step 2:} For every $k \in \mathbb{N}$ with $k \geq \tilde{h}$, we construct a pairwise disjoint Borel partition $\{C_1, \ldots, C_k \}$ of $B$ such that
for every $h = 1, \ldots, k$
\[
\varphi_h (x, y, t) = \max_{1 \leq i \leq k}  \varphi_i (x, y, t), \quad \text{ for $\mathcal{H}^n$-a.e. }  (x, y, t) \in \partial^* \Sigma^{u_0} \cap \Phi_{n+1} (C_h \times \mathbb{S}^{1}).
\]
Let $k \in \mathbb{N}$ with $k \geq \tilde{h}$ be fixed. We divide this step into sub-steps.

\vspace{.2cm}

\noindent
\textbf{Step 2a:} We show that for every $h = 1, \ldots, k$ the set $S(A_h)$ is $\sigma$-measurable,
where $S$ and $A_h$ are defined in \eqref{def map S} and \eqref{def Ah}, respectively.

Let $h \in \{ 1, \ldots, k \}$ be fixed. Since $A_h$ is Borel and $S$ is continuous, $S(A_h)$ is an analytic subset of $(0, \infty) \times \R^{n-2} \times \R$ (see \cite[Section~8.2]{cohn}). Note that, in general, there is no guarantee that $S(A_h)$ is Borel (see \cite[Corollary~8.2.17]{cohn}). However, since $\sigma$ is a finite Borel measure on $(0, \infty) \times \R^{n-2} \times \R$, we have that $S(A_h)$ is $\sigma$--measurable, see \cite[Theorem~8.4.1]{cohn}.

\vspace{.2cm}

\noindent
\textbf{Step 2b:} We show that for every $h = 1, \ldots, k$
\begin{equation} \label{annulus h}
A_h = \partial^* \Sigma^{u_0} \cap \Phi_{n+1} (S (A_h) \times \mathbb{S}^{1}). 
\end{equation}
The inclusion $A_h \subset \partial^* \Sigma^{u_0} \cap \Phi_{n+1} (S (A_h) \times \mathbb{S}^{1})$ follows by the definition of $S$, so we only need to prove the opposite inclusion.
Let $(x', y, t) \in \partial^* \Sigma^{u_0} \cap \Phi_{n+1} (S(A_h) \times \mathbb{S}^{1})$. 
In particular, this implies that $(A_h)_{(|x'|, y, t)} \neq \emptyset$, and so
 there exists $x \in (A_h)_{(|x'|, y, t)}$ such that $(x, y, t) \in A_h$. Then, by definition of $A_h$, 
 \[
 \varphi_h (x, y, t) = \max_{1 \leq i \leq k}  \varphi_i (x, y, t).
 \]
 On the other hand, since $x, x' \in (\partial^* \Sigma^{u_0})_{(|x'|, y, t)}$, by  \eqref{constant along slices}
 we have $\varphi_i (x', y, t) = \varphi_i (x, y, t)$ for every $i = 1, \ldots, k$. Therefore, 
 \[
\varphi_h (x', y, t) = \varphi_h (x, y, t) = \max_{1 \leq i \leq k}  \varphi_i (x, y, t)
= \max_{1 \leq i \leq k}  \varphi_i (x', y, t),
\]
so that $(x', y, t) \in A_h$, and this shows \eqref{annulus h}.

\vspace{.2cm}

\noindent
\textbf{Step 2c:} We conclude the proof of Step~2. For each $h = 1, \ldots, k$, we can find a Borel set $R_h$ such that $\sigma (S(A_h)  \Delta R_h) = 0$. 
Note now that, by definition of $\sigma$, 
\begin{align*}
0 &= \sigma (R_h \setminus  S(A_h) ) 
= \lambda (S^{-1} (R_h \setminus  S(A_h))) 
= \lambda (
\Phi_{n+1} ((R_h \setminus  S(A_h)) \times \mathbb{S}^{1}) \\
&= \mathcal{H}^n \left(\partial^* \Sigma^{u_0} \cap \Phi_{n+1} ((R_h  \setminus  S(A_h)) \times \mathbb{S}^{1}) \right).
\end{align*}
Combining the above identity with \eqref{def Ah} and \eqref{annulus h}, we obtain that for every $h = 1, \ldots, k$
\begin{equation} \label{union of Rh}
\varphi_h (x, y, t) = \max_{1 \leq i \leq k}  \varphi_i (x, y, t), 
\quad \text{ for $\mathcal{H}^n$-a.e. }  (x, y, t) \in \partial^* \Sigma^{u_0} \cap \Phi_{n+1} (R_h \times \mathbb{S}^{1}).
\end{equation}
Therefore, the Borel sets $R_1, \ldots, R_h$ are such that \eqref{union of Rh} holds, and
\[
\bigcup_{i=1}^h R_i = B' := \left\{ (r, y, t) \in B : (\partial^* \Sigma^{u_0})_{(r, y, t)} \neq \emptyset \right\}.
\]
We now set $C_1:=  R_1$,
\[ 
C_h := R_h \setminus 
\bigg( \bigcup_{1 \leq i < h} R_i \bigg), \quad h = 2, \ldots, k-1, \qquad \qquad
R_k:= B \setminus \bigg( \bigcup_{1 \leq i < k} R_i \bigg),
\]
and the claim follows.

\vspace{.2cm}

\noindent
\textbf{Step 3:} We conclude. Thanks to Step~1 and Step~2, 
\begin{align*}
&\int_{\partial^* \Sigma^{u_0} \cap \Phi_{n+1} (B \times \mathbb{S}^{1}) }  
\max_{1 \leq h \leq k} 
  \varphi_h (x, y, t) \, d \mathcal{H}^n (x, y, t) \\
 &= \sum_{h = 1}^k 
 \int_{\partial^* \Sigma^{u_0} \cap \Phi_{n+1} (C_h \times \mathbb{S}^{1})} 
 \varphi_h (x, y, t) \, d \mathcal{H}^n (x, y, t) \\
&=  \sup \left\{\sum_{h = 1}^k 
 \int_{\partial^* \Sigma^{u_0} \cap \Phi_{n+1} (B_h \times \mathbb{S}^{1})} 
 \varphi_h (x, y, t) \, d \mathcal{H}^n (x, y, t) \right\},
\end{align*}
where the last supremum ranges over all the $k$-ples $B_1, \ldots, B_k$
of pairwise disjoint Borel partitions of $B$.
This shows \eqref{before k goes to infinity} and, in turn, the lemma.

\end{proof}
\noindent
We can now state a general P\'olya-Szeg\"o inequality under circular rearrangement, which is the main result of this section.
\begin{theorem} \label{the thing}
Let $n \in \mathbb{N}$ with $n \geq 2$, let $\Omega \subset \R^n$ be open, and let $\mu \in \mathcal{A} (\Pi_{n-1} (\Omega) \times \R)$.
Let $a \in L^{\infty} ((0, \infty) \times \R^{n-2} \times \R)$ with $a \geq 0$ $\mathcal{H}^n$-a.e., 
and let $f \in \mathscr{F}$.
Then, for every $\mu$-distributed function $u \in W^{1, 1}_{0, \tau} (\Omega)$ we have
\begin{equation} \label{gen polya szego}
\begin{split}
&\int_{\partial^* \Sigma^{v_\mu} \cap \left( \Phi_{n+1} (B \times \mathbb{S}^{1}) \right)}  
 \hspace{-1cm}
  \frac{a (|x|, y, v_\mu) f (\hat{x} \cdot \nabla_x v_\mu ,  \Dx v_\mu , 
 \nabla_y v_\mu )}{\sqrt{1 + |\nabla v_\mu |^2}} \, d \mathcal{H}^n (x, y, t)  \\[3pt]
& \leq \int_{\partial^* \Sigma^{u_0} \cap \left( \Phi_{n+1} (B \times \mathbb{S}^{1}) \right)}  
 \hspace{-1cm} \frac{a (|x|, y, u_0) f (\hat{x} \cdot \nabla_x u_0 ,  \Dx u_0 , 
 \nabla_y u_0 )}{\sqrt{1 + |\nabla u_0 |^2}} \, d \mathcal{H}^n (x, y, t), 
 \end{split}
\end{equation}
for every Borel set $B \subset \Pi_{n-1} (\Omega) \times \mathbb{R}$.
\end{theorem}
\noindent
Before discussing the proof, we show some consequences of this result.

\subsection{Consequences of Theorem~\ref{the thing}}
We start by showing that Theorem~\ref{the thing} 
implies Theorem~\ref{thm_Polya-Szego inequality local}.
\begin{proof}[Proof of Theorem~\ref{thm_Polya-Szego inequality local}]
The result follows immediately from \eqref{gen polya szego}, choosing Borel sets 
of the type $B = \tilde{B} \times \mathbb{R}$ with $\tilde{B} \subset \Pi_{n-1} (\Omega)$.
\end{proof}
The next result is a consequence of Theorem~\ref{thm_Polya-Szego inequality local}.
\begin{corollary} \label{symmetric function in w1p}
Let $n \in \mathbb{N}$ with $n \geq 2$, let $\Omega \subset \mathbb{R}^n$ be open, and let $p \in [1, \infty)$. 
Let $u \in W^{1,p}_{0, \tau} (\Omega)$, let $\mu$ be given by \eqref{mu distr}, and let $v_{\mu}$ be defined by \eqref{def vmu}.
Then, $v_\mu \mid_{\Omega^s} \in W^{1, p}_{0, \tau} (\Omega^s)$.
\end{corollary}

\begin{proof}
The case $p = 1$ was already considered in Proposition~\ref{ us sobolev}, 
so let us assume $p >1$,  and let $u \in W^{1, p}_{0, \tau} (\Omega)$.
By Remark~\ref{rem holder} it follows that $u \in W^{1, 1}_{0, \tau} (\Omega)$ and so, 
by Proposition~\ref{ us sobolev}, we have $v_\mu \mid_{\Omega^s} \in W^{1, 1}_{0, \tau} (\Omega^s)$.

Let now $A \subset \subset \Pi_{n-1} (\Omega)$ be open.
Fom Proposition~\ref{prop Lp norm} it follows that $v_\mu \in L^p (\Phi_n (A \times \mathbb{S}^1))$.
Moreover, Applying Theorem~\ref{thm_Polya-Szego inequality local} with $a \equiv 1$, $B=A$ and 
\[
f (\eta, \zeta, \tau) = \left( \eta^2 + \zeta^2 + | \tau |^2 \right)^{p/2},
\] 
we obtain that $v_\mu \in W^{1,p} (\Phi_n (A \times \mathbb{S}^1))$. Since $A$ was arbitrary, 
thanks to \eqref{v_mu is zero a.e. outside} we conclude.
\end{proof}
\noindent
Let us now show that Theorem~\ref{thm_Polya-Szego inequality local} implies Theorem~\ref{thm_Polya-Szego inequality}.
\begin{proof}[Proof of Theorem~\ref{thm_Polya-Szego inequality}]
Let $A \subset \subset \Pi_{n-1} (\Omega)$ be a Borel set. 
From Theorem~\ref{thm_Polya-Szego inequality local} it follows that
\begin{equation} \label{ineq A large}
\begin{split}
&\int_{\Phi_n ( A \times \mathbb{S}^1)} a (|x|, y, v_\mu) f (\hat{x} \cdot \nabla_x v_\mu ,  \Dx v_\mu,  \nabla_y v_\mu ) \, dx dy \\
&\hspace{.5cm}\leq
\int_{\Phi_n ( A \times \mathbb{S}^1)} a (|x|, y, u_0) f (\hat{x} \cdot \nabla_x u_0 ,  \Dx u_0,  \nabla_y u_0 ) \, dx dy,
\end{split}
\end{equation}
where this is an inequality between extended nonnegative real numbers.

Note now that, thanks to \eqref{us is zero a.e. outside}, we have $v_{\mu} = 0$ 
$\mathcal{H}^n$-a.e. in $\Phi_n (\Pi_{n-1}  (\Omega) \times \mathbb{S}^1) \setminus \Omega^s$. 
Then (see, for instance, \cite[Section 4.2.2, Theorem 4, part iv]{EvaGa}) it follows that $\nabla v_{\mu} = 0$ $\mathcal{H}^n$-a.e. in $\Phi_n (\Pi_{n-1}  (\Omega) \times \mathbb{S}^1) \setminus \Omega^s$. For the same reason, since $u_0 = 0$ $\mathcal{H}^n$-a.e. in $\Phi_n (\Pi_{n-1}  (\Omega) \times \mathbb{S}^1) \setminus \Omega$, 
we have $\nabla u_0 = 0$ $\mathcal{H}^n$-a.e. in $\Phi_n (\Pi_{n-1}  (\Omega) \times \mathbb{S}^1) \setminus \Omega$.
Therefore, 
\begin{equation} \label{ineq A large setminus}
\begin{split}
&\int_{\Phi_n ( A \times \mathbb{S}^1) \setminus \Omega^s} a (|x|, y, v_\mu) f (\hat{x} \cdot \nabla_x v_\mu ,  \Dx v_\mu,  \nabla_y v_\mu ) \, dx dy \\
&\hspace{.5cm}= f ( 0 ,  0,  0 ) \int_{\Phi_n ( A \times \mathbb{S}^1) \setminus \Omega^s} a (|x|, y, 0)  \, dx dy \\
&\hspace{.5cm}= f ( 0 ,  0,  0 ) \int_{A} a (r, y, 0) \, \left( \int_{\partial D (r) \setminus (\Omega^s)_{(r, y)}} 
 d \mathcal{H}^1 (x) \right)\, d r \, dy \\
  &\hspace{.5cm}= f ( 0 ,  0,  0 ) \int_{A} a (r, y, 0)  \Big( 2 \pi r - \mathcal{H}^1 \left( (\Omega^s)_{(r, y)} \right) \Big) \, d r \, dy \\
  &\hspace{.5cm}= f ( 0 ,  0,  0 ) \int_{A} a (r, y, 0)  \Big( 2 \pi r - \mathcal{H}^1 \left( (\Omega)_{(r, y)} \right) \Big) \, d r \, dy \\
  &\hspace{.5cm}= f ( 0 ,  0,  0 ) \int_{\Phi_n ( A \times \mathbb{S}^1) \setminus \Omega} a (|x|, y, 0)  \, dx dy \\
  &\hspace{.5cm}=
  \int_{\Phi_n ( A \times \mathbb{S}^1) \setminus \Omega} a (|x|, y, u_0) f (\hat{x} \cdot \nabla_x u_0 ,  \Dx u_0,  \nabla_y u_0 ) \, dx dy, 
\end{split}
\end{equation}
where we used the fact that by definition of $\Omega^s$ we have 
\[
\mathcal{H}^1 \left( (\Omega^s)_{(r, y)} \right)
= \mathcal{H}^1 \left( (\Omega)_{(r, y)} \right) \quad \text{ for every } (r, y) \in A.
\]
Note that, since $A \subset \subset \Pi_{n-1} (\Omega)$ and $a$ is bounded, all the integrals appearing in \eqref{ineq A large setminus}
 are finite. Therefore, we can combine \eqref{ineq A large} and \eqref{ineq A large setminus}, obtaining 
\begin{align*} 
& \int_{\Phi_n ( A \times \mathbb{S}^1) \cap\Omega^s} a (|x|, y, v_\mu) f (\hat{x} \cdot \nabla_x v_\mu ,  \Dx v_\mu,  \nabla_y v_\mu ) \, dx dy  \\
& \leq \int_{\Phi_n ( A \times \mathbb{S}^1) \cap \Omega} a (|x|, y, u_0) f (\hat{x} \cdot \nabla_x u_0 ,  \Dx u_0,  \nabla_y u_0 ) \, dx dy,
\end{align*}
for every Borel set $A \subset \subset \Pi_{n-1} (\Omega)$.
By considering a sequence $\{ A_j \}_{j \in \mathbb{N}}$ of Borel sets with $A_j \subset \subset \Pi_{n-1} (\Omega)$
and $A_j  \nearrow \Pi_{n-1} (\Omega)$, we conclude.
\end{proof}


\noindent
We can now give the proof of Theorem~\ref{the thing}.
\begin{proof}[Proof of Theorem~\ref{the thing}]
First of all, note that it is not restrictive to assume 
\begin{equation} \label{B nice}
B \subset \{ \alpha_\mu^{\vee} > 0 \} \cap \{ \alpha_\mu^{\wedge} < \pi \}.
\end{equation}
Indeed, if not, one can split $B$ as the disjoint union $B = B_1 \cup B_2$, where
\[
B_1 = B \cap  \{ \alpha_\mu^{\vee} > 0 \} \cap \{ \alpha_\mu^{\wedge} < \pi \}, 
\quad B_2 = B \setminus ( \{ \alpha_\mu^{\vee} > 0 \} \cap \{ \alpha_\mu^{\wedge} < \pi \} ),
\] 
and observe that, thanks to \eqref{the only parts that count}, inequality \eqref{gen polya szego} is trivially satisfied in  $B_2$. 
Let's then assume that \eqref{B nice} is satisfied.
We can also suppose $B \subset A \times \mathbb{R}$ for some open set $A \subset \subset \Pi_{n-1} (\Omega)$, 
since the general case can be obtained by approximation.

By assumption, $u_0 \in W^{1, 1} (\Phi_{n} (A \times \mathbb{S}^{1}))$ and 
thanks to Proposition~\ref{ us sobolev} $v_{\mu} \in W^{1, 1} (\Phi_{n} (A \times \mathbb{S}^{1}))$.
Let $G_{v_\mu}$ and $G_{u_0}$ 
denote the sets given by Proposition~\ref{volpert theorem} (note that in general both $G_{v_\mu}$ and $G_{u_0}$ will depend on $A$). 
Then, setting $B_{v_{\mu}, u_0}:= B \cap G_{v_\mu} \cap G_{u_0}$ we have
\begin{align*}
&\int_{\partial^* \Sigma^{v_\mu} \cap \left( \Phi_{n+1} (B \times \mathbb{S}^{1}) \right)}   \hspace{-1.5cm}
 \frac{a (|x|, y, v_\mu) f (\hat{x} \cdot \nabla_x v_\mu,  \Dx v_\mu , 
 \nabla_y v_\mu )}{\sqrt{1 + |\nabla v_\mu |^2}} \, d \mathcal{H}^n (x, y, t) \\
 &= \int_{\partial^* \Sigma^{v_\mu} \cap \left( \Phi_{n+1} ( B  \times \mathbb{S}^{1}) \right)}   \hspace{-1.5cm}
 \frac{a (|x|, y, t) f (\hat{x} \cdot \nabla_x v_\mu,  \Dx v_\mu , 
 \nabla_y v_\mu )}{\sqrt{1 + |\nabla v_\mu|^2}} \, d \mathcal{H}^n (x, y, t) = I_1 + I_2,
 \end{align*}
where  
 \begin{align*}
I_1 &:= \int_{\partial^* \Sigma^{v_\mu} \cap \left( \Phi_{n+1} (B_{v_{\mu}, u_0}\times \mathbb{S}^{1}) \right)}  
\hspace{-1.5cm}
 \frac{a (|x|, y, t) f (\hat{x} \cdot \nabla_x v_\mu,  \Dx v_\mu , 
 \nabla_y v_\mu )}{\sqrt{1 + |\nabla v_\mu |^2}} \, d \mathcal{H}^n (x, y, t), \\[4pt]
 I_2 &:= \int_{\partial^* \Sigma^{v_\mu} \cap \left( \Phi_{n+1} \left(
\big( B \setminus B_{v_{\mu}, u_0}
\big)  \times \mathbb{S}^{1}
\right) \right)}  
\hspace{-1.5cm}
 \frac{a (|x|, y, t) f (\hat{x} \cdot \nabla_x v_\mu,  \Dx v_\mu , 
 \nabla_y v_\mu )}{\sqrt{1 + |\nabla v_\mu |^2}} \, d \mathcal{H}^n (x, y, t),
 \end{align*}
and where we used the fact that, since $v_{\mu} \in W^{1, 1} (\Phi_{n} (A \times \mathbb{S}^{1}))$, 
thanks to Proposition~\ref{giaquinta} we have \( v^{\wedge}_\mu (x, y) = v^{\vee}_\mu (x, y) = t\)
for $\mathcal{H}^n$-a.e. $(x, y, t) \in \partial^* \Sigma^{v_\mu}$. We divide the proof into two steps.
 
\vspace{.2cm}

\noindent
\textbf{Step 1:} We show that 
\begin{align} \label{estimate on I_1}
I_1 \leq \int_{\partial^* \Sigma^{u_0} \cap \left( \Phi_{n+1} (B_{v_{\mu}, u_0} \times \mathbb{S}^{1}) \right)}  
\hspace{-1.2cm} \frac{a (|x|, y, u_0) f (\hat{x} \cdot \nabla_x u_0,  \Dx u_0 , 
 \nabla_y u_0 )}{\sqrt{1 + |\nabla u_0|^2}} \, d \mathcal{H}^n (x, y, t).
\end{align}
Taking into account property (f2) of Definition~\ref{def_admissable integrands}, we have
\begin{equation} \label{integral I1}
\begin{split}
I_1 &= \int_{\partial^* \Sigma^{v_\mu} \cap \left( \Phi_{n+1} (B_{v_{\mu}, u_0} \times \mathbb{S}^{1}) \cap (\{ \Dx v_\mu \neq 0\} \times \mathbb{R}) \right)}   \hspace{-1.5cm}
 \frac{a (|x|, y, t) f (\hat{x} \cdot \nabla_x v_\mu,  \Dx v_\mu , 
 \nabla_y v_\mu )}{\sqrt{1 + |\nabla v_\mu|^2}} \, d \mathcal{H}^n (x, y, t)  \\[4pt]
& = \int_{B_{v_{\mu}, u_0}} a (r, y, t) \int_{\partial^* ( ( \Sigma^{v_\mu} )_{(r, y, t)} )}
 \hspace{-1cm} \frac{f (\hat{x} \cdot \nabla_x v_\mu,  | \Dx v_\mu | , 
 \nabla_y v_\mu )}{| \Dx v_\mu|}   \, d \mathcal{H}^0 (x) \, dr \, dy \, dt  \\ 
 &= \int_{B_{v_{\mu}, u_0}}
 a (r, y, t) \, ( - \partial_t \mu (r, y, t) )
  f \left( \frac{r \partial_r \xi_\mu (r,y, t)}{- \partial_t \mu (r, y, t)}  , 
 \frac{2}{- \partial_t \mu (r, y, t)} , \frac{\nabla_y \mu (r,y, t)}{- \partial_t \mu (r, y, t)} \right) 
 \, dr \, dy \, dt, 
\end{split}
\end{equation}
where the first equality is due to Remark~\ref{remark us constant functions new}, 
the second one to Proposition~\ref{coarea formula}, Remark~\ref{rem parallel normal if u sobolev}
and Proposition~\ref{volpert theorem}, and in the last one we used Remark~\ref{this is now the magic property}, Remark~\ref{nice derivatives}, and the fact that $\mathcal{H}^0 \left(  \partial^* ( ( \Sigma^{v_{\mu}} )_{(r, y, t)} ) \right) = 2$ for $\mathcal{H}^n$-a.e. $(r, y, t) \in B_{v_{\mu}, u_0}$.

Now we would like to compare the integral above with an integral 
involving the function $u_0$. To this aim, we first observe that, thanks to \eqref{isop ineq},
\begin{equation} \label{isop ineq circle}
2 \leq \mathcal{H}^0 \left(  \partial^* ( ( \Sigma^{u_0} )_{(r, y, t)} ) \right), \quad \text{ for $\mathcal{H}^n$-a.e. $(r, y, t) \in B_{v_{\mu}, u_0}$.}  
\end{equation}
Let us now introduce, for every $(r, y, t) \in B_{v_{\mu}, u_0}$, 
the probability measure $\lambda_{r, y}^t$ on $( \partial^{*} \Sigma^{u_0})_{(r,y, t)}$ given by
\[
\lambda_{r, y}^t:=   
\frac{\displaystyle \frac{1}{|\Dx u_0 (x, y) |} \, \mathcal{H}^0 \mres ( \partial^{*} \Sigma^{u_0})_{(r,y, t)}}  
{\displaystyle \int_{( \partial^{*} \Sigma^{u_0})_{(r,y, t)}} \frac{1}{|\Dx u_0 (x, y) |} \, d \mathcal{H}^0 (x)}.
\]
In the following, to ease the notation, 
in the integrals below we drop the integration set 
$( \partial^{*} \Sigma^{u_0})_{(r,y, t)}$.
Using Lemma~\ref{lemma equivalent increasing}
and \eqref{isop ineq circle},  together with Proposition~\ref{abs cont parts}, 
for $\mathcal{H}^n$-a.e. $(r, y, t) \in B_{v_{\mu}, u_0}$ one has 

\smallskip

\begin{equation} \label{first inequality polya}
\begin{split}
&a (r, y, t) \, ( - \partial_t \mu (r, y, t) )
  f \left( \frac{r \partial_r \xi_\mu (r,y, t)}{- \partial_t \mu (r, y, t)}  , 
 \frac{2}{- \partial_t \mu (r, y, t)} , \frac{\nabla_y \mu (r,y, t)}{- \partial_t \mu (r, y, t)} \right)  \\
 & \\
 &=a (r, y, t) \left( \int \frac{d \mathcal{H}^0 (x)}{|\Dx u_0 |}   \right)
   f \left( \frac{\left( \displaystyle \int   \frac{\hat{x} \cdot \nabla_x u_0}{|\Dx u_0 |} \, d \mathcal{H}^0 (x), 
 2,   
 \int \frac{\nabla_y u_0}{|\Dx u_0 |} \, d \mathcal{H}^0 (x)
 \right)}{\displaystyle \int \frac{\, d \mathcal{H}^0 (x)}{|\Dx u_0 |} }
  \right)  \\
  &  \\
  &\leq a (r, y, t) \left( \int \frac{d \mathcal{H}^0 (x)}{|\Dx u_0 |}   \right)
   f \left( \frac{\left( \displaystyle \int   \frac{\hat{x} \cdot \nabla_x u_0}{|\Dx u_0 |} \, d \mathcal{H}^0 (x), 
 \int \, d \mathcal{H}^0 (x) ,   
 \int \frac{\nabla_y u_0}{|\Dx u_0 |} \, d \mathcal{H}^0 (x)
 \right)}{\displaystyle \int \frac{\, d \mathcal{H}^0 (x)}{|\Dx u_0 |} }
  \right) \\
  &  \\
&= a (r, y, t) \left( \int \frac{d \mathcal{H}^0 (x)}{|\Dx u_0 |}  \right) 
  f \left( \int \hat{x} \cdot \nabla_x u_0 \, d \lambda_{r, y}^t (x), \int  |\Dx u_0 | \, d \lambda_{r, y}^t (x),
  \int  \nabla_y u_0 \, d \lambda_{r, y}^t (x) \right).
\end{split}
 \end{equation}
Thanks to properties (f1) of Definition~\ref{def_admissable integrands}, we can apply Jensen's inequality. Therefore, using also property (f2) of Definition~\ref{def_admissable integrands}, we obtain
\begin{align}
&a (r, y, t) \, ( - \partial_t \mu (r, y, t) )
   f \left( \frac{r \partial_r \xi (r,y, t)}{- \partial_t \mu (r, y, t)}  , 
 \frac{2}{- \partial_t \mu (r, y, t)} , \frac{\partial_y \mu (r,y, t)}{- \partial_t \mu (r, y, t)} \right) \nonumber \\
&\leq a (r, y, t) \left( \int \frac{d \mathcal{H}^0 (x)}{|\Dx u_0 |}  \right) 
 \int f ( \hat{x} \cdot \nabla_x u_0,  | \Dx u_0 |, 
 \nabla_y u_0 )  \, d \lambda_{r, y}^t (x) \label{jensen ineq}  \\
    &= a (r, y, t) \int_{( \partial^{*} \Sigma^{u_0})_{(r,y, t)}} \frac{f (  \hat{x} \cdot \nabla_x u_0,  \Dx u_0, 
 \nabla_y u_0 ) }{|\Dx u_0 |} \, d \mathcal{H}^0 (x), \nonumber
 \end{align}
for $\mathcal{H}^n$-a.e. $(r, y, t) \in B_{v_{\mu}, u_0}$.
From this and \eqref{integral I1}, 
using again Proposition~\ref{coarea formula}, it follows that
\begin{align*}
I_1 &\leq \int_{B_{v_{\mu}, u_0}} a (r, y,  t)
\int_{( \partial^{*} \Sigma^{u_0})_{(r,y, t)}} 
\hspace{-1cm} \frac{f ( \hat{x} \cdot \nabla_x u_0,  \Dx u_0, 
 \nabla_y u_0 ) }{|\Dx u_0 |} \, d \mathcal{H}^0 (x)
\, dr \, dy \, dt \\
&=\int_{\partial^* \Sigma^{u_0} \cap \left( \Phi_{n+1} (B_{v_{\mu}, u_0} \times \mathbb{S}^{1}) \right)}  
\hspace{-1.2cm} \frac{ a ( |x|, y, u_0 ) f ( \hat{x} \cdot \nabla_x u_0,  \Dx u_0, 
 \nabla_y u_0 ) }{\sqrt{1 + |\nabla u_0|^2}} \, d \mathcal{H}^n (x, y, t), 
\end{align*}
thus showing \eqref{estimate on I_1}.

\vspace{.2cm}

\noindent
\textbf{Step 2:} We conclude, showing that
\begin{equation} \label{estimate on I_2}
I_2  \leq \int_{\partial^* \Sigma^{u_0} \cap \left( \Phi_{n+1} \left(
\big( B \setminus B_{v_{\mu}, u_0}
\big)  \times \mathbb{S}^{1}
\right)  \right)}
 \hspace{-2.3cm} \frac{ a (|x|, y, u_0) \, f ( \hat{x} \cdot \nabla_x u_0 (x, y),  \Dx u_0 (x, y), 
 \nabla_y u_0 (x, y) ) }{\sqrt{1 + |\nabla u_0 (x, y)|^2}} \, d \mathcal{H}^n (x, y, t).
\end{equation}
Recalling the definition of $B_{v_{\mu}, u_0}$ and the properties of the sets $G_{v_\mu}$ and $G_{u_0}$, from \eqref{B nice} it follows that
$\mathcal{H}^n (B \setminus B_{v_{\mu}, u_0}) = 0$. Therefore, 
applying \eqref{only tangential non zero} to $v_\mu$ we obtain 
\begin{align*}
I_2 &= \int_{\partial^* \Sigma^{v_\mu} \cap \left( \Phi_{n+1} \left(
\big( B \setminus B_{v_{\mu}, u_0}
\big)  \times \mathbb{S}^{1}
\right) \right)}  
\hspace{-1.5cm}
 \frac{a (|x|, y, t) f (\hat{x} \cdot \nabla_x v_\mu,  \Dx v_\mu , 
 \nabla_y v_\mu )}{\sqrt{1 + |\nabla v_\mu |^2}} \, d \mathcal{H}^n (x, y, t) \\[4pt]
 &= \int_{\partial^* \Sigma^{v_\mu} \cap \left( \Phi_{n+1} \left(
\big( B \setminus B_{v_{\mu}, u_0}
\big)  \times \mathbb{S}^{1}
\right) \right)}  
\hspace{-1.2cm}
 \frac{a (|x|, y, t) f (\hat{x} \cdot \nabla_x v_\mu,  0 , 
 \nabla_y v_\mu )}{\sqrt{1 + |\nabla v_\mu |^2}} \, d \mathcal{H}^n (x, y, t)
 \\[4pt]
 &= \int_{\partial^* \Sigma^{v_\mu} \cap \left( \Phi_{n+1} \left(
\big( B \setminus B_{v_{\mu}, u_0}
\big)  \times \mathbb{S}^{1}
\right) \right)}  
\,  \sup_{h \in \mathbb{N}} \varphi_h (x, y, t)
 \,  d \mathcal{H}^n (x, y, t),
\end{align*}
where for every $h \in \mathbb{N}$ we set
\[
\varphi_h (x, y, t) := a (|x|, y, t ) \frac{w_h \cdot (\hat{x} \cdot \nabla_x v_{\mu} (x, y), 0,  \nabla_y v_{\mu} (x, y) ) 
- f^* (w_h) }{\sqrt{1 + |\nabla v_{\mu} (x, y)|^2}},
\]
and we used \eqref{leg transf twice}, where $\{ w_h \}_{h \in \mathbb{N}}$ is a countable dense subset of
$\{ w \in \R^n: f^* (w) < + \infty \}$. 
Note now that 
\[
|\varphi_1 (x, y, t)| \leq a (|x|, y, t ) \frac{|w_1|  | \nabla v_{\mu} (x, y)| 
+ | f^* (w_1) |}{\sqrt{1 + |\nabla v_{\mu} (x, y)|^2}} \leq  C ( |w_1| + | f^* (w_1) | ),
\]
where $C > 0$ is such that $|a| \leq C$. Then,  
\begin{align*}
&\int_{\partial^* \Sigma^{v_\mu} \cap \left( \Phi_{n+1} \left(
\big( B \setminus B_{v_{\mu}, u_0}
\big)  \times \mathbb{S}^{1}
\right) \right)} | \varphi_1 (x, y, t)| \,  d \mathcal{H}^n (x, y, t) \\[3pt]
&\hspace{.5cm}\leq C ( |w_1| + | f^* (w_1) | ) \mathcal{H}^n \left( \partial^* \Sigma^{v_\mu} \cap \left( \Phi_{n+1} ( B  \times \mathbb{S}^{1}) \right)  \right) \\[3pt]
&\hspace{.5cm}\leq \mathcal{H}^n \left( \partial^* \Sigma^{v_\mu} \cap \left( \Phi_{n+1} ( (A \times \R)  \times \mathbb{S}^{1}) \right)  \right) < \infty.
\end{align*}
Thanks to Remark~\ref{this is now the magic property}, we can apply Lemma~\ref{thanks for this lemma}. 
Therefore, we have 
\begin{align*}
I_2 &= \sup_{H} \left\{ \sum_{h\in H} 
 \int_{\partial^* \Sigma^{v_{\mu}} \cap \Phi_{n+1} (B_h \times \mathbb{S}^{1})} 
 \varphi_h (x, y, t) \, d \mathcal{H}^n (x, y, t) \right\}, 
\end{align*}
where the supremum ranges over all finite sets $H\subset \mathbb{N}$ 
and over all  partitions $\{ B_h \}_{h\in H}$ of $B \setminus B_{u, v_{\mu}}$ composed of pairwise disjoint Borel sets.
Let now $H\subset \mathbb{N}$ be a finite set, and let $h \in H$.
Thanks to Proposition~\ref{derivatives of mu}, applied first to $v_{\mu}$ and then to $u$, and recalling \eqref{normal1}, we have
 \begin{align*}
&  \int_{\partial^* \Sigma^{v_{\mu}} \cap \Phi_{n+1} (B_h \times \mathbb{S}^{1})} 
 \varphi_h (x, y, t) \, d \mathcal{H}^n (x, y, t) \\
 &=  \int_{\partial^* \Sigma^{v_{\mu}} \cap \Phi_{n+1} (B_h \times \mathbb{S}^{1})} a( |x|, y, t )
 \frac{w_h \cdot (\hat{x} \cdot \nabla_x v_{\mu} , 0,  \nabla_y v_{\mu}  ) -  f^* ( w_h)}{\sqrt{1 + |\nabla v_{\mu} |^2}} \, d \mathcal{H}^n (x, y, t) \\
 &= \int_{\partial^* \Sigma^{v_{\mu}} \cap \Phi_{n+1} (B_h \times \mathbb{S}^{1})} a( |x|, y, t ) 
 \left( w_h \cdot   \Big( 
 \hat{x} \hspace{-.05cm} \cdot \hspace{-.05cm} \nu^{\Sigma^{v_\mu}}_x (x, y, t), 0, \nu_y^{\Sigma^{v_\mu}} (x, y, t)
 \Big) \right. \\
 &\left. \hspace{1cm}+ f^* (w_h) \, \nu_t^{\Sigma^{v_\mu}} (x, y, t) \right) 
  \, d \mathcal{H}^n (x, y, t) \\
 &=  
 w_h \cdot   \left(
  \int_{B_h} a( r, y, t ) \, r  \, d D_r \xi_\mu (r, y, t)
 , 0 , \int_{B_h} a( r, y, t ) \,  d D_y \mu (r, y, t)
  \right)  \\
&\hspace{1cm}+ \int_{B_h} a( r, y, t ) f^* (w_h) \, d D_t \mu (r, y, t) \\
&=  \int_{\partial^* \Sigma^{u_0} \cap \Phi_{n+1} (B_h \times \mathbb{S}^{1})} a( |x|, y, t )
 \frac{w_h \cdot (\hat{x} \cdot \nabla_x u_0 , 0,  \nabla_y u_0  ) -  f^* ( w_h)}{\sqrt{1 + |\nabla u_0 |^2}} 
 \, d \mathcal{H}^n (x, y, t) \\
 &\leq \int_{\partial^* \Sigma^{u_0} \cap \left( \Phi_{n+1} (B_h \times \mathbb{S}^{1}) \right)}
 \hspace{-1cm} \frac{a(|x|, y, t)  f (\hat{x} \cdot \nabla_x u_0, 0, 
 \nabla_y u_0) }{\sqrt{1 + |\nabla u_0|^2}} \, d \mathcal{H}^n (x, y, t),
 \end{align*}
where the last inequality follows from \eqref{leg transf twice}. 
From this, it follows that
 \begin{align*}
&\sum_{h\in H}  \int_{\partial^* \Sigma^{v_{\mu}} \cap \Phi_{n+1} (B_h \times \mathbb{S}^{1})} 
 \varphi_h (x, y, t) \, d \mathcal{H}^n (x, y, t) \\
& \leq  \sum_{h\in H}
 \int_{\partial^* \Sigma^{u_0} \cap \left( \Phi_{n+1} (B_h \times \mathbb{S}^{1}) \right)}
 \hspace{-1cm} \frac{a(|x|, y, t)  f (\hat{x} \cdot \nabla_x u_0, 0, 
 \nabla_y u_0) }{\sqrt{1 + |\nabla u_0|^2}} \, d \mathcal{H}^n (x, y, t) \\
 &=  \int_{\partial^* \Sigma^{u_0} \cap \left( \Phi_{n+1} \left(
\big( B \setminus B_{v_{\mu}, u_0}
\big)  \times \mathbb{S}^{1}
\right) \right)} 
 \hspace{-1cm} \frac{ a(|x|, y, t) f (\hat{x} \cdot \nabla_x u_0 ,  0, 
 \nabla_y u_0 ) }{\sqrt{1 + |\nabla u_0 |^2}} \, d \mathcal{H}^n (x, y, t) \\
&=  \int_{\partial^* \Sigma^{u_0} \cap \left( \Phi_{n+1} \left(
\big( B \setminus B_{v_{\mu}, u_0}
\big)  \times \mathbb{S}^{1}
\right) \right)} 
 \hspace{-1cm} \frac{ a(|x|, y, t) f (\hat{x} \cdot \nabla_x u_0,  \Dx u_0, 
 \nabla_y u_0  ) }{\sqrt{1 + |\nabla u_0|^2}} \, d \mathcal{H}^n (x, y, t) ,
\end{align*}
where in the last equality we applied \eqref{only tangential non zero}
to $u_0$ with $R = B \setminus B_{v_{\mu}, u_0}$.
Taking the supremum with respect to $H$ we obtain \eqref{estimate on I_2} which, combined with \eqref{estimate on I_1}, 
allows us to conclude.
\end{proof}

\section{Sufficient conditions for rigidity} \label{sect rigidity}

In this section we give the proof of Theorem~\ref{thm suff conditions}.
We start by considering inequality \eqref{gen polya szego} with $B = \Pi_{n-1} (\Omega) \times \R$:
\begin{equation} \label{gen polya szego full set}
\begin{split}
&\int_{\partial^* \Sigma^{v_\mu} \cap \left( \Phi_{n+1} ((\Pi_{n-1} (\Omega) \times \R)\times \mathbb{S}^{1}) \right)}  
  \hspace{-0.3cm}  \frac{a (|x|, y, v_\mu) f (\hat{x} \cdot \nabla_x v_\mu ,  \Dx v_\mu , 
\nabla_y v_\mu )
}{\sqrt{1 + |\nabla v_\mu |^2}} \, d \mathcal{H}^n (x, y, t)  \\[4pt]
& \leq \int_{\partial^* \Sigma^{u_0} \cap \left( \Phi_{n+1} ((\Pi_{n-1} (\Omega) \times \R)\times \mathbb{S}^{1}) \right)}  
 \hspace{-0.3cm}
   \frac{a (|x|, y, u_0) f (\hat{x} \cdot \nabla_x u_0 ,  \Dx u_0 , 
 \nabla_y u_0 )  }{\sqrt{1 + |\nabla u_0 |^2}} \, d \mathcal{H}^n (x, y, t),
\end{split}
\end{equation}
under the assumption that
\begin{equation} \label{finite integral}
\int_{\partial^* \Sigma^{v_\mu} \cap \left( \Phi_{n+1} ((\Pi_{n-1} (\Omega) \times \R)\times \mathbb{S}^{1}) \right)}  
  \frac{a (|x|, y, v_\mu) f( \hat{x} \cdot \nabla_x v_\mu ,  \Dx v_\mu , 
 \nabla_y v_\mu )}{\sqrt{1 + |\nabla v_\mu |^2}} \, d \mathcal{H}^n (x, y, t) < \infty.
\end{equation}
First of all, we define the set of extremals in \eqref{gen polya szego full set}.
\begin{definition}
Let $n \in \mathbb{N}$ with $n \geq 2$, let $\Omega \subset \R^n$ be open, and let $\mu \in \mathcal{A} (\Pi_{n-1} (\Omega) \times \R)$.
Let $a \in L^{\infty} ((0, \infty) \times \R^{n-2} \times \R)$ with $a > 0$ $\mathcal{H}^n$-a.e.,
let $f \in \mathscr{F}'$, and suppose that \eqref{finite integral} is satisfied.
Then, we set
\[
\mathcal{E} (\mu, \Omega) := \{ u \in W^{1, 1}_{0, \tau} (\Omega): \text{$u$ is $\mu$-distributed and equality holds in } \eqref{gen polya szego full set}
 \}.
\]
\end{definition}

\begin{remark}
Note that, under the assumption \eqref{finite integral}, using \eqref{ineq A large setminus} we have that 
\eqref{gen polya szego full set} is equivalent to \eqref{equality} and to equality in \eqref{ineq omega intro}.
\end{remark}

We start with some necessary conditions for a function $u$ to belong 
to $\mathcal{E} (\mu,\Omega)$.
\begin{proposition} \label{first necessary conds for extremals}
Let $n \in \mathbb{N}$ with $n \geq 2$, let $\Omega \subset \R^n$ be open, and let $\mu \in \mathcal{A} (\Pi_{n-1} (\Omega) \times \R)$.
Let $a \in L^{\infty} ((0, \infty) \times \R^{n-2} \times \R)$ with $a > 0$ $\mathcal{H}^n$-a.e.,
let $f \in \mathscr{F}'$, and suppose that \eqref{finite integral} is satisfied.
Let $u \in \mathcal{E} (\mu, \Omega)$.
Then, for $\mathcal{H}^n$-a.e. $(r, y, t) \in \{ 0< \alpha_\mu < \pi \}$
\begin{equation} \label{connected sections}
(\Sigma^{u_0})_{(r, y, t)} =_{\mathcal{H}^1} \mathbf{B}_{\alpha_\mu (r, y, t)} (p  (r, y, t)) \quad \text{ for some } p  (r, y, t) \in \mathbb{S}^1,
\end{equation}
and the functions
\[
x \longmapsto \hat{x} \cdot \nabla u_0 (x, y),
\qquad
x \longmapsto | \Dx u_0 (x, y)|,  
\qquad
x \longmapsto  \nabla_y u_0 (x, y),
\]
are constant on $\partial^* ( (\Sigma^{u_0})_{(r, y, t)} )$.
\end{proposition}

\begin{proof}
Since $u \in \mathcal{E} (\mu, \Omega)$, equality in \eqref{gen polya szego full set} holds. Then, thanks to Theorem~\ref{the thing}
equality holds in \eqref{gen polya szego} for every Borel set $B \subset \Pi_{n-1} (\Omega) \times \R$.
In particular, equality holds in \eqref{first inequality polya} for $\mathcal{H}^n$-a.e. $(r, y, t) \in \{ 0< \alpha_\mu < \pi \}$.
Since $a > 0$ $\mathcal{H}^n$-a.e. and since, thanks to Lemma~\ref{lemma equivalent increasing}, 
$\tau \longmapsto f (\eta, \tau, \zeta)$ is strictly increasing in $[0, \infty)$, 
we conclude that equality holds in \eqref{isop ineq circle} for $\mathcal{H}^n$-a.e. $(r, y, t) \in \{ 0 < \alpha_\mu < \pi \}$. 
Therefore, thanks to \eqref{isop eq} we obtain \eqref{connected sections}.

Observe now that for every Borel set $B \subset \Pi_{n-1} (\Omega) \times \R$ we also have equality in Jensen's inequality \eqref{jensen ineq}.
Since $a > 0$ $\mathcal{H}^n$-a.e. and $f$ is strictly convex, the remaining part of the statement follows.
\end{proof}
If $u \in W^{1, 1}_{0, \tau} (\Omega)$ satisfies \eqref{connected sections}, 
we can define the direction function 
$d_{u_0} : \Pi_{n-1} (\Omega) \times \R \to \mathbb{S}^1$ given by
\begin{equation} \label{def direction}
d_{u_0} (r, y, t) := 
\begin{cases}
\displaystyle \frac{1}{2 r \sin \alpha_{\mu}(r, y, t)} \int_{(\Sigma^{u_0})_{(r, y, t)}} \hspace{-.8cm}
\hat{x} \, d \mathcal{H}^1 (x) 
& \text{ if } (r, y, t)  \in \{ 0 < \alpha_\mu < \pi \}, \\
e_1 & \text{ otherwise in } \Pi_{n-1} (\Omega) \times \R.
\end{cases}
\end{equation}
\begin{remark}
If $u \in W^{1, 1}_{0, \tau} (\Omega)$ satisfies \eqref{connected sections} and $d_{u_0}$ is defined by \eqref{def direction}, then a direct calculation shows that
\[
\Sigma^{u_0} \cap \{ 0 < \alpha_\mu < \pi \} =_{\mathcal{H}^{n+1}} \{ (x, y, t) \in \mathbb{R}^2_0 \times \mathbb{R}^{n-2} \times \mathbb{R} : \, 
d_{\mathbb{S}^1} (\hat{x} , d_{u_0} ( | x |, y, t) )  < \alpha_\mu (|x|, y, t)  \}.
\]
\end{remark}
We now give a regularity result for $d_{u}$.
\begin{theorem} \label{regularity d_u}
Let $n \in \mathbb{N}$ with $n \geq 2$, let $\Omega \subset \R^n$ be open, and let $\mu \in \mathcal{A} (\Pi_{n-1} (\Omega) \times \R)$.
Let $u \in W^{1, 1}_{0, \tau} (\Omega)$ be a $\mu$-distributed function satisfying \eqref{connected sections},
let $A \subset \subset \Pi_{n-1} (\Omega)$ be open, let $d > 0$, and let $\delta \in (0, \pi/2)$.
 Then, $d_{u_0}^{\delta} \in BV (A \times (-d, + \infty); \mathbb{R}^2)$, where
\[
d_{u_0}^{\delta} (r, y, t):= \frac{1}{2 r \sin \alpha^{\delta}_{\mu}(r, y, t)} \int_{(\Sigma^{u_0})_{(r, y, t)}} \hat{x} \, d \mathcal{H}^1 (x), \qquad
 \alpha^{\delta}_{\mu} := (\alpha_{\mu} \vee \delta) \wedge (\pi - \delta).
\]
Moreover, the following coarea formula holds. Using the notation $d_{u_0}^{\delta} = ((d_{u_0}^{\delta})_1, (d_{u_0}^{\delta})_2)$, for every $j \in \{ 1, 2 \}$ we have
\begin{equation} \label{coarea d_u}
\begin{split}
&\int_{\mathbb{R}} \mathcal{H}^{n-1} (B \cap \partial^{{\rm e}}
 \{ (d_{u_0}^{\delta})_j > s \})    \, ds  \\
 &= \int_{B} \, |\nabla (d_{u_0}^{\delta})_j | dr \, dy \, dt 
+ \int_{B  \cap S_{(d_{u_0}^{\delta})_j}} \, [ (d_{u_0}^{\delta})_j] \, d \mathcal{H}^{n-1} + | D^c (d_{u_0}^{\delta})_j| (B),
\end{split}
\end{equation}
for every Borel set $B \subset A \times (-d, + \infty)$. 
\end{theorem}

\begin{proof}
Let $c > 0$ be such that $r \geq c$ for every $(r, y, t) \in A \times (-d, + \infty)$.
To ease the notation, we define the function $m_{u_0}: \Pi_{n-1} (\Omega) \times \mathbb{R} \to  \mathbb{R}^2$ as
\begin{equation} \label{def m}
m_{u_0} (r, y, t) := \frac{1}{r} \int_{(\Sigma^{u_0})_{(r, y, t)}} \hat{x} \, d \mathcal{H}^1 (x), 
\quad \text{ for every } (r, y, t) \in \Pi_{n-1} (\Omega) \times \mathbb{R}.
\end{equation}

\vspace{.2cm}

\noindent
\textbf{Step 1:} We show that $m_{u} \in BV (A \times (-d, + \infty); \mathbb{R}^2)$.
We have 
\begin{align*}
\int_{A \times (-d, + \infty)} |m_{u_0} (r, y, t)| \, dr \, dy \, dt 
&\leq \frac{1}{c}
\int_{A \times (-d, + \infty)} \mathcal{H}^1((\Sigma^{u_0})_{(r, y, t)}) \, dr \, dy \, dt \\
&= \frac{1}{c} \| \mu \|_{L^1 (A \times (-d, + \infty ))}  < \infty,
\end{align*}
where the last inequality follows from Proposition~\ref{prop_circular rear and circular sym are connected}.
This shows that $m_{u_0} \in L^1 (A \times (-d, + \infty); \mathbb{R}^2)$.

Let now $j \in \{ 1, 2 \}$ and, denoting by $(m_{u_0})_j$ the $j$-th component of $m_{u_0}$, let us
show that $D (m_{u_0})_j$ is a bounded Radon measure on $A \times (-d, + \infty)$.
Let $\psi \in C^{1}_c (A \times (-d, + \infty))$.
Then, 
\begin{align*}
&\int_{A \times (-d, + \infty)} \frac{\partial \psi}{\partial r} (r, y, t) (m_{u_0})_j (r, y, t) \, dr \, dy \, dt \\
&= \int_{A \times (-d, + \infty)} \frac{\partial \psi}{\partial r} (r, y, t) 
\left(  \frac{1}{r} \int_{(\Sigma^{u_0})_{(r, y, t)}} \hat{x}_j \, d \mathcal{H}^1 (x) \right)  \, dr \, dy \, dt \\
&= \int_{A \times (-d, + \infty)} \frac{\partial \psi}{\partial r} (r, y, t) 
\left(  \frac{1}{r} \int_{\partial D (r)} \frac{x_j}{r} \chi_{\Sigma^{u_0}} (x, y, t) \, d \mathcal{H}^1 (x) \right)  \, dr \, dy \, dt \\
&= \int_{  \Phi_{n+1} ((A \times (-d, + \infty)) \times \mathbb{S}^{1})  } \frac{x_j}{|x|^2} \frac{\partial \psi}{\partial r} 
(|x|, y, t) 
  \chi_{\Sigma^{u_0}} (x, y, t)   \, dx \, dy \, dt.
\end{align*}
A direct calculation shows that
\[
\textnormal{div}_x \left( \frac{x_j}{|x|^2} \psi (|x|, y, t) \hat{x} \right)
= \frac{x_j}{|x|^2} \frac{\partial \psi}{\partial r} 
(|x|, y, t),
\]
where $\textnormal{div}_x$ denotes the divergence with respect to the variables $(x_1, x_2)$.
Therefore, 
\begin{equation} \label{D_rm_uj}
\begin{split}
&\int_{A \times (-d, + \infty)} \frac{\partial \psi}{\partial r} (r, y, t) (m_{u_0})_j (r, y, t) \, dr \, dy \, dt \\
&= \int_{  \Phi_{n+1} ((A \times (-d, + \infty)) \times \mathbb{S}^{1}) )  } \textnormal{div}_x \left( \frac{x_j}{|x|^2} \psi (|x|, y, t) \hat{x} \right)  \chi_{\Sigma^{u_0}} (x, y, t)   \, dx \, dy \, dt \\
&= - \int_{ \Phi_{n+1} ((A \times (-d, + \infty)) \times \mathbb{S}^{1})  } 
\frac{x_j}{|x|^2} \psi (|x|, y, t) \, \hat{x} \cdot d D_x \chi_{\Sigma^{u_0}} (x, y, t) \\
&= - \int_{ \partial^* \Sigma^{u_0} \cap \Phi_{n+1} ((A \times (-d, + \infty)) \times \mathbb{S}^{1})  } 
\psi (|x|, y, t) \frac{x_j}{|x|^2} \frac{\hat{x} \cdot \nabla_x u_0 (x, y)}{\sqrt{1 + |\nabla u_0 (x, y)|^2}} d \mathcal{H}^n (x, y, t). 
\end{split}
\end{equation}
Taking the supremum over all $\psi \in C^{1}_c (A \times (-d, + \infty))$ with $| \psi| \leq 1$ we obtain that
$D_r (m_{u_0})_j$ is a bounded Radon measure over $A \times (-d, + \infty)$ with 
\begin{align*}
|D_r (m_{u_0})_j | (A \times (-d, + \infty)) 
\leq \frac{1}{c} \mathcal{H}^n ( \partial^* \Sigma^{u_0} \cap \Phi_{n+1} ((A \times (-d, + \infty)) \times \mathbb{S}^{1})  )  < \infty.
\end{align*}
From \eqref{D_rm_uj}, by approximation we obtain that 
\begin{equation} \label{deriv r muj}
D_r (m_{u_0})_j (B) 
= \int_{ \partial^* \Sigma^{u_0} \cap \Phi_{n+1} ( B \times \mathbb{S}^1 )  } 
 \frac{x_j}{|x|^2} \frac{\hat{x} \cdot \nabla_x u_0 (x, y)}{\sqrt{1 + |\nabla u_0 (x, y)|^2}} d \mathcal{H}^n (x, y, t),
\end{equation}
for every Borel set $B \subset A \times (-d, + \infty)$.
In a similar way, one can show that $D_y (m_{u_0})_j$ 
and $D_t (m_{u_0})_j$ are bounded Radon measures
on $A \times (-d, + \infty)$, with
\begin{equation} \label{deriv muj}
\begin{split}
D_y (m_{u_0})_j (B) 
&= \int_{ \partial^* \Sigma^{u_0} \cap \Phi_{n+1} ( B \times \mathbb{S}^1 )  } 
 \frac{x_j}{|x|^2} \frac{\nabla_y u_0 (x, y)}{\sqrt{1 + |\nabla u_0 (x, y)|^2}} d \mathcal{H}^n (x, y, t), \\
D_t (m_{u_0})_j (B) 
&= - \int_{ \partial^* \Sigma^{u_0} \cap \Phi_{n+1} ( B \times \mathbb{S}^1 )  } 
 \frac{x_j}{|x|^2} \frac{1}{\sqrt{1 + |\nabla u_0 (x, y)|^2}} d \mathcal{H}^n (x, y, t),
\end{split}
\end{equation}
for every Borel set $B \subset A \times (-d, + \infty)$.

\vspace{.2cm}

\noindent
\textbf{Step 2:} We conclude. 
First of all, observe that the function $f_\delta : [0, \pi] \to \mathbb{R}$ defined as
\begin{equation} \label{lip function}
f_\delta (s) := \frac{1}{\sin \Big( (s \vee \delta) \wedge (\pi - \delta) \Big)}, 
\end{equation}
is Lipschitz continuous.
Therefore, by the chain rule in $BV$ \cite[Theorem~3.69]{AFP} we have that 
\[
f_\delta (\alpha_\mu) := \frac{1}{ 2\sin \Big( (\alpha_{\mu} \vee \delta) \wedge (\pi - \delta) \Big)} \in BV (A \times (-d, + \infty)).
\]
We can now write $d_{u_0}^{\delta}$ as
\[
d_{u_0}^{\delta} (r, y, t) = f_\delta (\alpha_\mu (r, y, t)) m_{u_0} (r, y, t).
\]
Thanks to Step 1, using again the chain rule in $BV$ (see, in particular, \cite[Example~3.97]{AFP}) we 
have that $d_{u_0}^{\delta} \in BV (A \times (-d, + \infty); \mathbb{R}^2)$.
Finally, let $j \in \{1, 2\}$. Thanks to Step~2 and \cite[Theorem~3.40]{AFP}, formula \eqref{coarea d_u} follows.
\end{proof}
We can now give further necessary conditions for a function $u$ to belong to the set $\mathcal{E} (\mu, \Omega)$ of extremals.
\begin{proposition} \label{nabla d zero}
Let $n \in \mathbb{N}$ with $n \geq 2$, let $\Omega \subset \R^n$ be open, and let $\mu \in \mathcal{A} (\Pi_{n-1} (\Omega) \times \R)$ be such that \eqref{set non empty} holds.
Let $a \in L^{\infty} ((0, \infty) \times \R^{n-2} \times \R)$ with $a > 0$ $\mathcal{H}^n$-a.e.,
and let $f \in \mathscr{F}'$.
Suppose that \eqref{finite integral} is satisfied, and let $u \in \mathcal{E} (\mu, \Omega)$.
Then, the function $d_{u_0}$ defined in \eqref{def direction} is $\mathcal{H}^n$-a.e. 
approximately differentiable in $\{ 0 < \alpha_\mu < \pi \}$ and 
\begin{equation} \label{nabla 0}
\nabla d_{u_0} (r, y, t) = 0  \quad \text{ for $\mathcal{H}^n$-a.e. $(r, y, t)$ in $\{ 0 < \alpha_\mu < \pi \}$.}
\end{equation}
Moreover, denoting the $j$-th component of $d_{u_0}$ by $(d_{u_0})_j$, for every $j =1, 2$ we have
\[
S_{(d_{u_0})_j} \cap \{ 0 < \alpha^{\wedge}_\mu \leq \alpha^{\vee}_\mu < \pi \} 
\subset_{\mathcal{H}^{n-1}} S_{\alpha_\mu} \cap \{ 0 < \alpha^{\wedge}_\mu \leq \alpha^{\vee}_\mu < \pi \}
\quad \text{ and } \quad | D^c (d_{u_0})_j|^+ \ll | D^c \alpha_\mu|,
\]
where $| D^c (d_{u_0})_j |^+$ is the Borel measure given by
\[
| D^c (d_{u_0})_j |^+ (B) := \lim_{\delta \to 0^+} | D^c (d_{u_0}^{\delta})_j| (B), 
\quad \text{ for every Borel set } B \subset  \Pi_{n-1} (\Omega) \times \mathbb{R}.
\]
Finally, if $W$
is a Borel subset of $\{ 0 < \alpha^{\wedge}_\mu \leq \alpha^{\vee}_\mu < \pi \}$
and $K$ is a concentration set for $D^c \alpha_\mu$, then for every $j =1, 2$
\begin{align*}
\int_{-1}^{1} \mathcal{H}^{n-1} (W \cap \partial^{{\rm e}}
 \{ (d_{u_0})_j > s \})    \, ds 
 =  \int_{W  \cap S_{(d_{u_0})_j} \cap S_{\alpha_\mu}} \, [ (d_{u_0})_j] \, d \mathcal{H}^{n-1} 
 + | D^c (d_{u_0})_j|^+ (W \cap K).
\end{align*}

\end{proposition}

\begin{proof}
Thanks to Proposition~\ref{volpert theorem} and Proposition~\ref{first necessary conds for extremals}, 
for $\mathcal{H}^n$-a.e. $(r, y, t) \in \{ 0< \alpha_\mu < \pi\}$, 
the set $(\partial^* \Sigma^{u_0})_{(r, y, t)}$ is composed of just two points, and 
the vector function
\[
x \longmapsto  \frac{\left( \hat{x} \cdot \nabla_x u_0 (x, y), \nabla_y u_0 (x, y), - 1 \right)}{| \Dx u_0 (x, y)|}
\]
is constant on $(\partial^* \Sigma^{u_0})_{(r, y, t)}$.
Therefore, for $\mathcal{H}^n$-a.e. $(r, y, t) \in \{ 0< \alpha_\mu < \pi\}$
we can define the vector $c_{u_0} (r, y, t) \in \R^n$ given by
\begin{equation} \label{cool}
c_{u_0} (r, y, t)  : = \frac{\left( \hat{x} \cdot \nabla_x u_0 (x, y), \nabla_y u_0 (x, y), - 1 \right)}{| \Dx u_0 (x, y)|}
= r \nabla \alpha_\mu (r, y, t),
\end{equation}
where $x \in (\partial^* \Sigma^{u_0})_{(r, y, t)}$, and  the last equality follows from Proposition~\ref{abs cont parts}.
We now divide the proof into several steps.

\vspace{.2cm}

\noindent
\textbf{Step 1}: We show that the density of the absolutely continuous part of $D m_{u_0}$ is given by
\[
\nabla m_{u_0} (r, y, t) = \frac{1}{r^2} \int_{( \partial^* \Sigma^{u_0})_{(r, y, t)}} 
   \frac{x \otimes \left( \hat{x} \cdot \nabla_x u_0 (x, y), \nabla_y u_0 (x, y), - 1 \right)}
 { |\Dx u_0 (x, y) |} \, d \mathcal{H}^0 (x),
\]
for $\mathcal{H}^n$-a.e. $(r, y, t)$ in $\{ 0 < \alpha_{\mu} < \pi \}$, where $m_{u_0}$ is defined by \eqref{def m}.

Let $A \subset \subset \Pi_{n-1} (\Omega)$ be open, let $d > 0$, let $G_{u_0}$ be given by Proposition~\ref{volpert theorem}.
 Thanks to \eqref{deriv r muj} and \eqref{deriv muj}, we have that
\begin{equation} \label{derivative m}
D m_{u_0} (B) = \int_{ \partial^* \Sigma^{u_0} \cap \Phi_{n+1} ( B \times \mathbb{S}^1 )  } 
 \frac{x \otimes \left( \hat{x} \cdot \nabla_x u_0 (x, y), \nabla_y u_0 (x, y), - 1 \right)}{|x|^2 \sqrt{1 + |\nabla u_0 (x, y)|^2}} d \mathcal{H}^n (x, y, t),
\end{equation}
for every Borel set $B \subset A \times (-d, + \infty)$.
Since $\mathcal{H}^n ( (\{ 0 < \alpha_\mu   < \pi \} \cap (A \times (-d, + \infty))) \setminus G_{u_0}) = 0$, we have
\begin{align*}
D^a m_{u_0} (B) &= D^a m_{u_0} (B \cap G_{u_0}) \\
&= \int_{ \partial^* \Sigma^{u_0} \cap \Phi_{n+1} ( (B \cap G_{u_0}) \times \mathbb{S}^1 )} 
 \frac{x \otimes \left( \hat{x} \cdot \nabla_x u_0 (x, y), \nabla_y u_0 (x, y), - 1 \right)}
 {|x|^2 \sqrt{1 + |\nabla u_0 (x, y)|^2}} d \mathcal{H}^n (x, y, t) \\
 &=\int_{B \cap G_{u_0}} \frac{1}{r^2} \left( \int_{( \partial^* \Sigma^{u_0})_{(r, y, t)}} 
   \frac{x \otimes \left( \hat{x} \cdot \nabla_x u_0 (x, y), \nabla_y u_0 (x, y), - 1 \right)}
 { |\Dx u_0 (x, y) |} \, d \mathcal{H}^0 (x) \right) \, dr \, dy \, dt \\
 &=\int_{B} \frac{1}{r^2} \left( \int_{( \partial^* \Sigma^{u_0})_{(r, y, t)}} 
   \frac{x \otimes \left( \hat{x} \cdot \nabla_x u_0 (x, y), \nabla_y u_0 (x, y), - 1 \right)}
 { |\Dx u_0 (x, y) |} \, d \mathcal{H}^0 (x) \right) \, dr \, dy \, dt.
\end{align*}
Since $A$ and $d$ are arbitrary, the conclusion follows.

\vspace{.2cm}

\noindent
\textbf{Step 2}: We show that for $\mathcal{H}^n$-a.e. $(r, y, t)$ in $\{ 0 < \alpha_{\mu} < \pi \}$
\begin{equation} \label{abs cont der mu}
\nabla m_{u_0} (r, y, t) = \frac{\cos \alpha_\mu (r, y, t)}{r \sin \alpha_\mu (r, y, t)} \Big( m_{u_0} (r, y, t)  \otimes c_{u_0} (r, y, t) \Big).
\end{equation}
where $m_{u_0}$ is defined by \eqref{def m} and $c_{u_0}$ is given by \eqref{cool}.
Indeed, since $( \partial^* \Sigma^{u_0})_{(r, y, t)}$ is composed 
by two points that are symmetric with respect to $d_{u_0} (r, y, t)$, using Step 1 we obtain that   
\begin{align*}
&\nabla m_{u_0} (r, y, t) 
= \frac{1}{r^2} \int_{( \partial^* \Sigma^{u_0})_{(r, y, t)}} 
   \frac{x \otimes \left( \hat{x} \cdot \nabla_x u_0 (x, y), \nabla_y u_0 (x, y), - 1 \right)}
 { |\Dx u_0 (x, y) |} \, d \mathcal{H}^0 (x) \\
 &= \frac{1}{r^2} \int_{( \partial^* \Sigma^{u_0})_{(r, y, t)}}
   ( x \otimes c_{u_0} (r, y, t) ) \, d \mathcal{H}^0 (x)  
= \frac{1}{r^2} \left(  \int_{( \partial^* \Sigma^{u_0})_{(r, y, t)}} x \, d \mathcal{H}^0 (x) \right) \otimes c (r, y, t) \\
&= \frac{1}{r^2} \left( 2 r \cos \alpha_\mu (r, y, t) \right) \Big( d_{u_0} (r, y, t)  \otimes c_{u_0} (r, y, t) \Big) 
= \frac{2 \cos \alpha_\mu (r, y, t)}{r} \Big( d_{u_0} (r, y, t)  \otimes c_{u_0} (r, y, t) \Big) \\
&= \frac{\cos \alpha_\mu (r, y, t)}{r \sin \alpha_\mu (r, y, t)} \Big( m_{u_0} (r, y, t)  \otimes c_{u_0} (r, y, t) \Big).
\end{align*}

\vspace{.2cm}

\noindent
\textbf{Step 3}: We show \eqref{nabla 0}. Let $\delta \in (0, \pi/2)$.
Note that $d_{u_0}^{\delta} = d_{u_0}$ $\mathcal{H}^n$-a.e. in
$\{ \delta < \alpha_\mu < \pi - \delta \}^{(1)}$. Then, thanks to Theorem~\ref{regularity d_u}
we can write 
\[
\nabla d_{u_0} = \nabla d_{u_0}^{\delta} = \frac{1}{2}\nabla \left( \frac{m_{u_0}}{\sin \alpha^{\delta}_{\mu}}  \right)
= \frac{1}{2}\nabla \left( \frac{m_{u_0}}{\sin \alpha_{\mu}}  \right), \quad \mathcal{H}^n\text{-a.e. in } \{ \delta < \alpha_\mu < \pi - \delta \}^{(1)}.
\]
Thanks to the chain rule in $BV$ (see \cite[Example~3.97]{AFP})
and using Step 2, we have 
that $\mathcal{H}^n$-a.e. in $\{ \delta < \alpha_\mu < \pi - \delta \}^{(1)}$
\begin{align*}
\nabla d_{u_0} 
&= \frac{1}{2} \frac{ \nabla m_{u_0}}{\sin \alpha_{\mu}} - \frac{1}{2} \frac{\cos \alpha_{\mu}}{(\sin \alpha_{\mu})^2}
( m_{u_0} \otimes \nabla \alpha_\mu ) \\
&= \frac{\cos \alpha_\mu}{2 r \sin^2 \alpha_\mu } ( m_{u_0}  \otimes c  )
- \frac{1}{2} \frac{\cos \alpha_{\mu}}{(\sin \alpha_{\mu})^2}
( m_{u_0} \otimes \nabla \alpha_\mu ) = 0,
\end{align*}
where we used \eqref{cool} and \eqref{abs cont der mu}.
This shows that
\[
\nabla d_{u_0} = 0, \qquad \text{$\mathcal{H}^n$-a.e. in $\{ \delta < \alpha_\mu < \pi - \delta \}^{(1)}$.}
\]
Note now that the set $\{ \delta < \alpha_\mu < \pi - \delta \}^{(1)}$ 
is decreasing in $\delta$ and that, thanks to \cite[(2.3) and (2.4)]{ccdpmGAUSS}, we have 
\begin{equation} \label{union over delta}
\bigcup_{\delta \in (0, \pi/2)}  \{ \delta < \alpha_\mu < \pi - \delta \}^{(1)} = \{ 0 < \alpha^{\wedge}_\mu \leq \alpha^{\vee}_\mu < \pi \}.
\end{equation}
Therefore, we obtain that 
\[
\nabla d_{u_0} = 0, \qquad \text{$\mathcal{H}^n$-a.e. in $\{ 0 < \alpha^{\wedge}_\mu \leq \alpha^{\vee}_\mu < \pi \}$.}
\]
Since $\{ 0 < \alpha^{\wedge}_\mu \leq \alpha^{\vee}_\mu < \pi \} =_{\mathcal{H}^n} \{ 0 < \alpha_{\mu} < \pi\}$,  the conclusion follows.

\vspace{.2cm}

\noindent
\textbf{Step 4}: We show that 
\begin{equation} \label{abs cont d m u j}
| D (m_{u_0})_j| \mres \{ 0 < \alpha^{\wedge}_\mu \leq \alpha^{\vee}_\mu < \pi \} 
\ll | D \mu | \mres \{ 0 < \alpha^{\wedge}_\mu \leq \alpha^{\vee}_\mu < \pi \}.
\end{equation}
Let $A \subset \subset \Pi_{n-1} (\Omega)$ be open, and let $d > 0$.

Let now $W \subset \{ 0 < \alpha^{\wedge}_\mu \leq \alpha^{\vee}_\mu < \pi \}$
with $W \subset A \times (-d, +\infty)$ be such that $| D \mu | (W) = 0$.
Then, thanks to \eqref{Dt mu}, 
\begin{align*}
0 = D_t \mu  (W) &= - \int_{ \partial^* \Sigma^{u_0} \cap \Phi_{n+1} ( W \times \mathbb{S}^1 )}  
 \, \frac{1}{\sqrt{1 + |\nabla u_0 (x, y)|^2}} \, d \mathcal{H}^n (x, y, t),
\end{align*} 
which implies
\begin{align} \label{we like it a lot}
\mathcal{H}^n \left( \partial^* \Sigma^{u_0} \cap \Phi_{n+1} ( W \times \mathbb{S}^1 ) \right) = 0.
\end{align}
On the other hand, from \eqref{derivative m}
\begin{align*}
| D_t (m_{u_0})_j  (W) | 
&= 
\left|
 \int_{ \partial^* \Sigma^{u_0} \cap \Phi_{n+1} ( W \times \mathbb{S}^1 )  } 
 \frac{x_j}{|x|^2} \frac{1}{\sqrt{1 + |\nabla u_0 (x, y)|^2}} d \mathcal{H}^n (x, y, t)
\right| \\
&\leq \frac{1}{c} \mathcal{H}^n 
\left( 
\partial^* \Sigma^{u_0} \cap \Phi_{n+1} ( W \times \mathbb{S}^1 )\right) = 0,
\end{align*}
where $c > 0$ is such that $r \geq c$ for every $(r, y, t) \in W$, and where we used \eqref{we like it a lot}.
Arguing in a similar way, one can show that 
\[
D_y (m_{u_0})_j (W) = 0, \quad \text{ and } \quad D_r (m_{u_0})_j (W) = 0.
\]
Since $A$ and $d$ are arbitrary, the conclusion follows.

\vspace{.2cm}

\noindent
\textbf{Step 5}: We show that for every $j = 1, 2$
\begin{equation} \label{fzkfsl}
S_{(d_{u_0})_j} \cap \{ 0 < \alpha^{\wedge}_\mu \leq \alpha^{\vee}_\mu < \pi \} 
\subset_{\mathcal{H}^{n-1}} S_{\alpha_\mu} \cap \{ 0 < \alpha^{\wedge}_\mu \leq \alpha^{\vee}_\mu < \pi \}.
\end{equation}
Indeed, let $j \in \{ 1, 2\}$, and let $\delta \in (0, \pi/2)$. Recall that $d_{u_0}^{\delta}$ is defined as 
\[
d_{u_0}^{\delta} (r, y, t) = f_\delta (\alpha_\mu (r, y, t)) m_{u_0} (r, y, t),
\]
where $f_\delta$ is the Lipschitz function given by \eqref{lip function}.
Then, recalling the definition of $d_{u_0}$, we have 
\begin{align} \label{chain 1}
S_{(d_{u_0})_j} \cap \{ \delta < \alpha_\mu < \pi - \delta \}^{(1)} 
= S_{(d_{u_0}^{\delta})_j} \cap \{ \delta < \alpha_\mu < \pi - \delta \}^{(1)}.  
\end{align}
On the other hand, thanks to the chain rule in $BV$
\begin{align} \label{chain 2}
S_{(d_{u_0}^{\delta})_j} \cap \{ \delta < \alpha_\mu < \pi - \delta \}^{(1)}  
\subset_{\mathcal{H}^{n-1}} (S_{\alpha_\mu} \cup S_{(m_{u_0})_j} ) \cap \{ \delta < \alpha_\mu < \pi - \delta \}^{(1)}. 
\end{align}
Observe now that, thanks to \eqref{abs cont d m u j}, and taking into account that $S_{\mu} = S_{\alpha_\mu}$, we have
\begin{align} \label{chain 3}
S_{(m_{u_0})_j} \cap \{ \delta < \alpha_\mu < \pi - \delta \}^{(1)} 
\subset_{\mathcal{H}^{n-1}} S_{\alpha_\mu} \cap \{ \delta < \alpha_\mu < \pi - \delta \}^{(1)}.
\end{align}
Combining \eqref{chain 1}, \eqref{chain 2}, and \eqref{chain 3}, 
and recalling \eqref{union over delta}, we conclude. 

\vspace{.2cm}

\noindent
\textbf{Step 6}: We show that for every $j = 1, 2$
\begin{equation} \label{cantor abs}
| D^c (d_{u_0})_j|^+ \mres \{ 0 < \alpha^{\wedge}_\mu \leq \alpha^{\vee}_\mu < \pi \} 
 \ll | D^c \alpha_\mu| \mres \{ 0 < \alpha^{\wedge}_\mu \leq \alpha^{\vee}_\mu < \pi \}.
\end{equation}
Let $\delta \in (0, \pi/2)$. Since $d_{u_0}^{\delta} = d_{u_0}$ $\mathcal{H}^n$-a.e. in $\{ \delta < \alpha_\mu < \pi - \delta \}^{(1)}$, thanks 
to \cite[Lemma~2.2]{CagnettiColomboDePhilippisMaggiSteiner}, we have that
\[
D^c (d_{u_0})_j \mres \{ \delta < \alpha_\mu < \pi - \delta \}^{(1)} 
= D^c (d_{u_0}^{\delta})_j \mres \{ \delta < \alpha_\mu < \pi - \delta \}^{(1)}.
\]
On the other hand, by definition of $d_{u_0}^{\delta}$ and using again the chain rule in BV (\cite[Theorem~3.96 and Example~3.97]{AFP})
\begin{equation} \label{fd1} 
\begin{split}
D^c \left((d_{u_0}^{\delta})_j \right)   &= D^c ( f_\delta (\alpha_\mu ) (m_{u_0})_j )
=   D^c ( f_\delta (\alpha_\mu ) ) (m_{u_0})_j 
+   f_\delta (\alpha_\mu ) D^c \left( (m_{u_0})_j  \right)    \\
&=  (m_{u_0})_j  (\nabla f_\delta (\alpha_\mu ))  D^c \alpha_\mu
+   f_\delta (\alpha_\mu ) D^c \left( (m_{u_0})_j  \right). 
\end{split}
\end{equation}
From \eqref{abs cont d m u j} we know that 
\begin{equation} \label{fd2} 
\begin{split}
| D^c \left( (m_{u_0})_j  \right) | \mres \{ \delta < \alpha_\mu < \pi - \delta \}^{(1)} 
&\ll | D^c \mu | \mres \{ \delta < \alpha_\mu < \pi - \delta \}^{(1)} \\
&\ll | D^c \alpha_\mu | \mres \{ \delta < \alpha_\mu < \pi - \delta \}^{(1)}.
\end{split}
\end{equation}
Combining \eqref{fd1} and \eqref{fd2} we then obtain
\[
| D^c \left((d_{u_0}^{\delta})_j \right) | \mres \{ \delta < \alpha_\mu < \pi - \delta \}^{(1)} 
\ll | D^c \alpha_\mu | \mres \{ \delta < \alpha_\mu < \pi - \delta \}^{(1)}.
\]
Thanks to \eqref{union over delta}, this implies that 
\[
| D^c \left((d_{u_0}^{\delta})_j \right) | \mres \{ 0 < \alpha^{\wedge}_\mu \leq \alpha^{\vee}_\mu < \pi \}  
\ll | D^c \alpha_\mu | \mres \{ 0 < \alpha^{\wedge}_\mu \leq \alpha^{\vee}_\mu < \pi \}.
\]
Let now $W \subset \{ 0 < \alpha^{\wedge}_\mu \leq \alpha^{\vee}_\mu < \pi \}$ be a Borel set with 
$| D^c \alpha_\mu | (W) = 0$. From what we have just proved it follows that
\[
| D^c \left((d_{u_0}^{\delta})_j \right) | (W) = 0, \qquad \text{ for every } \delta \in (0, \pi/2).
\] 
Then, 
\[
| D^c (d_{u_0})_j|^+ (W) = \lim_{\delta \to 0^+} | D^c \left((d_{u_0}^{\delta})_j \right) | (W) = 0,
\]
which shows the claim.

\vspace{.2cm}

\noindent
\textbf{Step 7}: We conclude.
Let $A \subset \subset \Pi_{n-1} (\Omega)$ be open, let $d > 0$, let $\delta \in (0, \pi/2)$, and let $j \in \{ 1, 2 \}$.
Since $(d^{\delta}_{u_0})_j = (d_{u_0})_j$ $\mathcal{H}^n$-a.e. on $\{ \delta < \alpha_\mu < \pi - \delta \}^{(1)}$, 
thanks to \eqref{coarea d_u} we have that for every Borel set $W  \subset (A \times (-d, +\infty)) \cap \{ 0 < \alpha^{\wedge}_\mu \leq \alpha^{\vee}_\mu < \pi \}$
\begin{align*}
&\int_{\mathbb{R}} \mathcal{H}^{n-1} (W \cap \{ \delta < \alpha_\mu < \pi - \delta \}^{(1)} \cap \partial^{{\rm e}}
 \{ (d_{u_0})_j > s \})    \, ds \\
& = \int_{W \cap \{ \delta < \alpha_\mu < \pi - \delta \}^{(1)}} \, |\nabla (d_{u_0})_j | dr \, dy \, dt 
+ \int_{W \cap \{ \delta < \alpha_\mu < \pi - \delta \}^{(1)} \cap S_{(d_{u_0})_j} \cap S_{\alpha_\mu}} \, [ (d_{u_0})_j] \, d \mathcal{H}^{n-1} \\
&+ | D^c (d_{u_0})_j| (W \cap \{ \delta < \alpha_\mu < \pi - \delta \}^{(1)}),
\end{align*}
where we also used \eqref{fzkfsl}.
Thanks to \eqref{union over delta}, passing to the limit as $\delta \to 0^+$, we obtain
\begin{align*}
&\int_{\mathbb{R}} \mathcal{H}^{n-1} (W \cap \{ 0 < \alpha^{\wedge}_\mu \leq \alpha^{\vee}_\mu < \pi \} \cap \partial^{{\rm e}}
 \{ (d_{u_0})_j > s \})    \, ds \\
& = \int_{W \cap \{ 0 < \alpha^{\wedge}_\mu \leq \alpha^{\vee}_\mu < \pi \}} \, |\nabla (d_{u_0})_j | dr \, dy \, dt 
+ \int_{W \cap \{ 0 < \alpha^{\wedge}_\mu \leq \alpha^{\vee}_\mu < \pi \} \cap S_{(d_{u_0})_j} \cap S_{\alpha_\mu}} \, [ (d_{u_0})_j] \, d \mathcal{H}^{n-1} \\
&+ | D^c (d_{u_0})_j|^+ (W \cap \{ 0 < \alpha^{\wedge}_\mu \leq \alpha^{\vee}_\mu < \pi \}).
\end{align*}
Let now $K$ be a concentration set for $D^c \alpha_\mu$. Thanks to \eqref{cantor abs}, we obtain
\begin{align*}
&\int_{\mathbb{R}} \mathcal{H}^{n-1} (W \cap \{ 0 < \alpha^{\wedge}_\mu \leq \alpha^{\vee}_\mu < \pi \} \cap \partial^{{\rm e}}
 \{ (d_{u_0})_j > s \})    \, ds \\
& = \int_{W \cap \{ 0 < \alpha^{\wedge}_\mu \leq \alpha^{\vee}_\mu < \pi \}} \, |\nabla (d_{u_0})_j | dr \, dy \, dt 
+ \int_{W \cap \{ 0 < \alpha^{\wedge}_\mu \leq \alpha^{\vee}_\mu < \pi \} \cap S_{(d_{u_0})_j} \cap S_{\alpha_\mu}} \, [ (d_{u_0})_j] \, d \mathcal{H}^{n-1} \\
&+ | D^c (d_{u_0})_j|^+ (W \cap K \cap \{ 0 < \alpha^{\wedge}_\mu \leq \alpha^{\vee}_\mu < \pi \}).
\end{align*}

Since $A$ and $d$ are arbitrary, the conclusion follows. 
\end{proof}
We are now ready to give the proof of Theorem~\ref{thm suff conditions}.
\begin{proof}[Proof of Theorem~\ref{thm suff conditions}]
Let $u \in \mathcal{E} (\mu, \Omega)$.
We want to show that the function $d_{u_0}$ defined in \eqref{def direction} is 
$\mathcal{H}^n$-a.e. constant on $\{ 0 < \alpha_{\mu} < \pi \}$.
Let $j \in \{1, 2\}$ be fixed. 
Thanks to Proposition~\ref{nabla d zero}, for every Borel subset $W$
of $\{ 0 < \alpha^{\wedge}_\mu \leq \alpha^{\vee}_\mu < \pi \}$ we have
\begin{equation} \label{hopefully this is it}
\begin{split}
&\int_{-1}^{1} \mathcal{H}^{n-1} (W \cap \partial^{{\rm e}}
 \{ (d_{u_0})_j > s \})    \, ds  \\
& =  \int_{W  \cap S_{(d_{u_0})_j} \cap S_{\alpha_\mu}} \, [ (d_{u_0})_j] \, d \mathcal{H}^{n-1} 
+ | D^c (d_{u_0})_j|^+ (W \cap K). 
\end{split}
\end{equation}
Suppose now, by contradiction, that $(d_{u_0})_j$ is not $\mathcal{H}^n$-a.e. constant on $\{ 0 < \alpha_\mu < \pi \}$.
Then, there exists a Lebesgue measurable set $I \subset (-1, 1)$
such that $\mathcal{H}^1 (I) > 0$ and, for every $s \in I$, the Borel sets
\[
W^s_+ = \{ (d_{u_0})_j > s \} \cap \{ 0 < \alpha_\mu < \pi \}
\quad \text{ and } \quad
W^s_- = \{ (d_{u_0})_j \leq s \} \cap \{ 0 < \alpha_\mu < \pi \}
\]
define a non trivial Borel partition $\{ W^s_+, W^s_- \}$ of $\{ 0 < \alpha_\mu < \pi \}$.

By assumption, $\{ \alpha^\wedge_\mu = 0 \} \cup \{ \alpha^\vee_\mu = \pi \} \cup S_{\alpha_\mu} \cup K$ 
does not essentially disconnect $\{ 0 < \alpha_\mu < \pi \}$. Then, for every $s \in I$ we have 
\begin{align} \label{ess disconn t} 
\mathcal{H}^{n-1} \left(  \{ 0 < \alpha_\mu < \pi \}^{(1)} 
\cap \partial^{{\rm e}} W^s_+ \cap \partial^{{\rm e}} W^s_- \setminus \left( \{ \alpha^\wedge_\mu = 0 \} \cup \{ \alpha^\vee_\mu = \pi \} \cup S_{\alpha_\mu} \cup K  \right) \right) > 0.
\end{align}
Note now that \begin{align*}
&\{ 0 < \alpha_\mu < \pi \}^{(1)} 
\cap \partial^{{\rm e}} W^s_+ \cap \partial^{{\rm e}} W^s_-
= \{ 0 < \alpha_\mu < \pi \}^{(1)} 
\cap \partial^{{\rm e}} W^s_- \\
&= \{ 0 < \alpha_\mu < \pi \}^{(1)} 
\cap \partial^{{\rm e}} W^s_+
= \{ 0 < \alpha_\mu < \pi \}^{(1)} \cap \partial^{{\rm e}} \{ (d_{u_0})_j > s \}.
\end{align*}
Thanks to last equality and \eqref{ess disconn t}
\begin{align} \label{positive}
\mathcal{H}^{n-1} \left(  \{ 0 < \alpha_\mu < \pi \}^{(1)} 
\cap \partial^{{\rm e}} \{ (d_{u_0})_j > s \} \setminus \left( \{ \alpha^\wedge_\mu = 0 \} \cup \{ \alpha^\vee_\mu = \pi \} \cup S_{\alpha_\mu} \cup K  \right) \right) > 0,
\end{align}
for every $s \in I$.
Let us now consider equality \eqref{hopefully this is it} with $W$ given by
\[
W = \{ 0 < \alpha_\mu < \pi \}^{(1)}  \setminus \left( \{ \alpha^\wedge_\mu = 0 \} \cup \{ \alpha^\vee_\mu = \pi \} \cup S_{\alpha_\mu} \cup K  \right)
\]
(note that $W \subset \{ 0 < \alpha^{\wedge}_\mu \leq \alpha^{\vee}_\mu < \pi \}$). We obtain
\begin{align*}
0&\stackrel{\eqref{positive}}{<}\int_{-1}^{1} \mathcal{H}^{n-1} \left(  \{ 0 < \alpha_\mu < \pi \}^{(1)} 
\cap \partial^{{\rm e}} \{ (d_{u_0})_j > s \} \setminus \left( \{ \alpha^\wedge_\mu = 0 \} \cup \{ \alpha^\vee_\mu = \pi \} \cup S_{\alpha_\mu} \cup K  \right) \right) \, ds \\
&= \int_{W  \cap S_{(d_{u_0})_j} \cap S_{\alpha_\mu}} \, [ (d_{u_0})_j] \, d \mathcal{H}^{n-1} 
+ | D^c (d_{u_0})_j|^+ (W \cap K) = 0,
\end{align*}
where we used the fact that with our choice of $W$ we have $W \cap S_{\alpha_\mu} = W \cap K = \emptyset$
and that, thanks to Proposition~\ref{nabla d zero}, $| D^c (d_{u_0})_j|^+ \ll |D^c \alpha_\mu|$.

This gives a contradiction, and thus shows that $(d_{u_0})_j $ is constant.
\end{proof}

\section*{Acknowledgments}
The authors would like to thank Ivor McGillivray for his valuable comments on a preliminary version of the paper.
F.~C. has been partially supported by PRIN 2022 Project ``Geometric Evolution Problems and Shape Optimization (GEPSO)'', PNRR Italia Domani, financed by European Union via the Program NextGenerationEU, CUP D53D23005820006. 
G.~D. was supported by an EPSRC scholarship under the grant EP/R513362/1 ``Free-Discontinuity Problems and
Perimeter Inequalities under symmetrization''.
F.~C. and M.~P. are also members of the Gruppo Nazionale per l'Analisi
Matematica, la Probabilit\`a e le loro Applicazioni (GNAMPA), which is part of the Istituto Nazionale di Alta Matematica (INdAM).

\def\cprime{$'$}

\nocite{*}

\end{document}